\documentclass[10pt]{article}
\usepackage{graphicx,enumerate,amssymb,tikz-cd}
\usepackage[most]{tcolorbox}
\usepackage[notrig]{physics}
\usepackage{amsthm}
\usepackage{dsfont}
\usepackage[font=footnotesize,labelfont=small]{subcaption}
\usepackage[ruled,vlined]{algorithm2e}
\usepackage[left=3.0cm,right=3.0cm,top=2.5cm,bottom=2.5cm,includeheadfoot]{geometry} 
\usepackage{tikz}
\usepackage{mathtools}
\usepackage{todonotes}
\usepackage{braket}
\usepackage{pgfplots}
\usepackage{float}
\usepackage{wrapfig}
\pgfplotsset{width=7cm,compat=1.18} 
\usetikzlibrary{shapes.geometric, arrows}

\usepackage[normalem]{ulem}

\usepackage{multirow}

\usepackage[colorlinks=true, pdfstartview=FitV, linkcolor=blue, 
            citecolor=blue, urlcolor=blue]{hyperref}
\usepackage{comment}

\numberwithin{equation}{section}
\numberwithin{figure}{section}
\newtheoremstyle{bfnote}%
  {}{}
  {\itshape}{}
  {\bfseries}{.}
  { }{\thmname{#1}\thmnumber{ #2}\thmnote{ (#3)}}

\newtheoremstyle{bfdef}%
  {}{}
  {}{}
  {\bfseries}{.}
  { }{\thmname{#1}\thmnumber{ #2}\thmnote{ (#3)}}

\theoremstyle{bfnote}
\newtheorem{theorem}{Theorem}[section]
\newtheorem{lemma}[theorem]{Lemma}
\newtheorem{proposition}[theorem]{Proposition}
\newtheorem{corollary}[theorem]{Corollary}
\theoremstyle{bfdef}
\newtheorem{definition}[theorem]{Definition}

\newtheorem{remark}[theorem]{Remark}

\newcommand{\bitem}{\begin{itemize}}
\newcommand{\eitem}{\end{itemize}}
\newcommand{\mc}[1]{\mathcal{#1}}
\newcommand{\mb}[1]{\mathbb{#1}}
\newcommand{\mbf}[1]{\mathbf{#1}}

\newcommand{\II}{\mathbb{I}}
\newcommand{\vp}{\mathfrak{p}}
\newcommand{\vP}{\mathfrak{P}}
\newcommand{\vq}{\mathfrak{S}}
\newcommand{\fq}{\mathsf{S}}
\newcommand{\fS}{\mathsf{S}}

\newcommand{\fh}{\mathsf{h}}
\newcommand{\fL}{\mathsf{L}}
\newcommand{\fD}{\mathsf{D}}
\newcommand{\vv}{\mathfrak{v}}
\newcommand{\vS}{\mathfrak{S}}
\newcommand{\cc}{\mathfrak{c}}
\newcommand{\HH}{\varphi}
\newcommand{\fr}{\mathfrak{g}}
\newcommand{\pd}{\mathfrak{d}}
\newcommand{\vt}{\vartheta}
\renewcommand{\S}{\scalebox{0.7}{$\overset{\circ}{\triangle}$}}

\newcommand{\N}{\mathbb{N}}
\newcommand{\R}{\mathbb{R}}

\newcommand{\TT}{\mathbb{T}}

\newcommand{\bpm}{\begin{pmatrix}}
\newcommand{\epm}{\end{pmatrix}}
\newcommand{\bvm}{\begin{vmatrix}}
\newcommand{\evm}{\end{vmatrix}}
\newcommand{\bsm}{\left(\begin{smallmatrix}}
\newcommand{\esm}{\end{smallmatrix}\right)}

\newcommand{\T}{\top}

\newcommand{\la}{\langle}
\newcommand{\ra}{\rangle}

\newcommand{\mrm}[1]{\mathrm{#1}}
\newcommand{\msf}[1]{\mathsf{#1}}
\newcommand{\mfk}[1]{\mathfrak{#1}}

\newcommand{\veps}{\varepsilon}

\newcommand{\w}{\omega}

\newcommand{\gdw}{\Leftrightarrow}
\newcommand{\vphi}{\varphi}

\newcommand{\tth}{\tilde{\theta}}

\DeclareMathSymbol{\mydiv}{\mathbin}{symbols}{"04}

\DeclareMathOperator{\sm}{sm}

\DeclareMathOperator{\logexp}{logexp}

\DeclareMathOperator{\ggrad}{grad}

\makeatletter
\def\widebreve{\mathpalette\wide@breve}
\def\wide@breve#1#2{\sbox\z@{$#1#2$}%
     \mathop{\vbox{\m@th\ialign{##\crcr
\kern0.08em\brevefill#1{0.8\wd\z@}\crcr\noalign{\nointerlineskip}%
                    $\hss#1#2\hss$\crcr}}}\limits}
\def\brevefill#1#2{$\m@th\sbox\tw@{$#1($}%
  \hss\resizebox{#2}{\wd\tw@}{\rotatebox[origin=c]{90}{\upshape(}}\hss$}
\makeatletter

\usepackage{scalerel}
\usepackage{authblk}

\providecommand{\keywords}[1]
{
  \small	
  \textbf{\textit{Keywords:}} #1
  \normalsize
}

\title{Sigma Flows for Image and Data Labeling and \\ Learning Structured Prediction\footnote{\textbf{Acknowledgements.}
This work is funded by the Deutsche Forschungsgemeinschaft (DFG), grant SCHN 457/17-2, within the priority programme SPP 2298: Theoretical Foundations of Deep Learning. This work is supported by the Deutsche Forschungsgemeinschaft (DFG, German Research Foundation) under Germany's Excellence Strategy EXC 2181/1 - 390900948 (the Heidelberg STRUCTURES Excellence Cluster).}
}

\author[1,4]{Jonas Cassel}
\author[1]{Bastian Boll}
\author[2]{Stefania Petra}
\author[3,4]{Peter Albers}
\author[1,4]{Christoph Schn\"{o}rr}
\affil[1]{\small Institute for Mathematics, Image and Pattern Analysis Group, Heidelberg University}
\affil[2]{Institute of Mathematics, Mathematical Imaging Group \& CAAPS, University of Augsburg}
\affil[3]{Institute for Mathematics, Symplectic Geometry Group, Heidelberg University}
\affil[4]{Research Station Geometry and Dynamics, Heidelberg University}
\date{\small\today\normalsize} 

\begin{document}
\maketitle

\begin{abstract}
This paper introduces the \textit{sigma flow model} for the prediction of structured labelings of data observed on Riemannian manifolds, including Euclidean image domains as special case. The approach combines the \textit{Laplace-Beltrami framework} for image denoising and enhancement, introduced by Sochen, Kimmel and Malladi about 25 years ago, and the \textit{assignment flow approach} introduced and studied by the authors. 

The sigma flow arises as the Riemannian gradient flow of generalized harmonic energies and is thus governed by a nonlinear geometric PDE which determines a harmonic map from a closed Riemannian domain manifold to a statistical manifold, equipped with the Fisher-Rao metric from information geometry. A specific ingredient of the sigma flow is the mutual dependency of the Riemannian metric of the domain manifold on the evolving state. This makes the approach amenable to machine learning in a specific way, by realizing this dependency through a mapping with compact time-variant parametrization that can be learned from data. Proof of concept experiments demonstrate the expressivity of the sigma flow model and prediction performance. 

Structural similarities to transformer network architectures and  networks generated by the geometric integration of sigma flows  are pointed out, which highlights the connection to deep learning and, conversely, may stimulate the use of geometric design principles for structured prediction in other areas of scientific machine learning.
\end{abstract}
\\[0.1cm]
\normalsize

\newpage
\tableofcontents
\newpage


\section{Introduction}
\subsection{Overview, Motivation}\label{sec:overview-motivation}

Since its beginnings, imaging science has been employing a broad range of mathematical methods \cite{Scherzer:2015aa}, including models based on partial differential equations (PDEs), variational methods, probabilistic graphical models and differential geometry. In addition, since more than a decade, machine learning has become an integral part of research in \textit{computer vision} in order to deal with complex real-world scenarios. This trend continues, at a slower rate, in the field of \textit{mathematical imaging} where the quest for explainability in methodological research is more emphasized than in computer vision. Naturally, this synergy between mathematical modeling and machine learning has been elaborated most, so far, in connection with the oldest class of problems of the field, image denoising \cite{Elad:2023aa}.

\vspace{0.2cm}
\noindent
The work presented in this paper has been motivated by three lines of research:
\begin{enumerate}[(1)]
\item The \textit{Laplace-Beltrami framework} \cite{sochenGeneralFrameworkLow1998} for low-level vision which introduced the mathematical framework of harmonic maps  \cite[Ch.~9]{jostRiemannianGeometryGeometric2017}
\begin{equation}\label{eq:f-M-N}
f\colon M\to N
\end{equation}
between two Riemannian manifolds 
$(M,h)$ and $(N,g)$, to the field of mathematical imaging and computer vision. $f$ is supposed to minimize the so-called harmonic energy, and the corresponding gradient flow defines a geometric diffusion-type PDE.  
For the specific case $N=\R$, this boils down to functions $f\colon M\to\R$ minimizing the corresponding Dirichlet integral, and specializing to functions $f \colon S^{1} \to N$ results in the familiar geodesic equations as Euler-Lagrange equation. 

In this sense, gradient flows corresponding to the general case \eqref{eq:f-M-N} may be considered as generalized higher-dimensional geodesics. Moreover, by making the Riemannian metric $h$ of the domain manifold $M$ dependent on the evolving state, a broad range of PDE-based models, both established ones and our novel model, may be devised in a systematic way, as shown in the present paper.
\item
\textit{Assignment flows} \cite{astromImageLabelingAssignment2017,schnorrAssignmentFlows2020} provide a framework for the analysis of metric data on graphs, including image (feature) data on grid graphs as special case. The basic idea is to adopt products of statistical manifolds, in the sense of information geometry \cite{amariMethodsInformationGeometry2000}, as state space $(N,\fr)$ equipped with the Fisher-Rao metric $\fr$, and to model contextual inference by flows which emerge from geometric coupling of individual flows on each factor space. Suitable parametrizations of these couplings are amenable to learning these parameters from data, due to the inherent smoothness of the model. Assignment flows, therefore, may be considered as `neural ODEs' from the viewpoint of machine learning.

The most basic instance of this framework concerns the product manifold of open probability simplices. The resulting assignment flows perform \textit{labeling} of metric data on graphs and may be represented as non-local graph-PDEs \cite{Sitenko:2023aa}. A continuous-domain formulation of a special case of assignment flows, on a flat image domain $M\subset\R^{2}$, was studied in \cite{Savarino:2021wt}.
\item
The use of geometric methods for representing both domains and data has become an active field of research in \textit{machine learning} as well \cite{Bronstein:2017aa}. This naturally motivates to consider the synergy between classical methods and data-driven machine learning, beyond image denoising.
\end{enumerate}
The \textit{goal of this paper} is to combine these lines of research in order to extend the Laplace-Beltrami framework (1) to an intrinsic approach for metric data \textit{labeling}. This is achieved by choosing the target manifold $(N,g)$ as the relative interior of the probability simplex equipped with Fisher-Rao metric (2). Like the assignment flow approach, the resulting \textit{sigma flow} approach is smooth and amenable to machine learning (3). Figure \ref{fig:top-level-view} illustrates the leitmotiv of this paper.

\begin{figure}
\centerline{
\includegraphics[width=0.65\textwidth]{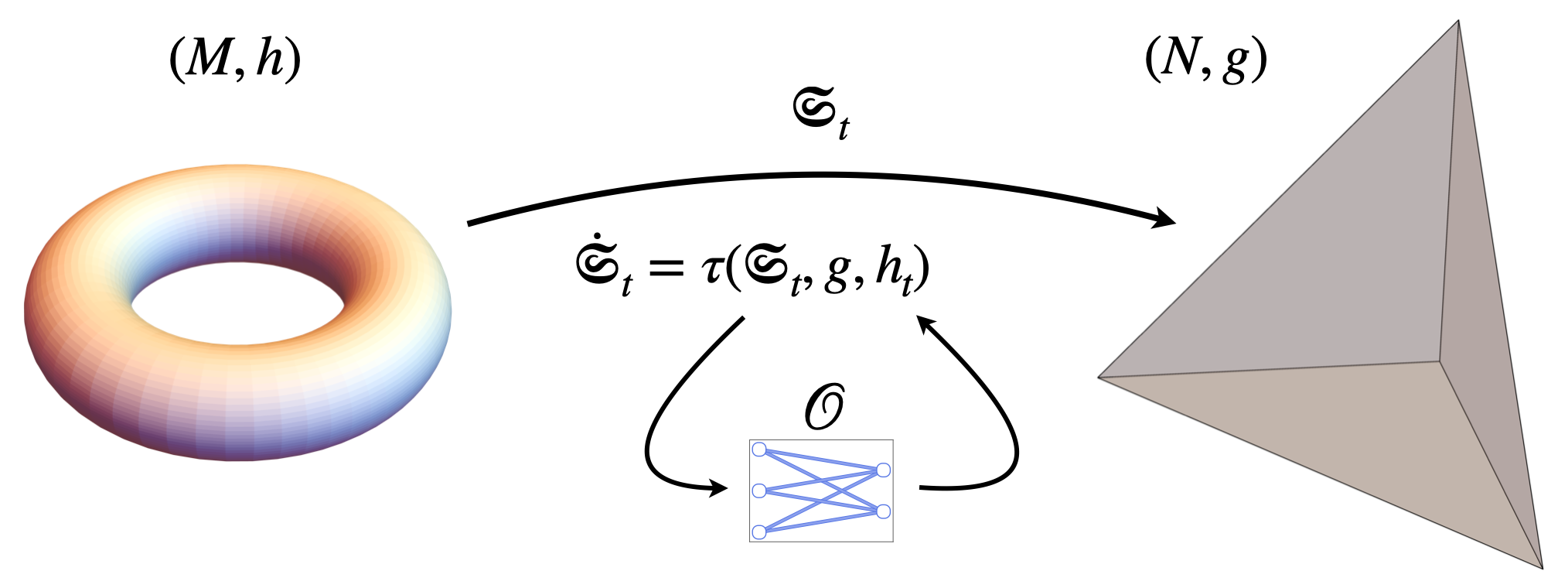}
}
\caption{
\textbf{Sigma flow model.} Top-level sketch of the approach introduced in this paper. The approach combines the assignment flow approach (Section \ref{sec:IG-AF}) and the Laplace-Beltrami framework (Section \ref{sec:Harmonic-Maps-Geometric-Diffusion}) by choosing the interior of the probability simplex, equipped with the Fisher-Rao metric $g=\fr$, as target manifold $N$. This naturally extends nonlinear geometric diffusion to metric data labeling through the novel sigma flow model (Section \ref{sec:Sigmaflow}). Choosing the metric $h$ of the domain manifold $M$ as a function $h_{t}=\mc{O}(\mfk{S}_{t})$ of the state $\mfk{S}_{t}$ whose evolution is governed by the sigma flow, enables to represent a range of established models of PDE-based mathematical image analysis (Sections \ref{sec:Betrami-Variants} and \ref{sec:AID}). In particular, parametrizing the mapping $\mc{O}$ by a neural network enables to learn adaptive sigma flows from data. According to the two-stage parametrization 
$\mfk{S}_{t}\xlongrightarrow{\mc{O}} h_{t}\xlongrightarrow{\Delta_{h_{t}}} \tau(\mfk{S}_{t},g,h_{t}) \xlongrightarrow{\int \dd t} \mfk{S}_{t}$ of the tension field $\tau$ that drives the sigma flow and is generated by the Laplace-Beltrami operator $\Delta_{h_{t}}$, low-dimensional parameterizations of the mapping $\mc{O}$ already achieve strong adaptivity.
}
\label{fig:top-level-view}
\end{figure}

`Sigma flow' reflects the similarity of our model, from the mathematical point of view, to sigma models of mathematical physics -- see Remark \ref{rem:theo-phys-sigma} on page \pageref{rem:theo-phys-sigma} and, e.g., \cite[Ch.~8]{Hori:2003aa} and \cite[Sec. 3.7]{Polchinski:1998rq} for a general discussion. This is not surprising since the original Laplace-Beltrami approach \cite{sochenGeneralFrameworkLow1998} has been also motivated by mathematical models of high-energy physics. 

Harmonic maps between level surfaces of Hessian domains in the sense of \cite{shimaGeometryHessianStructures2007}, and relative to the $\alpha$-connections of information geometry \cite{amariMethodsInformationGeometry2000}, were studied in \cite{Uohashi:2014aa}. While our target manifold $(N,\fr)$ is a Hessian manifold, the domain manifold $(M,h)$, with a metric $h$ learned from data, generally is not. The paper \cite{Uohashi:2014aa} concludes: ``It is an important problem to find applications of non-trivial harmonic maps relative to $\alpha$-connections.'' The approach introduced in the present paper provides such an application using a more general set-up.

\subsection{Related Work}
No attempt is made here to review the vast literature. We merely point out few prominent works in order to contextualize our paper from the three different viewpoints outlined above.

\subsubsection{PDE-Based Image and Multi-Dimensional Data Analysis}\label{sec:related-PDE-AF}

\textit{PDE-based image analysis} has started with the seminal paper \cite{Alvarez:1993aa}, which reports a fundamental study of PDEs whose solutions provide meaningful multiscale transformations of the input data. Here `meaningful' refers to properties like \textit{locality}, \textit{recursivity}, \textit{causality} (for time-variant data) and \textit{invariance} with respect to various transformations. This work still impacts current research. For example, continuous PDE-based formulations of the basic operations of mathematical morphology (dilation, erosion) form the basis for state-of-the-art network architectures that accomplish the equivariant detection of `thin structure'  and perceptual grouping in noisy 2D and 3D image data \cite{smetsPDEbasedGroupEquivariant2023}. For further basic PDE models and the corresponding background, we refer to \cite{Weickert:1998aa}.

Numerous papers of the classical literature on image processing deal with approaches using iterative neighborhood processing that may be more or less related to PDE-based or variational approaches; see, e.g., 
\cite{Manduchi:1998aa} and the survey papers \cite{Buades:2010aa,Milanfar:2013ab}. For a detailed discussion of such schemes from the PDE point of view, we refer to \cite{Buades:2006ue}.

Regarding \textit{variational image denoising}, another line of research based on non-smooth convex functionals and using the total variation (TV) functional as regularizer, has been initiated by \cite{rudinNonlinearTotalVariation1992}. We refer to \cite{Bredies:2010aa,Lefkimmiatis:2015aa,Duran:2016aa} for advanced generalized TV models and to \cite{Caselles:2015ws} for a survey. 
In this context, related variational models for \textit{image labeling} were studied \cite{Lellmann:2011aa,Chambolle:2012aa} which constitute convex relaxations of the combinatorial image partitioning problem. From the viewpoint of contemporary research, the inherent non-smoothness of such approaches constitutes a serious obstacle for enhancing model expressivity by parametrization and, in particular, by \textit{learning}  parameter values from data. 

A powerful class of approaches to \textit{binary image segmentation} in terms of `diffusion and threshold' dynamics was initiated by \cite{Merriman:1994aa}. A survey of this line of research is provided in \cite{Gennip:2014aa,Bertozzi:2016aa}, where an extension to graphs of the underlying Allen-Cahn equation, as $L^{2}$-gradient of the Ginzburg-Landau functional, is studied. A drawback of this approach, from our viewpoint, is the lack of a natural formulation of the image and graph data labeling problem with \textit{multiple} labels (non-binary segmentation).

A foundational paper for \textit{geometric PDE}-based image analysis is the \textit{Beltrami flow} introduced in \cite{sochenGeneralFrameworkLow1998}, which minimizes an energy functional in terms of the embedding map of a given two-dimensional scalar- or vector-valued image into an Euclidean space. This framework can be used to `geometrize' the aforementioned classical PDE-based approaches, like mean curvature motion induced by the total variation measure, Perona-Malik edge-preserving nonlinear diffusion, etc., in order to achieve also image data \textit{enhancement}, besides image denoising, by representing images as manifolds in Euclidean feature spaces \cite{kimmel2000images}. A \textit{continuous}-domain \textit{geometric} perspective turned out to be essential also for interpreting graph-Laplacian based denoising schemes in \cite{Pang:2017aa}.  
The Laplace-Beltrami framework has been extended to generalized Laplacians on vector bundles in \cite{batardHeatEquationsVector2011} and further generalized to equivariant nonlinear diffusion of vector valued data in \cite{batardClassGeneralizedLaplacians2014}, taking the $SO(3)$-action on the HSL color space as a case study. Further examples of works which motivate PDEs from various geometric viewpoints (sub-Riemannian geometry, homogeneous spaces) for data denoising, inpainting, enhancement and thin structure detection, include \cite{Citti:2016aa,Boscain:2018aa,Smets:2021tw}.

\textit{Assignment flows} \cite{astromImageLabelingAssignment2017,schnorrAssignmentFlows2020} denote a class of approaches for the analysis of metric data on graphs and for structured prediction. The basic idea, motivated by information geometry \cite{amariMethodsInformationGeometry2000}, is to assign to each vertex an elementary statistical manifold as state space equipped with the Fisher-Rao metric and to couple the corresponding Riemannian ascent flows by means of a parametrized affinity function across the graph. Geometric integration of the coupled continuous-time flow \cite{zeilmannGeometricNumericalIntegration2020} generates a network with layers indexed by the corresponding discrete points of time, whose parameters are amenable to learning from data by minimizing a suitable loss function \cite{Huhnerbein:2021th}. 
The choice of a particular statistical manifold as state space depends on the data analysis task. The most basic choice, adopted also in this paper, is the relative interior of the probability simplex for node-wise classification, i.e.~data \textit{labeling}. The more expressive case of density matrices as a non-commutative alternative regarding probabilistic models, has been recently introduced and studied in \cite{Schwarz:2023ab}. 

Major differences of assignment flows to the classes of approaches sketched above include (i) that assignment flows constitute a natural approach to \textit{non-binary} labeling with an arbitrary number of labels and (ii) that integral solutions are obtained after convergence by `continuous rounding', induced by the underlying geometry which couples diffusion and rounding in a \textit{single} process.  Stability and convergence to integral solutions was studied in \cite{Zern:2020aa}. Extensions to \textit{unsupervised} and \textit{self-supervised} data labeling were presented in \cite{Zern:2020ab,Zisler:2020aa}. The recent paper \cite{Boll:2024ac} utilizes randomized assignment flows for the \textit{generative} modeling of high-dimensional joint probability distributions of discrete random variables, via measure transport on the assignment manifold and training by Riemannian flow matching.

Assignment flows can be also represented by a \textit{nonlocal} graph-PDE and shown to comprise a range of related earlier work for data labeling on graphs as special case \cite{Sitenko:2023aa}. The \textit{local} counterpart as PDE on a \textit{continuous} domain was introduced in \cite{Savarino:2021wt}, which is the starting point of the present paper.

\subsubsection{Harmonic Maps and Geometric Gradient Flows}\label{sec:Harmonic-Maps-Gradient-Flows}

We focus briefly on the problem to show the existence and global convergence of gradient flows corresponding to energy functionals which determine harmonic maps and their regularity. This requires us to consider more general spaces like $L^{2}(M,N)$ and the corresponding Sobolev space $W^{1,2}(M,N)$ that contain also non-smooth maps $f$.

For compact Riemannian manifolds $M, N$, lower-semicontinuity of the harmonic map energy functional with respect to $L^{2}$-convergence is established via $\Gamma$-convergence in \cite[Ch.~9]{jostRiemannianGeometryGeometric2017}. Furthermore, existence is shown assuming that $N$ has nonpositive sectional curvature, employing convexity properties of the energy which can be deduced in this case. The curvature condition and further assumptions that are violated by our models introduced in the present paper, are also adopted for related scenarios studied, e.g., in 
\cite{Jost:1997ab,Nishikawa:2002aa,Jost:2009aa,Han:2019aa}. 

Assuming that $M$ is a compact connected Riemannian manifold with non-empty boundary and that $N$ is a complete Riemannian manifold, the paper \cite{Hildebrandt:1977aa} established existence for the corresponding Dirichlet problem, merely assuming a positive upper bound of the sectional curvature of $N$.

A major relevant line of research is based on the \L{}ojasiewicz-Simon gradient inequality \cite{Simon:1983ab,Huang:2006ab}. From this angle, the harmonic map problem has been comprehensively studied by 
\cite{Feehan:2019aa} recently, still assuming that the Riemannian manifold $N$ is closed. The approach requires considerable functional-analytic machinery and a corresponding careful study of the Banach manifold structure of the space of Sobolev $W^{k,p}$ maps.

\subsubsection{Machine Learning}

The recent paper \cite{chamberlain2021beltrami} promotes \textit{graph Beltrami flows} as a proper basis for learning continuous features and evolving  the topology of an underlying graph simultaneously. In particular, the authors consider the approach as general enough to overcome a range of limitations of current state-of-the-art deep \textit{graph neural networks (GNNs)}, motivated by the intimate mathematical connection of GNNs to discretized diffusion equations. In fact, there seems to be a trend in machine learning to reconsider concepts like `message passing', `attention' etc.~from a mathematical viewpoint and their relation to established concepts (nonlinear, non-local diffusion, continuous-time models, state-dependent inner products, etc.), in order to categorize the great variety of GNN architectures proposed in machine learning during the recent years \cite{Han:2023aa}.

The graph Beltrami flow proposed by \cite{chamberlain2021beltrami} considers maps $\mc{V}\to\R^{d}\times\R^{d'}$ of graph nodes $i\in\mc{V}$ to $z_{i}=(u_{i},\alpha x_{i}),\,\alpha>0$, comprising positional encodings $u_{i}$ computed in a preprocessing step and continuous feature vectors $x_{i}$. Defining a discrete gradient operator by finite differences and a discrete divergence operator as adjoint with respect to an inner product, yields the discrete Laplace-Beltrami operator and flow. By constraining the resulting diffusivity, the evolution equation can be written in self-adjoint form and shown to be the gradient flow of a discrete version of the Polyakov action studied in \cite{sochenGeneralFrameworkLow1998}.

\subsection{Contribution and Organization}\label{sec:Contribution}

In this paper, we adopt the mathematical framework of harmonic maps $f\colon M\to N$ in order to extend the assignment flow approach to maps of the form \eqref{eq:f-M-N} and the minimization of a corresponding harmonic energy functional, known as Dirichlet energy in the case of functions $f\colon M\to\R$. The target manifold $(N,g)$ will be the interior of the probability simplex $(\S_c,\fr)$ equipped with the Fisher-Rao metric $\fr$. The domain manifold $(M,h)$ with metric $h$ can be any compact Riemannian manifold. 

\vspace{0.1cm}
Since our scenario violates basic assumptions made in the literature above ($N$ is open with positive sectional curvature, non-metric affine connection) and generalizes the basic harmonic map problem to sigma models, we leave the problem of existence and global convergence (cf.~Section \ref{sec:Harmonic-Maps-Gradient-Flows}) of the gradient flow for future work and solely focus on \textit{geometric} aspects in this paper.  

Throughout the paper, we make the assumptions: $(M,h)$ will be a compact, oriented connected Riemannian manifold without boundary and we consider smooth maps $f\in C^{\infty}(M,N)$.
Specifically, in the case of images, we choose the torus $M=\mb{T}^{2}$ corresponding to the image domain possibly extended by a constant margin, and doubly-periodic boundary conditions, with a metric $h$ induced by data. The compactness assumption ensures the direct application of established results about the spectrum of the Laplace-Beltrami operator, as a basis to devise a Lyapunov functional for the new sigma flow model. This substantiates numerical experiments with $N$-valued sigma flows on a graph embedded in $M$, i.e.~after a discretization of $M$.

The main purpose of this paper is to provide, from the viewpoint of geometric modeling and using the framework of harmonic maps, a \textit{continuous}-domain extension $\vP\colon M\to N$ of assignment flows with a \textit{learnable time-variant metric} $h$ of the domain manifold $M$. Since the metric also depends on the evolving state governed by the sigma flow, compact parametrizations enable decent model expressivity. 

Our new approach, the \textit{sigma flow} model, is general and applies to labeling tasks of data observed on any compact domain manifold $(M,h)$. In addition, we consider our result as a mathematical approach to geometric deep learning, contributing a design principle for the generation of neural ODEs by discretizing sigma flows, that accomplish structured data labelings in natural manner using concepts of information and differential geometry. The similarity of our approach to concepts employed in mathematical physics may be of independent interest.

\vspace{0.2cm}
\noindent
The paper is organized as follows.
\begin{description}
\item[Section \ref{sec:IG-AF}] introduces basic notation, recalls the information geometry of the target manifold $(N,g)=(\S_c,\fr)$ and a particular formulation of the assignment flow 
\cite{Savarino:2021wt} as starting point. The reformulation of the assignment flow approach as nonlinear nonlocal graph PDE \cite{Sitenko:2023aa} characterizes assignment flows as a tool for generating graph-based neural networks for metric data labeling. 
\item[Section \ref{sec:Harmonic-Maps-Geometric-Diffusion}] recalls basic notions related to harmonic maps and the Beltrami flow approach \cite{sochenGeneralFrameworkLow1998}. Variants of this approach are obtained by making the domain metric $h$ dependent on the evolving map. The relation to nonlinear anisotropic diffusion, in particular, is considered in more detail from this viewpoint.
\item[Section \ref{sec:Sigmaflow}] presents our main distribution, the sigma flow model. The general formulation is complemented by concrete implementable expressions using the two basic affine coordinate systems of information geometry. Convergence of the solution are shown under the assumptions stated above about the domain manifold $M$ and the smoothness of mappings $f$. Finally, we consider the entropic harmonic energy functional which turns the sigma flow model into a proper labeling approach.
\item[Section \ref{sec:Experiments}] provides implementation details and few experimental results concerning model expressivity and prediction performance, as proof of concept. 
\item[Section \ref{sec:Conclusion}] Comparisons and structural similarities to the S-flow version of the assignment flow approach (Section \ref{sec:Comparison-S-Flow}) and to transformer network architectures (Section \ref{sec:transformer}), respectively, point out the relevance of the sigma flow model also in a broader context, that we take up and briefly discuss to conclude the paper.
\item[Appendix \ref{appendix}] lists symbols and their definitions.
\end{description}

\section{Information Geometry and Assignment Flows}\label{sec:IG-AF}
We collect a few definitions and set notation which will be used throughout this paper. Appendix \ref{appendix} lists  symbols and their definitions.

\subsection{Basic Notation}

Let $M$ be a closed, oriented, connected smooth manifold of dimension $m$. For a function $f \in C^{\infty}(M)$, we use the shorthand
\begin{subequations}\label{eq:int-M}
\begin{align}\label{eq:int-M-a}
\int_M f &\coloneqq \int_{M} f \w_{h}
\intertext{
with volume measure $\w_{h}$ defined by the metric $h$ and locally given as} 
\w_{h}&=\sqrt{ | h| }\dd x^{1}\dotsb \dd x^{m} 
=: \sqrt{ | h| }\dd x,
\end{align}
\end{subequations}
where $| h | = \det h$ is the determinant of the metric tensor. We often omit the argument $x$ of functions in integrals, like in \eqref{eq:int-M-a}, to enhance the readability of formulae. 
For vector bundles $E$ over $M$, we denote by $\Gamma(E)$ the global sections of $E$. We furthermore use
\begin{equation}
    \Omega^{1}(M,E) = \Gamma\big(T^\ast M \otimes E \big).   
\end{equation}
and $\Omega^{1}(M) = \Omega^{1}(M,M \times \R)$ for one-forms.
We set
\begin{subequations}
\begin{align}
[n] &:= \{1,2,\dotsc,n\},\qquad\forall n\in\N \\
[n]_{0} &:= \{0,1,\dotsc,n\},\qquad\forall n\in\N.
\end{align}
\end{subequations}
Greek indices denote coordinates for $M$ and roman indices coordinates on the specific target manifold $N=\S_c$ (Section \ref{sec:Hessian-Simplex}). In this case, the general Riemannian metric $g$ of $(N,g)=(\S_c,\fr)$ is denoted by $\fr$. Local coordinates on $M$ are denoted by  $x^{\mu },\mu\in[m]$ with coordinate derivative operators 
\begin{equation}
\partial _\mu  \coloneqq  \pdv{x^{\mu }},\qquad\mu\in[m].
\end{equation}
We denote the number of categories (classes, labels) by $c \in \N$ and set 
\begin{equation}\label{eq:def-cc}
\cc \coloneqq c-1. 
\end{equation}
The canonical (natural, exponential) local coordinates on $\S_c$ are denoted by $\theta^{i},\,i\in[\cc]$ with coordinate derivative operators 
\begin{equation}
    \label{eq:def-pdi}
\pd_i \coloneqq \pdv{\theta ^{i}},\qquad i\in[\cc].
\end{equation}
The identity matrix is denoted by
\begin{equation}
\II=\II_{n} = (\delta_{ij})_{i,j\in[n]},
\end{equation}
with the Kronecker symbol $\delta_{ij}$, and with the dimension $n$ indicated as subscript whenever the dimension may not be clear from the context. Angular brackets $\la\cdot,\cdot\ra$ are generically used for denoting inner products, with the symbol $g$ of the metric as subscript in the case of a Riemannian metric $\la\cdot,\cdot\ra_{g}$. The Einstein summation convention is employed throughout this paper. 
For two functions $f,g \in C^{\infty}(M)$, we write
\begin{equation}
    \la \dd f , \dd g \ra_{h} : x \mapsto  h^{\mu \nu }(x) \partial _\mu  f(x) \partial _\nu g(x)
\end{equation}
for the natural pairing of one-forms induced by $h$. 

\subsection{Hessian Geometry of the Probability Simplex}
\label{sec:Hessian-Simplex}

This section defines few basic concepts and notation related to the geometry of the target manifold $(\S_c,\fr)$. 
The probability simplex of categorial distributions is denoted by
\begin{equation}\label{eq:def-Delta-p}
    \triangle_c = \bigg\{p= (p_0,\dots ,p_\cc) \in \R^{c} \colon p_i \ge 0, \sum_{i\in[\cc]_{0}} p_i = 1 \bigg\}.
\end{equation}
Its relative interior is a smooth manifold
\begin{equation}\label{eq:def-Sc}
    \overset{\circ}{\triangle}_c = \bigg\{p=(p_0,\dots ,p_\cc) \in \R^{c} \colon p_i > 0, \sum_{i\in[\cc]_{0}} p_i = 1\bigg\}.
\end{equation}
Each distribution $p\in\S_c$ has full support and governs a discrete $[c]$-valued random variable $Y$ with $p_{i}=\Pr(Y=i),\,i\in[\cc]$ and $p_{0}=\Pr(Y=c)$. The manifold \eqref{eq:def-Sc} is covered by the two single coordinate charts $\chi_{e}$ and $\chi_{m}$ given by
\begin{subequations}\label{eq:def-kappa-chart}
\begin{align}
    \chi_{e}: \S_c &\to \R^{\cc}, &
    p = (p_0 , \dots , p_\cc) &\mapsto 
    \theta = \Big(\log \frac{p_1}{p_0} , \dots  ,\log \frac{p_\cc}{p_{0} }\Big), 
\label{eq:def-kappa-chart-e} \\ \label{eq:def-kappa-chart-m}
    \chi_{m} : \S_c &\to \bigg\{p \in \R^\cc \colon p_i > 0, \sum_{i\in[\cc]} p_i <1\bigg\}, & 
    p =(p_0, \dots ,p_\cc) &\mapsto (p_1,\dots ,p_\cc). 
\end{align}
\end{subequations}
Denoting the corresponding coordinate functions by
\begin{subequations}\label{eq:def-theta-p}
\begin{align}
    \theta^i : \S_c &\to \mathbb{R},  &
    p_i : \S_c &\to \mathbb{R}, &
    i &\in[\cc],
\intertext{one has}
\theta^{i}(p) &= \log \frac{p_i}{1- \sum_{j\in[\cc]}p_j}, &
p_i(\theta) &= \frac{e^{\theta^{i}}}{1 + \sum_{j\in[\cc]} e^{\theta^{j}}}, & 
i &\in [\cc].
\end{align}
\end{subequations}
The use of sub- and superscripts here is intentional: if $\theta^{i},\,i\in[\cc]$ transform contravariantly, then $p_{i},\,i\in[\cc]$ transform covariantly. Consequently, $p_{i}\dd\theta^{i}$ is a differential one-form, which turns out to be exact 
\begin{equation}
    \dd \psi = p_i \dd \theta ^{i},
\end{equation} 
with the potential given by the 
log-Laplace transform (partition function)
\begin{equation}\label{eq:def-psi-partition}
\psi : \S_c \to \R, \qquad 
\psi(\theta) = \log\bigg( 1+ \sum_{j\in[\cc]}e^{\theta^j}\bigg).
\end{equation}
The Legendre transform yields as conjugate potential
\begin{equation}
\vphi := \psi^{\ast}
\end{equation}
the negative entropy
\begin{equation}
\label{eq:HH}
\varphi : \S_c \to \R,\qquad
\varphi(\theta) = \sum_{j\in[\cc]}e^{\theta^j - \psi(\theta)} \theta^j  - \psi(\theta) = p_j \theta^j - \psi(\theta).
\end{equation}
Basic relations include
\begin{subequations}
\begin{align}
\label{eq:partial-psi}
\pd_{i}\psi 
&= p_i = e^{\theta^{i} - \psi}, & i&\in[\cc],
\\ \label{eq:partial-vphi}
\pd_{i}\vphi
&= \fr_{ij}\theta^{j}, & i&\in[\cc],
\intertext{with the metric tensor of the Fisher-Rao metric} 
\label{eq:fr_metric}
\fr_{ij} &= \pd_{i} \pd_{j} \psi = \pd_{j} p_{i}
= \delta_{i j} p_{i} - p_{i} p_{j},
& i,j &\in [\cc].
\end{align}
\end{subequations}
The Christoffel symbols of the Levi-Civita (metric, Riemannian) connection with respect to the $\theta$ coordinates read
\begin{equation}
\label{eq:gamma0}
\Gamma^{i}_{jk} = 
\frac{1}{2} \fr^{il}\pd_{l}\pd_{j}\pd_{k} \psi = \frac{1}{2}(\delta^{i}_{j}\delta_{jk} - \delta^i_j p_k - \delta^i_k p_j),
\end{equation}
where $\delta^{i}_{j}\equiv\delta_{ij},\forall i,j$. 
The Christoffel symbols of $\alpha$-connection in $\theta$-coordinates are then given by \cite{amariMethodsInformationGeometry2000}
\begin{subequations}\label{eq:def-alpha-Gamma}
\begin{align}
    \,^{\alpha}\Gamma ^{i}_{jk} &= \frac{1-\alpha}{2} \,^{-1}\Gamma^{i}_{jk} + \frac{1+\alpha}{2} \,^{1}\Gamma^{i}_{jk} 
    \\
    &=  \frac{1-\alpha}{2} \,^{-1}\Gamma ^{i}_{jk} = (1-\alpha)\Gamma ^{i}_{jk}, & i,j,k &\in [\cc],
\intertext{with}
\frac{1}{2} \,^{-1}\Gamma^{i}_{jk} 
&= \Gamma^{i}_{jk}, & i,j,k &\in [\cc].
\end{align}
\end{subequations}

\subsection{S Flows}\label{sec:s-flows}
As briefly reported in Section \ref{sec:related-PDE-AF}, assignment flows provide a framework for labeling metric data observed on graphs, at every node, utilizing the geometric structure of the simplex $\triangle_{c}$ defined above. We confine ourselves to a particular parametrization of assignment flows, called \textit{S Flows} \cite[Section 3.2]{Savarino:2021wt}.

Let $\mc{G}_{\Omega} = (\mc{V},\mc{E},\Omega)$ be an undirected weighted graph with $\mc{V}=[N],\; N\in\N$ 
and non-negative weight matrix
\begin{equation}\label{eq:omega-s}
    \Omega \in \R^{N \times N}, \quad \Omega_{ab}\geq 0,\quad \forall a,b \in [N],
\end{equation}
which is symmetric and supported on the edges
\begin{equation}
    \Omega_{ab} = \Omega_{ba} \qquad\text{and}\qquad
    \Omega_{ab} = 0 \qq{if} (a,b) \notin \mathcal{E}.
\end{equation}
The S flow is a dynamical system evolving on the \textit{assignment manifold}
\begin{equation}
    \label{eq:def-mcW}
\big(\mc{W}^{N}_c,\fr\big),\qquad
\mc{W}^{N}_c := \underbrace{\S_c\times\dotsb\times\S_c}_{N \;\text{factors}}
\end{equation}
with the Fisher-Rao product metric $\fr$ defined factorwise by \eqref{eq:fr_metric}. Elements $\fS$ of the assignment manifold $\mathcal{W}^{N}_c$ are conveniently represented by \textit{assignment matrices} $\fS \in \R^{N \times c}$ with rows $\fS_{a} \in \S_c,a \in [N]$.
The S flow is the Riemannian gradient descent flow corresponding to the objective function
\begin{equation}
    \label{eq:def-JJ}
\mc{J}: \mathcal{W}^{N}_c  \to \R, \qquad  
    \mc{J}(\fS) := - \frac{1}{4}\sum_{a,b\in [N]} \Omega_{ab} \norm{\fS_a - \fS_b}^2 + \frac{1}{2} \sum_{a \in [N]}  \norm{\fS_a}^2
\end{equation}
given by the equation
\begin{equation}
    \label{eq:def-s-flow-0}
\dot \fS = -\ggrad \mc{J}(\fS),\qquad \fS(0) = \msf{P}.
\end{equation}
The Riemannian gradient can be specified explicitly with some more notation.
For a given assignment matrix $\fS$ we define the \textit{replicator tensor} $R_{\fS} \in \R^{N \times c \times c}$ with entries
\begin{equation}\label{eq:replicator-tensor}
    (R_{\fS})_a = \mathrm{Diag}(\fS_a) - \fS_a \otimes \fS_a \in \R^{c\times c}, \quad a\in N. 
\end{equation}
The tensor $R_{\fS}$ acts on elements of $\R^{N \times c}$ by row-wise matrix multiplication, projecting to the tangent space $T \mathcal{W}^{N}_c \subset \R^{N \times c}$. For $Y \in \R^{N\times c}$, we write
\begin{equation}
    R_\fS Y \in T\mathcal{W}^{N}_c \qq{with rows}
    (R_\fS Y)_a = (R_\fS)_a Y_a. 
\end{equation} 
Furthermore, we denote by 
\begin{equation}\label{eq:Omega-L}
    \fL_{\Omega} \coloneqq \Omega - \II_{N} \in \R^{N \times N}
\end{equation}
the $\Omega$-induced graph-Laplacian, acting on assignment matrices by matrix multiplication. 
Using this notation, the gradient of $J$ takes the form
\begin{equation}
    -\ggrad J(\fS) = R_S(\fL_\Omega \fS + \fS).
\end{equation}
The \textbf{S flow}
\begin{equation}
    \label{eq:def-s-flow}
    \dot \fS  = R_\fS (\fL_{\Omega}\fS + \fS)  
\end{equation}
is thus as a dynamical system parametrized by the weight matrix $\Omega$. Under mild conditions on $\Omega$ \cite{Zern:2020aa} and sufficiently large $T>0$, $\fS(T)$ satisfies the entropy constraint
\begin{equation}
    \label{eq:entropy-eps}
    H\big(\fS_{a}(T)\big) < \veps,\qquad\forall a\in\mc{V},
\end{equation}
where $H = -\varphi \colon\S_c\to\R$ denotes the categorical entropy function, see \eqref{eq:HH}. This implies that at every node $i \in \mc{V}$, the corresponding row $\fS_{a}(T) \in \S_c$ of the assignment matrix is very close to a unit vector which uniquely assigns the corresponding label to given feature data $f \colon \mc{V} \to \mc{F}$, where $\mc{F}$ is any metric space and $f$ is encoded by the initial point $\msf{P}$ of \eqref{eq:def-s-flow-0}.
If the weights $\Omega$ are allowed to be adjusted, the S flow ODE \eqref{eq:def-s-flow} may be interpreted as a \textit{neural ODE} \cite{chen2018neural}, where the weights $\Omega$ can be learned from data. Among other choices \cite{Boll:2024aa}, $\Omega$ could be parametrized by a deep neural network, as demonstrated e.g.~in \cite{Boll:2023ab}.

More abstractly, we can think of \eqref{eq:def-s-flow} as being parametrized by the Laplace-operator $\fL_{\Omega}$ rather than by the weights themselves. This perspective is useful regarding generalizations of this formalism, due to the plethora of different structures admitting Laplacian operators, including simplicial complexes \cite{desbrun2005discrete,lim2020hodge}, meshes \cite{Gu:2002aa} and manifolds \cite{jostRiemannianGeometryGeometric2017}. This general perspective on the S flow is the departure point of the present paper.  The goal is to take advantage of the combination of ideas related to data labeling based on assignment flow architectures, with concepts from geometric data processing based on the \textit{manifold hypothesis} \cite{Fefferman:2016aa}.

To this end, the paper \cite{Savarino:2021wt} provides a natural starting point for our work, where a continuum limit for the S flow was proposed, replacing the graph $\mathcal{G}$ by an open domain $X \subset \R^{d}$ for some $d \in \N$. 
The objective function to be minimized is then replaced by the \textbf{S flow energy}
\begin{equation}\label{eq:cont-domain-s-flow}
    E_S : C^{\infty}(X,\S_c) \to \R, 
    \qquad 
    E_S(\vP) = \frac{1}{2}\int_X (\norm{\partial \vP}^2 - \norm{\vP}^{2}). 
\end{equation}  
Numerical optimization of this functional allows to perform data labeling similar to the S flow. 
In Section \ref{sec:Sigmaflow}, we explore how this approach can be generalized when a Riemannian manifold $(M,h)$ is considered instead of a Euclidean domain $X$, with a metric $h$ depending on data.


\section{Harmonic Maps and Geometric Diffusion}\label{sec:Harmonic-Maps-Geometric-Diffusion}
\subsection{Riemannian Harmonic Maps}
\label{sec:Riemann-Harmonics}
Harmonic maps originate in differential geometry when minimizing the energy of functions as defined below. For background and further reading, we refer to \cite{Helein:2008aa}, \cite[Ch.~9]{jostRiemannianGeometryGeometric2017}. 

Let 
\begin{equation}\label{eq:Mh-Ng-general}
(M,h)\qquad\text{and}\qquad (N,g)
\end{equation}
denote smooth, oriented Riemannian manifolds without boundary of dimensions $m$ and $n$, respectively. Furthermore, let $M$ be compact.
For a function $f \in C^{\infty}(M,N)$ with coordinate functions $f^{i},i=1,\dots , n$, its differential 
\begin{equation}\label{eq:df-space}
\dd f\in \Omega^{1}(M,f^{\ast}TN)
\end{equation}
is a section of the vector bundle $T^{\ast} M\otimes f^{\ast}T N$ of $1$-forms with values in the pullback bundle $f^{\ast}TN$ over $M$. The latter is equipped with the pullback metric $f^* g$ whereas $T^{\ast}M$ carries the metric $
h^{-1}$. Thus, denoting the corresponding metric by
\begin{equation}
\la\cdot,\cdot\ra := \la\cdot,\cdot\ra_{T^{\ast}M\otimes f^{\ast}TN},
\end{equation}
one locally has with the induced norm $\|\cdot\|=\sqrt{\la\cdot,\cdot\ra}$,
\begin{equation}
\|\dd f\|^{2}\in C^{\infty}(M),\qquad
\|\dd f\|^{2}(x) = h^{\mu \nu}(x) \partial_\mu  f^{i}(x) \partial_\nu f^{j}(x) g_{ij}\big(f(x)\big).
\end{equation}
We call the functional
\begin{equation}\label{eq:def-EH-general}
E_{H}\colon C^{\infty}(M,N)\to\R,\qquad
E_{H}(f) = \frac{1}{2}\int_{M}\|\dd f\|^{2}
\end{equation}
the \textit{harmonic energy} of $f$. Let $f_{t}\colon M\times (-\veps,\veps)\to N$ denote a smooth one-parameter family of variations of $f_{0}=f$, then the first variation of the harmonic energy is given by
\begin{equation}\label{eq:def-EH}
(\delta E_{H})(f)[v] := \frac{\dd}{\dd t} E_{H}(f_{t})\big|_{t=0}
= -\int_{M}\la\tau(f; g,h),v\ra \w_{h},\qquad
v = \frac{\partial  f_{t}}{\partial t}\Big|_{t=0},
\end{equation}
with the \textbf{tension field} of $f$ given by
\begin{equation}
\label{eq:tension_field}
\tau(f; g,h) := \tr\nabla\dd f,
\end{equation}
where $\nabla$ denotes the induced connection on $T^{\ast}M\otimes f^{\ast}T N$. 
Critical points $f$ of $E_{H}$ are called \textit{harmonic maps}. The corresponding Euler-Lagrange equations $\tau(f; g,h)=0$ are more explicitly given by
\begin{equation}\label{eq:tension-f-explicitly}
0 = \Delta_{h} f^{i}(x) + \Gamma^{i}_{jk}\big(f(x)\big) \partial_{\mu } f^{j}(x) \partial _\nu f^{k}(x) h^{\mu \nu}(x), \quad x \in M,\quad i\in[n],
\end{equation}
with the Christoffel symbols $\Gamma^{i}_{jk}$ associated with the metric $g$ on $N$ and the Laplace-Beltrami operator on $(M,h)$ given by
\begin{equation}\label{eq:delta-EH}
\Delta_h f 
= \frac{1}{\sqrt{|h|}} \partial_\mu ( \sqrt{|h|} h^{\mu \nu} \partial _\nu f),\qquad f\in C^{\infty}(M).
\end{equation}
The dependency of the tension field \eqref{eq:tension_field} on the metrics $g$ and, in particular, on $h$ of \eqref{eq:Mh-Ng-general} will be key ingredients of models considered in Sections \ref{sec:Variants}, \ref{sec:AID} and \ref{sec:Sigmaflow}.

\begin{remark}[sign convention of $\Delta_h$]
    \label{rem:sign-laplace}
    The sign convention used here agrees with the one from \cite[Eq.~(38)] {sochenGeneralFrameworkLow1998} and \cite[Section I.1, Eq.~(33)]{chavelEigenvaluesRiemannianGeometry1984}, which is opposite to the convention in \cite[pp.~18]{Rosenberg:1997aa}, \cite[pp.~496]{jostRiemannianGeometryGeometric2017}. Consequently, the operator $\Delta_h$ is \textit{negative semi-definite} with a non-positive spectrum. 
\end{remark}

\begin{remark}[harmonic maps in theoretical physics]\label{rem:theo-phys-sigma}
    We briefly comment on the role of harmonic maps in theoretical physics. 
    Physical theories describing maps $f:M\to N$ between manifolds $M,N$ are generally referred to as \textit{sigma models} \cite[pp.~146]{Hori:2003aa} (the nomenclature is due to historical reasons), where it is often assumed that the metric $h$ is not fixed but dynamical, however. The harmonic energy functional \eqref{eq:def-EH-general} is known as the \textit{non-linear sigma model action with target $N$} in the context of quantum field theory and string theory \cite[Sec.~3.7]{Polchinski:1998rq}. 
    A distinguished member of the family of non-linear sigma models is the \textit{Polyakov action} which assumes the special case $N = \R^{d}$ with the Euclidean metric or Lorentzian pseudo-metric. This case is of great interest in bosonic string theory as laid out in \cite[Sec.~1.2]{Polchinski:1998rq} and \cite{polyakovQUANTUMGEOMETRYBOSONIC1981}, allowing for tractable quantization.
\end{remark}

\subsection{Beltrami Flow and Variants}\label{sec:Betrami-Variants}
\subsubsection{Beltrami Flow}
The \textit{Beltrami flow approach} \cite{sochenGeneralFrameworkLow1998} considers images as mappings from surfaces to the RGB color space: A given image array $\mathsf{f} \in \R^{N\times 3}$ arises as discretization of a mapping $f : M \to \R^{3}$ where, for simplicity, we assume $M$ to be a smooth closed two-dimensional manifold. A basic example are images on a torus $M = \mb{T}^{2}$ with periodic boundary conditions. 

Choosing Riemannian metrics 
\begin{equation}
(M,h)\qquad\text{and}\qquad (\R^{3},g)
\end{equation}
yields an instance of the harmonic map setting \eqref{eq:Mh-Ng-general} and one may consider the harmonic energy $E_{H}(f)$ of $f$ given by \eqref{eq:def-EH}. The Beltrami flow approach amounts to process $f$ by minimizing $E_{H}$ and to integrate the corresponding gradient descent equation. Setting for a given $T > 0$
\begin{equation}\label{eq:def-MT}
    M_{T} \coloneqq (0,T) \times  M
\end{equation}
and for $u \in C^{\infty}(M_T,\R^{3})$ and $t \in (0,T)$ fixed, we write 
\begin{equation}\label{eq:def-ut}
    u_{t}=u(t,\cdot)\in C^{\infty}(M,\R^{3}), \qq{and} u_{0} := \lim_{t \to 0^{+}} u(t,\cdot).
\end{equation}
The \textbf{basic Beltrami flow} system reads
\begin{equation}
    u \in C^{\infty}(M_{T},\R^{3}) \qq{s.t.} 
    \begin{dcases}
        \partial_t u_t = \tau(u_t; g,h), \quad  \forall  t \in (0,T), \\  
        u_0  = f, 
    \end{dcases}
\end{equation}
with $\tau$ given by \eqref{eq:tension_field}. A common choice is the Euclidean metric 
\begin{equation}\label{eq:color-g-delta-metric}
    g_{ij}=\delta_{ij}
\end{equation}
for the color space. See, e.g., \cite{Resnikoff:1974aa,Provenzi:2016aa} for color spaces that better conform to human color perception.
    
\subsubsection{Variants: Dynamic Metrics}\label{sec:Variants}
Variants of the Beltrami flow approach result from coupling the metric $h$ and the function $u$ via a differential equation. Instead of a single metric $h$, we consider a family of metrics $h_t$ depending on $u_t, t \in (0,T)$. A natural choice is the system 
\begin{equation}\label{eq:Beltrami-ht}
    u \in C^{\infty}(M_{T},\R^{3}) \qq{s.t.} 
    \begin{dcases}
        \partial_t u_t = \tau(u_t; g, h_{t}), \quad & \forall t > 0, \\  
        h_t = u_t^{*} \delta, \quad & \forall t > 0,\\ 
        u_0  = f,
    \end{dcases}
\end{equation}
with $h_{t}$ is determined by pulling back the Euclidean metric \eqref{eq:color-g-delta-metric} via $u_{t}$. A particular case concerns mappings
\begin{equation}
f\colon M\to\R^{3},\qquad
f(x) = f(x_{1},x_{2}) = \big(x_{1},x_{2}, k(x_{1},x_{2})\big)
\end{equation}
defined as graph of a function $k\colon M\to\R$. Then $u_{t}$ solving \eqref{eq:Beltrami-ht} defines a family of surfaces governed by the \textbf{mean curvature flow} equation
\begin{equation}\label{eq:def-MCF}
    \partial _t u_t = 2 \mc H(u_t) \mc N(u_t),
\end{equation}
where $\mc{H}(u_{t})$ is the Gaussian mean curvature and $\mc{N}(u_{t})$ is the unit normal to the surface defined by $u_t$. For a derivation of this equation, see \cite[Sec. 4.3]{sochenGeneralFrameworkLow1998}. 
We refer, e.g., to \cite{Weickert:1998aa,Smets:2021tw} for further reading, to \cite{Morel:2015aa} for connections to local adaptive filtering, to \cite{Garcke:2013aa} for connections to other areas of applied mathematics, and to  \cite[Sec. 1.2 and 3.7]{Polchinski:1998rq} for relations to theoretical physics.

\subsection{Anisotropic Image Diffusion}\label{sec:AID}
The anisotropic diffusion approach to image processing promoted by Weickert \cite{Weickert:1998aa} adopts a somewhat complementary viewpoint. We briefly elucidate differences to, and common aspects with, the Beltrami flow approach.

For the basic case of a gray value image function $f\colon \R^{2}\supset\Omega\to\R$, the system of \textbf{anisotropic diffusion equations} reads
\begin{equation}\label{eq:AnisDiff}
    u \in C^{\infty }(\Omega_{T}) \qq{s.t.} 
    \begin{dcases}
        \partial_t u_t = \mathrm{div}(V_{t}\nabla u_t) , \quad & \forall  0 < t < T, \\  
        V_t = \mc{O}(u_t) ,\quad & \forall t > 0,\\ 
        u_0  = f, 
    \end{dcases}
\end{equation}
with partial differential operator $\mc{O}$ and a matrix-valued \textit{diffusion tensor} $V\colon\Omega_{T}\to\R^{2\times 2}$ satisfying a uniform positive definiteness constraint $\mc{O}(f) \succeq C(\mc{O})\II$, for all $f \in C^{\infty}(\Omega)$ and a positive constant $C(\mc{O})$. The class of operators $\mc{O}$ considered in \cite{Weickert:1998aa} have the form
\begin{subequations}\label{eq:AnisDiff-def-mcL}
\begin{align}
    (\mc{O}(f))(x) 
    &= D\big( K_{\rho} * ( \nabla f_{\sigma}  \otimes \nabla f_{\sigma}  )(x)\big) 
    \\ \label{eq:D-ast-delta}
    &= D \big(K_{\rho} * (f_{\sigma}^{\ast} \delta )(x)\big),
    \qquad\qquad
   f_{\sigma} := K_{\sigma}\ast f,\qquad
   \rho,\sigma>0,
 \end{align}
\end{subequations}
where $D\in C^{\infty}(\R^{2\times 2},\R^{2\times 2})$ preserves symmetry and positive definiteness of the matrix argument, $\ast$ denotes spatial convolution and $K_{\rho},K_{\sigma}$ are lowpass (typically: Gaussian) filter kernels at scales $\rho$ and $\sigma$, respectively.

The approach \eqref{eq:AnisDiff} has a more narrow scope in that possible manifold structures on $\Omega$ are ignored. On the other hand, in view of the second equation of \eqref{eq:AnisDiff} governing $V_{t}$, Equation \eqref{eq:D-ast-delta} generalizes the role of $h_{t}$ in \eqref{eq:Beltrami-ht}.

In order to bring the anisotropic diffusion approach closer to the Beltrami flow approach, we introduce an additional positive warp factor $\kappa\colon\Omega_{T}\to\R_{>}$ and define the \textbf{warped anisotropic diffusion (WAD)}, which reads
\begin{equation}\label{eq:WAD-system}
    u_{t} \in C^{\infty }(\Omega_{T}) \qq{s.t.} 
    \begin{dcases}
        \partial_t u = \kappa_t \mathrm{div}(V_{t} \nabla u_t) , \quad & \forall t > 0, \\  
        V_t = \mc{O}(u_t) ,\quad & \forall t > 0,\\ 
        u_0  = f. 
    \end{dcases}
\end{equation}
If we replace the Euclidean domain $\Omega$ by a domain manifold $M$, then the Beltrami flow becomes a special case of the WAD with
\begin{equation}
    V_{t}   = h_{t}^{-1}\sqrt{ |h_t| }  \qq{and} \kappa_t = \frac{1}{\sqrt{ |h_t|}},
\end{equation}
assuming the matrices $h_{t}\succ 0$ are symmetric and positive definite. This characterizes the Beltrami flow as WAD with a diffusion tensor that has a unit determinant, and it enables to exploit established numerical methods for anisotropic diffusion after choosing a coordinate system on $M$.

Alternatively, we may generalize the Beltrami flow system to produce WAD systems with more general diffusion tensors. For a $C^{1}$ function $B\colon\R\to\R$ with $B'>0$, consider the generalized harmonic energy
\begin{equation}
    \label{eq:def-EB}
E_{B}\colon C^{\infty}(M,N)\to\R,\qquad
E_{B}(f) := \frac{1}{2}\int_{M} B(\|\dd f\|^{2}).
\end{equation}
    Calculating the functional derivative yields
    \begin{equation}
        (\delta E_B)(f)[v] = - \int_M \Braket{\tau_B(f; g,h),v}, \quad v = \pdv{f_t}{t}\Big|_{t=0},
    \end{equation}
    where $\tau(f; g,h) \in \Gamma (f^{\ast} TN)$ is locally given by
    \begin{equation}\label{eq:tau-Beltrami-WAD}
        \tau_B(f; g,h)^{i} = 
        \frac{1}{\sqrt{|h|}} \partial _\mu \qty(B^{\prime} (\|\dd f\|^2) \sqrt{ |h| } \, h^{\mu \nu } \partial _\nu  f ^{i}) + B^{\prime} (\|\dd f\|^2) \Gamma ^{i}_{jk} \partial _\mu f^{j} \partial _\nu  f^{k} h^{\mu \nu },\qquad i\in[n].
    \end{equation}
Choosing specifically $(N,g)=(\R^{n},\delta)$, we obtain as a special case of \eqref{eq:tau-Beltrami-WAD}
\begin{equation}
    \tau_B (f; g,h)^{i} = \frac{1}{\sqrt{|h|}} \partial _\mu \qty(B^{\prime} (\|\dd f\|^2) h^{\mu\nu} \sqrt{|h|} \,\partial _\nu  f ^{i}),\qquad i\in[n].
\end{equation}
The corresponding gradient descent system 
\begin{equation}\label{eq:WAD-LB}
    u \in C^{\infty }(M_{T}) \qq{s.t.} 
    \begin{dcases}
        \partial_t u_t = \tau_B(u_t; g,h_{t}), \quad & \forall t \in (0,T), \\  
        h_t = \mc{O}(u_t) ,\quad & \forall t \in (0,T),\\ 
        u_0  = f. 
    \end{dcases}
\end{equation}
Further restricting $M$ to an open domain $\Omega$ yields again a system of the form \eqref{eq:WAD-system}, but with
\begin{equation}
\det V_{t} = B'(\|\dd u_{t}\|^{2}).
\end{equation}
Prominent special cases include \textbf{Perona-Malik denoising} 
\cite{peronaScalespaceEdgeDetection1990,Kichenassamy:2008aa} with
\begin{equation}\label{eq:PM-ht}
    h_t = \delta, \qquad  B^{\prime}(x) = e^{-x} \qq{or} B^{\prime} (x) = \frac{1}{1 + x}
\end{equation}
and \textbf{total variation denoising} \cite{rudinNonlinearTotalVariation1992,chambolleAlgorithmTotalVariation2004} with 
\begin{equation}
    h_t = \delta, \qquad  B(x)= 2\sqrt{x}. 
\end{equation}
This demonstrates the versatility of the Beltrami flow approach and its variants for representing a range of established methods of PDE-based image processing.

\vspace{0.5cm}
We conclude this section by pointing out two more useful properties of the Beltrami flow approach.
\begin{description}
\item[Reparametrization invariance.] The harmonic energy $E_{H}$ \eqref{eq:def-EH-general} is invariant under coordinate transformations, i.e.~under reparametrizations of the domain manifold $M$. This provides a basis for coordinate-invariant discretization of the PDE system $\partial_{t} u_{t} = \tau(u_t)$.  We refer, e.g., to \cite{Gu:2002aa,Wardetzky:2008aa} for related work.
\item[Conformal invariance.] Assume $m=\dim M = 2$. A \textit{conformal transformation} is a rescaling 
\begin{equation}\label{eq:conformal-transformation}
h\to \lambda h,\qquad
\lambda\colon M\to\R_{>0}
\end{equation}
of the metric $h$ on $M$ with respect to some positive function $\lambda$. 
Writing more explicitly $E_{H}(f) = E_{H}(f; g,h)$ for the harmonic energy of the map $f\colon M\to N$ between the Riemannian manifolds $(M,h)$ and $(N,g)$, \textit{conformal invariance} of $E_{H}$ means
\begin{equation}
E_{H}(f;g,h) = E_{H}(f;g,\lambda h),
\end{equation}
which in terms of the functional derivative translates to
\begin{equation}
\delta E_{H}(f;g,\lambda h) = \lambda\delta E_{H}(f;g,h).
\end{equation}
The consequence for the associated diffusion process is
    \begin{equation}
        \partial u_{t} =  \tau(u_{t};g,h) = \frac{1}{\lambda}\tau(u_{t};g,\lambda h),
    \end{equation}
which concerns discretization. Setting $\partial_{t} u\approx \frac{u_{t+\eta}-u_{t}}{\eta}$ with step size $0 < \eta \ll 1$, one has
\begin{equation}
\frac{u_{t+\eta}-u_{t}}{\eta} 
\approx \frac{1}{\lambda}\delta \tau(u_{t};g,\lambda h)
\qquad\gdw\qquad
u_{t+\eta}\approx u_{t} + \frac{\eta}{\lambda}\delta \tau(u_{t};g,\lambda h).
\end{equation}
This shows that the time scale $\eta$ used for discretization is entangled with the scale $\lambda$ of the metric $h$. Since $\lambda$ is a function varying over the domain $M$, this also introduces a \textit{spatially resolved} time scale for discretization. 
\end{description}

\section{Sigma Flow Model}\label{sec:Sigmaflow}
This section presents the main contribution of the paper, the \textit{sigma flow model} for labeling metric data on a smooth compact, oriented closed manifold $M$ equipped with a Riemannian metric $h$. This is achieved by combining the Beltrami flow and the assignment flow frameworks. Regarding image segmentation, the sigma flow model differs from the methodology presented in \cite{sochenGeneralFrameworkLow1998} in that it works for multiple classes and is an inherently geometric approach to data \textit{labeling}.

Section \ref{sec:Harmonic-Simplex} details the Beltrami flow approach for the specific choice 
\begin{equation}
(N,g)=(\S_c,\fr)
\end{equation}
as target manifold equipped with the Fisher-Rao metric and simplex-valued mappings 
\begin{equation}
\vP: M \to \S_c. 
\end{equation}
Section \ref{sec:Regularized-Sigmaflow} introduces the sigma flow model and shows that is constitutes a proper geometric diffusion approach. The extension of the sigma flow model from the metric connection to the $\alpha$-family of connections from information geometry is worked out in Section \ref{sec:alpha-SigmaFlow}. Finally, by additionally taking into account an entropic potential in Section \ref{sec:Entropic-Potential}, the sigma flow model becomes a proper labeling approach. 

This version of the novel sigma flow model for data labeling bears resemblance to basic models of mathematical physics (cf.~Remarks \ref{rem:theo-phys-sigma} and \ref{rem:potentials-in-physics}) and constitutes the natural geometric extension of the continuous-domain formulation of the assignment flow approach presented by \cite{Savarino:2021wt}.

\subsection{Harmonic Energy of Probability Simplex-Valued Mappings}\label{sec:Harmonic-Simplex}

Using the notation of Section \ref{sec:Riemann-Harmonics}, we consider the harmonic energy
\begin{equation}\label{eq:def-EH-simplex}
    E_H(\vP) = \frac{1}{2}\int _M \norm{\dd \vP}^{2}. 
\end{equation} 
The corresponding Beltrami flow is generated by the system
\begin{equation}\label{eq:Bf-simplex}
    \vS \in C^{\infty }(M_T,\S_c) \qq{s.t.} 
    \begin{dcases}
        \partial_t \vS_{t} = \tau(\vS_{t};\fr,h_{t}), \quad & \forall t > 0, \\  
        h_{t} = \mc{O}(\vS_{t}) ,\quad & \forall t > 0,\\ 
        \vS_{0}  = \mfk{P}, 
    \end{dcases}
\end{equation}
with a differential operator $\mc{O}$ specified later and initial condition $\mfk{P} \in C^{\infty}(M,\S_c)$. We use both local coordinates
\begin{equation}
    \label{eq:def-vt-vp}
\theta^{1},\dotsc,\theta^{\cc}
\qquad\text{and}\qquad
p_{1},\dotsc,p_{\cc} 
\end{equation}
of $\S_c$ introduced in Section \ref{sec:Hessian-Simplex}. The corresponding coordinate expressions of $\vS$ are denoted by 
\begin{equation}
\vartheta^{1},\dotsc,\vartheta^{\cc}
\qquad\text{and}\qquad
\vp_{1},\dotsc, \vp_{\cc}, 
\end{equation}
respectively.
\begin{proposition}[harmonic energy on $\S_c$]\label{prop:EH-simplex}
    For a smooth map $\vP : M \to \S_c$ with local $\theta$-coordinate functions $\vartheta^{i},i\in[\cc]$ and $p$-coordinate functions $\vp_i,\, i\in[\cc]$, the harmonic energy \eqref{eq:EH-simplex} evaluates to
    \begin{equation}
        \label{eq:EH-simplex}
        E_H(\vP) = \frac{1}{2} \int _M \fr_{ij} \la \dd \vartheta ^{i} , \dd  \vartheta^{j} \ra_{h} 
        =    \frac{1}{ 2}  \int_M  \la \dd \vp_i,  \dd  \vartheta ^{i} \ra_{h},  
    \end{equation}    
    with respect to a Riemannian metric $h$ on $M$. 
\end{proposition}
\begin{proof}
By \eqref{eq:fr_metric}, one has $\fr_{ij}=\pd_{j}\vp_{i}$ and hence by the chain rule
\begin{equation}\label{eq:proof-chain-rule-vp}
\partial_{\mu}\vp_{i} = \pd_{j}\vp_{i}\partial_{\mu}\vartheta^{j} = \fr_{ij}\partial_{\mu}\vartheta^{j}.
\end{equation}
    Thus, taking into account the symmetry of $\fr_{ij}$, \eqref{eq:def-EH-general} yields 
    \begin{subequations}
    \begin{align}
    E_{H}(\vP) &= \frac{1}{2}\int_{M} h^{\mu\nu}\partial_{\mu}\vartheta^{i}\partial_{\nu}\vartheta^{j}\fr_{ij}
    = \frac{1}{2}\int_{M} h^{\mu\nu}\partial_{\mu}\vartheta^{j}\partial_{\nu}\vartheta^{i}\fr_{ij}
    = \frac{1}{2}\int_{M} h^{\mu\nu}\partial_{\mu}\vp_{i}\partial_{\nu}\vartheta^{i}
    \\
    &= \frac{1}{2}\int_{M}\la\dd\vp_{i},\dd\vartheta^{i}\ra_{h}.
    \end{align}
    \end{subequations}
\end{proof}
We compute the key ingredient of the sigma flow, the tension field $\tau$ from Eq. \eqref{eq:tension_field}, in local coordinates.
\begin{proposition}[tension field in coordinates]
The tension field of $\vP : M \to \S_c$ with $\theta$ coordinate functions $\vt^{i},i\in[\cc]$ is given locally by
    \begin{equation}\label{eq:tension-coordinates}
        \tau(\vP; \fr,h)^{i}
        = \Delta _h  \vartheta^{i} + \frac{1}{2} \la \dd  (\vartheta^{i} - \Psi) , \dd (\vartheta^{i} - \Psi) \ra_h  - \frac{1}{2} \la \dd \Psi , \dd \Psi \ra_h,\qquad i\in[\cc],
    \end{equation}
    where $\Psi = \psi \circ \vP : M \to \R$ from \eqref{eq:def-psi-partition}. 
\end{proposition}
\begin{proof}
\eqref{eq:tension_field} reads
\begin{equation}\label{eq:proof-tension-locally}
\tau(\vP; \fr,h)^{i} = \Delta_h \vartheta^{i} + \Gamma^{i}_{jk}\partial_{\mu}\vartheta^{j}\partial_{\nu}\vartheta^{k} h^{\mu\nu}.
\end{equation}
Using the relation $\vp_{i}=\pd_{i}\psi$ by \eqref{eq:partial-psi} and hence by the chain rule
\begin{equation}
\partial_{\mu}\Psi = \pd_{i}\psi\partial_{\mu}\vartheta^{i} = \vp_{i}\partial_{\mu}\vartheta^{i},
\end{equation}
substitution of the Christoffel symbols \eqref{eq:gamma0} yields
\begin{subequations}
\begin{align}
\tau(\vP; \fr,h)^{i} 
&= \Delta_h \vartheta^{i} +
\frac{1}{2}(\delta^{i}_{j}\delta_{jk} - \delta^i_j \vp_k - \delta^i_k \vp_j) \partial_{\mu}\vartheta^{j}\partial_{\nu}\vartheta^{k} h^{\mu\nu}
\\
&= \Delta_h \vartheta^{i} + \frac{1}{2}(
\partial_{\mu}\vartheta^{i}\partial_{\nu}\vartheta^{i} 
- \vp_{k}\partial_{\mu}\vartheta^{i}\partial_{\nu}\vartheta^{k} - \vp_{j}\partial_{\mu}\vartheta^{j}\partial_{\nu}\vartheta^{i}) h^{\mu\nu}
\\
&= \Delta_h \vartheta^{i} + \frac{1}{2}(
\partial_{\mu}\vartheta^{i}\partial_{\nu}\vartheta^{i} 
- \partial_{\mu}\vartheta^{i}\partial_{\nu}\Psi 
- \partial_{\mu}\Psi\partial_{\nu}\vartheta^{i}
) h^{\mu\nu}
\\
&= \Delta_h \vartheta^{i} + \frac{1}{2}\big( \la \dd\vartheta^{i}, \dd\vartheta^{i}\ra_{h} - 2\la \dd\vartheta^{i}, \dd \Psi\ra_{h}\big)
\end{align}
\end{subequations}
which is equal to \eqref{eq:tension-coordinates}.
\end{proof}
Since mathematically equivalent expressions may behave differently when they are evaluated numerically, we derive another expression for the tension field in local coordinates.
\begin{proposition}[alternative form of the tension field]\label{prop:tau-vp-alternative}
The tension field of $\vP:M\to\S_c$ with $\theta$-coordinate functions $\vt^{i},i\in[\cc]$ and $p$-coordinate functions $\vp_i,i\in[\cc]$ is given locally, with respect to the $\theta$-coordinate system, by   
    \begin{equation}\label{eq:tau-vp-alternative}
        \tau(\vP; \fr,h)^{i} = \frac{1}{2}(\Delta _h \vt^{i} + \fr^{ij}\Delta _h \vp_j),\qquad i\in[\cc]. 
    \end{equation}
\end{proposition}

\begin{proof}
In this proof, we use explicitly the volume measure $\omega_h$ \eqref{eq:int-M}, which we denote by $\omega_h = \sqrt{ |h| } \dd x$.
We compute the first variation of the harmonic energy given by Proposition \ref{prop:EH-simplex}. For any smooth functions $f, g\colon M\to\R$, one has
\begin{subequations}\label{eq:laplace-adjoint}
\begin{align}
    \int_{M} \la df, dg\ra_{h} \,\omega _h 
    &= 
    \int_M \partial_\mu f \partial_\nu g h^{\mu \nu } 
    \sqrt{|h|} \dd x \\ 
    &= 
    - \int_M  f \, \partial_\mu( h^{\mu \nu } \sqrt{|h|} \partial_\nu g)
    \dd x\\ 
    &= 
    - \int_M  f \, \partial_\mu( h^{\mu \nu } \sqrt{|h|} \partial_\nu g)
    \frac{\sqrt{|h|} }{\sqrt{|h|} }
    \dd x\\ 
    &= 
    -\int_{M} (f \Delta_{h} g) \,\omega _h 
\end{align}
\end{subequations}
with 
\begin{equation}\label{eq:proof-Delta-h-expression}
    \Delta_{h} g = 
    \frac{1 }{\sqrt{|h|} }
    \partial_\mu( h^{\mu \nu } \sqrt{|h|} \partial_\nu g).
\end{equation}
Thus
\begin{subequations}
    \begin{align}
        \dv{t}  E_H(\vt + t \eta  ) \big|_{t=0}
        &= \dv{t} \frac{1}{2}\int _M \la \dd \vp_i(\vt +  t \eta) ,{\dd} (\vt^{i} + t \eta ^{i}) \ra_h \big|_{t=0}\\
        &= \frac{1}{2} \int _M \la \dd (\eta^{j} \pd_j \vp_i ) ,\dd \vt^{i} \ra_h  + \frac{1}{2} \int _M \la\dd \vp_i ,\dd \eta ^{i} \ra_h 
        \qquad(\text{using $\pd_{j}\vp_{i}=\fr_{ji}=\fr_{ij}$})
        \\ \label{eq:proof-dt-EH-alternative-c}
        &= 
        \frac{1}{2} \int _M \eta ^{j} \la \dd \fr_{ij} , \dd \vt^{i}\ra_h  
        + \frac{1}{2} \int _M \fr_{ij} \la \dd \eta ^{j} ,\dd \vt^{i} \ra_h 
        - \frac{1}{2} \int _M \eta ^{i} \Delta _h \vp_i 
    \end{align}
\end{subequations}
We rewrite the second integral using partial integration.
\begin{subequations}
    \begin{align}
        \int _M \fr_{ij} \la \dd \eta ^{j} ,\dd \vt^{i} \ra_h         
        \overset{\eqref{eq:int-M}}{=}  
        \int _M \fr_{ij} \la \dd \eta ^{j} ,\dd \vt^{i} \ra_h \omega_h
        & = 
        \int _M  \fr_{ij} \partial _\mu \eta ^{j}  \partial _\nu  \vt^{i} h^{\mu \nu } \sqrt{|h|} \dd x    \\ 
        &= 
        -\int _M  \eta ^{j}  \partial _\mu \fr_{ij} \partial _\nu  \vt^{i} h^{\mu \nu } \sqrt{|h|} \dd x    
        -
        \int _M  \eta ^{j}  \fr_{ij} \partial _\mu ( h^{\mu \nu } \sqrt{|h|} \partial _\nu  \vt^{i}) \dd x 
        \\ 
        &\overset{\eqref{eq:proof-Delta-h-expression}}{=} 
        -\int _M  \eta ^{j}  \la \dd \fr_{ij} , \dd  \vt^{i} \ra_h \sqrt{|h|} \dd x    
        -\int _M  \eta ^{j}  \fr_{ij} \Delta_h \vt^{i} \sqrt{|h|} \dd x.
    \end{align}
\end{subequations}
Substitution in \eqref{eq:proof-dt-EH-alternative-c} yields
\begin{subequations}
    \begin{align}
    \dv{t} & E_H(\vt + t \eta  ) \big|_{t=0} \\
        &= 
        \frac{1}{2} \int _M \eta ^{j}  \la \dd \fr_{ij} ,\dd \vt^{i} \ra_h
        - \frac{1}{2} \int _M \eta ^{j}  \la \dd \fr_{ij},  \dd \vt^{i} \ra_h  
        - \frac{1}{2} \int _M \fr_{ij} \eta ^{j} \Delta _h \vt^{i}  
        - \frac{1}{2} \int _M \eta ^{i} \Delta _h \vp_i \\
        &=
        - \frac{1}{2} \int _M \fr_{ij} \eta ^{j} \Delta _h \vt^{i}  
        - \frac{1}{2} \int _M \eta ^{j} \Delta _h \vp_j 
        = -\frac{1}{2}\int_{M} \fr_{ij}(\Delta_{h}\vartheta^{i}+\fr^{ij}\Delta_{h}\vp_{j}) \eta^{j} \\
        &= 
        -\int _M \fr_{ij} \tau(\vP; \fr,h)^{i} \eta^{j},
    \end{align}
\end{subequations}
which proves \eqref{eq:tau-vp-alternative}.
\end{proof}

\begin{remark}[Fisher-Rao metric becomes singular]
We note that the Fisher-Rao metric
\begin{equation}\label{eq:def-fr-ij}
    \fr_{ij} = p_ i \delta _{ij} - p_i p_j,\qquad i,j\in[\cc]
    \end{equation}
converges to a singular matrix along paths approaching the boundary of $\S_c \subset \triangle_c$. 
As a remedy, in the rest of this section, we use a regularized metric $\fr_{\varepsilon}$ given in the $\theta$-coordinate system by
\begin{equation}\label{eq:def-fr-eps}
(\fr_{\varepsilon} )_{ij} = \fr_{ij} + \varepsilon \delta_{ij}, \qquad  i,j\in[\cc]
\end{equation}
that is bounded from below by $\varepsilon \II$. The derivatives of the metric are preserved, however. 
\end{remark}
\begin{proposition}[Christoffel symbols of $\fr_{\varepsilon}$]
    \label{prop:Gamma-Eps}
    The Christoffel symbols of $\fr_{\varepsilon }$, denoted by $\Gamma _\varepsilon $, are given by 
    \begin{subequations}\label{eq:def-Gamma-veps}
    \begin{align}
        (\Gamma _\varepsilon )_{ijk} &= \Gamma _{ijk} =\frac{1}{2} \pd_i \pd_j \pd_k \psi 
    \intertext{and}
        (\Gamma _\varepsilon )^{i}_{jk} &= (\fr_\varepsilon )^{il} \Gamma _{ljk},\qquad
        i,j,k\in[\cc].
        \end{align}
    \end{subequations}
\end{proposition}
\begin{proof}
    The relation $(\Gamma _\varepsilon )_{ijk} = \Gamma _{ijk} = \frac{1}{2} \pd_i \pd_j \pd_k \psi$ follows by the definition of the connection and $\pd_i (\fr_{\varepsilon})_{jk} = \pd_i \fr _{jk}$:
    \begin{equation}
        (\Gamma _\varepsilon )_{ijk}  
        =
        \frac{1}{2}
        (
            \pd_k (\fr_{\varepsilon } )_{ij} 
            +
            \pd_j (\fr_{\varepsilon } )_{ki}
            - 
            \pd_i (\fr_{\varepsilon } )_{jk}
        )
        =
        \frac{1}{2}
        (
            \pd_k \fr_{ij} 
            +
            \pd_j \fr_{ki}
            - 
            \pd_i \fr_{jk}
        )
        = \Gamma _{ijk}.
    \end{equation}
    Further noting \cite[Section 3.3]{amariMethodsInformationGeometry2000}
    \begin{equation}
        \pd_i \fr _{jk}
        =
        \pd_i \pd_j \pd_k \psi 
    \end{equation}
    is permutation invariant shows $\Gamma _{ijk} =\frac{1}{2} \pd_i \pd_j \pd_k \psi$. The second claim is the defining relation of $(\Gamma_{\varepsilon } )^{i}_{jk} $.
\end{proof}

\subsection{Sigma Flows}\label{sec:Regularized-Sigmaflow}

The following definition introduces our new model, defined as a Beltrami flow with dynamic metric $h$ and target manifold $(\S_c,\fr_{\veps})$.
\begin{definition}[sigma flow]
    Let $\mfk{P} \in C^{\infty }(M,\S_c)$ and $T> 0 ,\varepsilon > 0$ be fixed. 
    The \textit{sigma flow} is the system of PDEs
    \begin{equation}
        \label{eq:sigmaflow}
        \tag{$\sigma$-flow}
        \vS \in C^{\infty }(M_T,\S_c) \qq{s.t.} 
        \begin{dcases}
            \partial_t \vS_{t} = \tau(\vS_{t}; \fr_{\veps},h_{t}), \quad & \forall t \in (0,T), \\
            h_{t} = \mc{O}(\vS_{t}) ,\quad & \forall t \in (0,T),\\ 
            \vS_0  = \mfk{P},
        \end{dcases}
    \end{equation}
    where
    \begin{equation}
        \mc{O} : C^{\infty }(M,\S_c) \to \Gamma(T ^{*}M \otimes T^* M)
    \end{equation}
    maps to the set of positive definite symmetric 2-tensors such that, for all $\mfk P \in C^{\infty}(M,\S_c)$ and $x \in M$,
    \begin{equation}\label{eq:ass-uniform-boundedness}
        \mc{O}(\mfk P)(x) \succeq C(\mc{O}) \, \II_{m}, \quad\text{for some}\; C(\mc{O}) > 0.
    \end{equation}
\end{definition}    
The condition on $\mc{O}$ is called \textit{uniform positive definiteness criterion} in \cite{Weickert:1998aa}. It is trivially satisfied by choosing a fixed metric $h$ independent of the state $\vS$.  
\\
In the following, we denote the coordinate expressions of $\vS$ by $\vt$ and $\vp$ respectively. To avoid cluttered formulae, we do not indicate the time dependence $t\mapsto \vt_t,t\mapsto \vp_t$.  
\begin{proposition}[sigma flow in coordinates]\label{prop:sigmaflow-coordinates}
Set $\veps=0$ in \eqref{eq:def-fr-eps}. 
    If $\vt, \vp$ denote the coordinate expressions of $\vS$ in $\theta$ and $p$ coordinates respectively, the first equation of the system \eqref{eq:sigmaflow} takes the form 
    \begin{subequations}
    \begin{align}
        \partial _t \vt^{i} &= \frac{1}{2}(\Delta_{h_{t}} \vt^{i} + \fr^{ij} \Delta_{h_{t}} \vp_j),\qquad i\in[\cc],
    \intertext{and}
        \partial _t \vp_i &= \frac{1}{2}(\fr_{ij} \Delta_{h_{t}} \vt ^{j} + \Delta_{h_{t}} \vp_i),\qquad i\in[\cc].
        \end{align}
    \end{subequations}
\end{proposition}
\begin{proof}
The first form follows directly from Proposition \ref{prop:tau-vp-alternative}. It yields the second form after applying equation \eqref{eq:proof-chain-rule-vp} with $\partial_{t}$ in place of $\partial_{\mu}$.
\end{proof}
Our next goal is to devise a Lyapunov functional for the sigma flow, after two preparatory Lemmata; see Proposition \ref{prop:lyapunov} below.
\begin{lemma}[spectrum of Laplace-Beltrami operator]\label{lem:LB-spectrum}
    For any metric $h$ on $M$, the Laplace-Beltrami operator $\Delta _h$ is diagonalizable. The eigenfunctions of $\Delta _h$, $\phi _n \in C^{\infty}(M), n \in \N_{0}$ exist and form an orthonormal Hilbert basis of $L^{2}(M)$. 
    Furthermore, let $\lambda _n, n \in \N$ denote the eigenvalues of $\Delta _h$, i.e. 
\begin{equation}\label{eq:Delta-h-spectrum}
\Delta_{h}\phi_{n} = \lambda_{n}\phi_{n},\qquad n\in\N_{0}.
\end{equation}
    Then $\lambda_0 = 0$ and $\lambda_i < 0$ for all $i > 0$. 
\end{lemma}

\begin{proof}
    See \cite[Thm.1]{chavelEigenvaluesRiemannianGeometry1984}.
\end{proof}
\begin{lemma}[upper/lower uniform boundedness]
    \label{lem:b}
    The mapping $\mathcal{B} : \S_c \to \R^{\cc \times \cc}$ defined in the $\theta $-coordinate system by 
    \begin{equation}\label{eq:def-mcB-theta}
        \mathcal{B}_{ij}(\theta) = \fr_{ij}(\theta) + \frac{1}{2} \pd_k \fr_{ij}(\theta) \theta ^{k},\qquad i,j\in[\cc]
    \end{equation}
    maps into the set of symmetric matrices and admits the bounds
    \begin{equation}\label{eq:bounds-mcB}
    \forall\theta,\qquad
        c_{2} \II \succeq \mathcal{B}(\theta) \succeq c_{1}\II,\qquad
        c_{1}=-\frac{c ^{2} -1}{2 e},\quad
        c_{2}=\frac{1}{2}\Big(1+\frac{c^{2} -1}{e}\Big),
    \end{equation}
    where $c_1, c_2>0$ only depend on $\dim \S_c = \cc = c-1$.  
\end{lemma}
\begin{proof}
In the proof we suppress the dependence on $\theta$ for all quantities.
Eq. \eqref{eq:fr_metric} yields
\begin{equation}
    \pd_k \fr_{ij} = \pd_{k}\pd_{i}\pd_{j} \psi         
\end{equation}
and implies the symmetry of $\mc{B}$ given by \eqref{eq:def-mcB-theta}. In the context of Hessian geometry, this relation is referred to as the \textit{Codazzi equation} \cite[Prop.~2.1]{shimaGeometryHessianStructures2007}. 
To establish the convexity bounds of $\mc B$, we compute first its entries. Introducing the notation
\begin{equation}\label{eq:proof-def-Theta-gamma}
    \gamma \coloneqq \max (0, \theta^1,\dots ,\theta^\cc),
\end{equation}
we define 
\begin{equation}\label{eq:proof-def-tilde-theta}
    \tilde{\theta}^{i} \coloneqq \theta^{i} - \gamma \leq 0 
    \qq{such that}
    \theta^{i} = \gamma + \tilde{\theta}^{i}.
\end{equation}
Recall from \eqref{eq:partial-psi} and \eqref{eq:def-Sc} the relations  
\begin{equation}
    \label{eq:proof-mcb-p}
    p_i = e^{\theta^i -\psi}, \; i\in[\cc] 
    \qquad\text{and}\qquad
    p_0 := 1 - \sum_{j \in [\cc]} p_j = e^{-\psi} 
\end{equation}
where by \eqref{eq:def-psi-partition} 
\begin{equation}\label{eq:proof-psi-logexp}
\psi = \logexp(\theta) = \log\bigg(1+\sum_{j\in[\cc]}e^{\theta^{j}}\bigg).
\end{equation}
Using $\gamma\geq 0$ and $\tilde\theta^{i},\; i\in[\cc]$ as defined by \eqref{eq:proof-def-Theta-gamma} and \eqref{eq:proof-def-tilde-theta}, we rewrite 
\begin{subequations}    
    \begin{align}
        \psi = \log\bigg(1+\sum_{j\in[\cc]}e^{\theta^{j}}\bigg) 
        &= 
        \log\bigg(e^{\gamma}(e^{-\gamma}+\sum_{j\in[\cc]}e^{\theta^{j}-\gamma})\bigg) \\
        &= 
        \gamma 
        +
        \log\bigg(e^{-\gamma} + \sum_{j \in [\cc]} e^{\tth^{i}}\bigg)
    \end{align}
\end{subequations}
and note that the set $\{ -\gamma,\tth^{1} , \dots, \tth^{\cc} \}$ contain only non-positive elements and that at least one of them must be 0 by \eqref{eq:proof-def-Theta-gamma}, \eqref{eq:proof-def-tilde-theta}. Thus $e^{\tilde{\theta}^{i}}\in (0,1),\;i\in[\cc]$ and $e^{-\gamma} \in (0,1)$, and at least one of them must be 1 such that
\begin{equation}\label{eq:proof-psi-bracket-bounds}
    \log\bigg(e^{-\gamma} + \sum_{j \in [\cc]} e^{\tth^{i}}\bigg) \in (0, \log c). 
\end{equation}
Consequently, we can write 
\begin{equation}\label{eq:proof-psi-gamma-xi}
    \psi = \gamma + \zeta \qq{for some} \zeta \in (0,\log c) 
\end{equation}
with
\begin{equation}
\zeta\to 0 \qquad\text{as}\qquad \gamma\to\infty,
\end{equation}
since then $\psi$ given by \eqref{eq:proof-psi-logexp} approaches the maximal component of the argument vector  \eqref{eq:proof-def-Theta-gamma}. We finally define the function
\begin{equation}
    \label{eq:proof-mcb-xexp}
    \mathrm{xexp} : \R_{>0} \to \R_{>0}, \qquad 
    x \mapsto x e^{-x} \;\in\; \Big(0,\frac{1}{e}\Big).
\end{equation} 
Now, rewriting the equation \eqref{eq:def-mcB-theta} defining $\mc{B}$ in the form
\begin{equation}\label{eq:proof-mcB-theta-mcA}
    \mc B = \fr + \frac{1}{2}\mc A, 
\end{equation}
we have
\begin{subequations}
\begin{align}\label{eq:proof-mcA-ij}
\mc{A}_{ij} 
&= \pd_k \fr_{ij} \theta ^{k}
= \theta_i \delta_{ij} - \theta_i p_{j} - \theta_j p_{i},
\intertext{with} \label{eq:proof-mcA-theta-i}
\theta_i 
&= \fr_{ij} \theta^j 
\overset{\eqref{eq:proof-def-tilde-theta}}{=} 
\gamma\sum_{j\in[\cc]} \fr_{ij} + \fr_{ij} \tth^j.
\end{align}
\end{subequations}
Regarding the first term, we compute
\begin{equation}
    \gamma \sum_{j\in[\cc]} \fr_{ij}
    = 
    \gamma 
    \sum_{j\in[\cc]} p_i(\delta_{ij} - p_j)
    = \gamma  p_i \Big(1 - \sum_{j\in[\cc]} p_j\Big) 
    = \gamma p_i p_0.
\end{equation}
Invoking the relations \eqref{eq:proof-def-tilde-theta}, \eqref{eq:proof-mcb-p} and \eqref{eq:proof-psi-gamma-xi}, we have
\begin{equation}\label{eq:proof-mcB-pi-expression}
    p_i = e^{\theta^{i} - \psi(\theta)}
    = 
    e^{\tth^i} e^{-\zeta}
    \qquad\text{and}\qquad
    p_0 = e^{-\psi(\theta)} = e^{-\gamma -\zeta} 
\end{equation}
and thus obtain for the first sum on the right-hand side of \eqref{eq:proof-mcA-theta-i}
\begin{equation}
    \gamma \sum_{j\in[\cc]} \fr_{ij} = 
    \gamma  p_i p_0 
    = \gamma e^{-\gamma} e^{-2\zeta } e^{\tilde \theta^{i}}
    \overset{\eqref{eq:proof-mcb-xexp}}{=} \mathrm{xexp}(\gamma) e^{-2\zeta } e^{\tilde \theta^{i}} \in\Big(0,\frac{1}{e}\Big),
\end{equation}
where the bounds follow from the bounds of \eqref{eq:proof-psi-bracket-bounds}, \eqref{eq:proof-mcb-xexp} and $e^{\tilde{\theta}^{i}}\in(0,1),\;i\in[\cc]$.

As for the second term on the right-hand side of \eqref{eq:proof-mcA-theta-i}, we have
\begin{subequations}
\begin{align} 
    \fr_{ij} \tilde \theta^{j} 
    &\overset{\eqref{eq:def-fr-ij}}{=}
    \tilde \theta^{j} p_i(\delta_{ij} - p_j)
    \overset{\eqref{eq:proof-mcB-pi-expression}}{=}
    \tilde \theta^{i} e^{\tilde \theta^{i} - \zeta } 
    - \sum_{j\in[\cc]} \tilde \theta^{j} e^{\tilde \theta^{j} - \zeta}  e^{\tilde \theta ^{i} - \zeta}
    \\
    &\overset{\tilde\theta^{i}\leq 0}{=}
    -\mathrm{xexp}( | \tth^i |) e^{-\zeta}
    + 
    \sum_{j\in[\cc]} \mathrm{xexp}(| \tth^j|)  e^{\tilde \theta ^{i}} e^{- 2\zeta}
    \quad
    \in
    \Big(-\frac{1}{e},\frac{\cc}{ e}\Big),
    \end{align}
\end{subequations}
Putting together, we obtain for \eqref{eq:proof-mcA-theta-i}
\begin{equation}
\theta_i = \gamma \sum_j \fr_{ij}  +  \tilde \theta^{j} \fr_{ij} \in 
\Big(-\frac{1}{e},\frac{c}{e}\Big), \quad c = \cc + 1.
\end{equation}
From $p_i \in (0,1)$ follows then
\begin{equation}
\mc{A}_{ij} 
\overset{\eqref{eq:proof-mcA-ij}}{=} 
\theta_i \delta_{ij} - \theta_i p_{j} - \theta_j p_{i}
\in
    \Big(-\frac{c+1}{e},\frac{c+1}{e}\Big)
    =: (-q, q),\qquad q = \frac{c+1}{e}
\end{equation}
which implies $\|\mc{A}\|_{\infty}\leq \cc q$ and in turn the bounds of the spectral radius of $\mc{A}$
\begin{equation}
\rho\big(\mc{A}\big) \leq \|\mc{A}\|_{\infty} \leq \cc q
\qquad\implies\qquad
-\cc q \II \preceq \mc{A} \preceq \cc q \II.
\end{equation}

\vspace{0.2cm}
Returning to \eqref{eq:proof-mcB-theta-mcA}, we infer
\begin{equation}
\forall v\in\R^{\cc},\qquad
\mc{B}_{ij} v^{i}v^{j}
= \fr_{ij} v^{i}v^{j} + \frac{1}{2}\mc{A}_{ij} v^{i}v^{j}
\overset{\fr\succeq 0}{\geq} \frac{1}{2}\mc{A}_{ij} v^{i}v^{j} \geq -\frac{\cc q}{2}\|v\|^{2}
\end{equation}
which gives the lower bound
\begin{equation}
\mc{B} \succeq -\frac{\cc q}{2} \II.
\end{equation}
As for an upper bound, we first note that with $p_{i}\in(0,1),\; i\in[\cc]$,
\begin{subequations}
\begin{align}
\fr_{ii} &= p_{i}-p_{i}^{2} 
= p_{i}(1-p_{i}) \in\Big(0,\frac{1}{4}\Big),\qquad i\in[\cc],
\\
\sum_{j\in[\cc]\setminus\{i\}}|\fr_{ij}|
&= p_{i}\sum_{j\in[\cc]\setminus\{i\}}p_{j}
= p_{i} (1-p_{0}-p_{i})
\leq p_{i}(1-p_{i}) \in\Big(0,\frac{1}{4}\Big),\qquad i\in[\cc]
\end{align}
\end{subequations}
such that by Gerschgorin's circle theorem \cite[Thm.~6.1.1]{Horn:2013aa}
\begin{equation}
\rho\big(\fr\big) \in \Big(0,\frac{1}{2}\Big).
\end{equation}
Consequently,
\begin{equation}
\mc{B}_{ij} v^{i} v^{j}
\leq \frac{1}{2}\|v\|^{2} + \frac{1}{2}\mc{A}_{ij} v^{i} v^{j}
\leq \frac{1}{2}(1 + \cc q) \|v\|^{2},
\end{equation}
that is
\begin{equation}
\mc{B} \preceq \frac{1}{2}(1 + \cc q) \II
\end{equation}
which together with
\begin{equation}
    \cc q = (c-1)\frac{c+1}{e} = \frac{c^{2}-1}{e}
\end{equation}
concludes the proof.
\end{proof}
\vspace{0.5cm}
\noindent
We are now in a position to devise a Lyapunov functional for the sigma flow.
\begin{proposition}[Lyapunov functional]
\label{prop:lyapunov}
Let $\HH\colon\S_{C}\to \R$ be given by \eqref{eq:HH} and $\varepsilon > 0$ be fixed. Then the functional 
\begin{equation}\label{eq:def-lyapunov}
    \Phi : C^{\infty }(M,\S_c) \to  \mathbb{R}_{\geq}, 
    \qquad  
    \Phi (\mfk P) = \int_M \big(\HH(\vP) + \frac{\varepsilon}{2} \delta _{ij} \vt^{i} \vt^{j} \big)
    \quad 
    \vt = \chi_{e} \circ \vP : M \to \R^{\cc}
\end{equation}
is a Lyapunov functional for the sigma flow \eqref{eq:sigmaflow} if 
\begin{equation}\label{eq:Lyapunov-def-beta}
    \varepsilon  + c_{1} = \beta  > 0,\qquad  c_{1}=-\frac{c ^{2} -1}{2 e}, 
\end{equation}   
with $c_{1}$ from \eqref{eq:bounds-mcB}.
\end{proposition}

\begin{proof}
Let $\vS\in C^{\infty}(M_T,\S_c)$ solve the sigma flow system \eqref{eq:sigmaflow} for $T>0$ fixed and let $\vt_t$ be the $\theta$-coordinate functions of $\vS_t$ and $h_t = \mc{O}(\vS_t)$ the time dependent metric. In order to show that $\Phi$ is a Lyapunov functional, we show that $t \mapsto \Phi(\vS_t)$ is bounded from below, continuous, differentiable and monotonically decreasing in time. 
    
Due to convexity of the integrand function ($\HH$ is the negative entropy), the functional is bounded from below by 
\begin{equation}
\Phi(0) = \int_M \HH(0) \overset{\eqref{eq:HH}}{=} - \int_M \psi(0) \overset{\eqref{eq:def-psi-partition}}{=} - \mathrm{vol}(M)\log c,
\end{equation}
where $0 \in C^{\infty}(M,\S_c)$ is the $0$ function in the $\theta$-coordinate system. As for the continuity at $0$, we have
\begin{subequations}
\begin{align}
\big|\Phi(\vS_t) &- \Phi(\vS_0)\big|
        \leq 
        \mathrm{vol}(M) \max_{x\in M}
        \Big(
            |\HH(\vt_t(x)) - \HH(\vt_0(x)) |
            + 
            | \delta_{ij} \vt_{t}^{i}(x) \vt_{t}^{j}(x) -  \delta_{ij} \vt_{0}^{i}(x) \vt_{0}^{j}(x) | 
        \Big)
        \\
        &\leq \mathrm{vol}(M) \max_{x\in M}
        \Big(
            |\HH(\vt_t(x)) - \HH(\vt_0(x)) |
            + \big(\|\vt_{t}(x)\|+\|\vt_{0}(x)\|\big)\big| 
            \|\vt_{t}(x)\|-\|\vt_{0}(x)\|\big|
           \Big).
\end{align}
\end{subequations}
Since $\vphi$ is continuous and $M$ is compact, the right-hand side goes to $0$ as $\vt_{t}\to\vt_{0}$, which shows the continuity of $\Phi$ at $0$. Further, note that the integrand of $\Phi(\vS_t)$ is continuously differentiable in time. The compactness of $M$ then implies that $\Phi(\vS_t)$ is differentiable for all $t \in (0,T)$. 

We do not indicate the time dependence of quantities in the rest of this proof to alleviate notation. We show now that $\Phi$ is monotonically decreasing in time.
\begin{subequations}
\begin{align}
    \dv{t} \Phi(\vS)
    &= \int _M \pd_i \HH \pdv{t} \vt^{i}+ \varepsilon \delta _{ij} \vt^{j} \pdv{t} \vt^{i}\\ 
    \label{eq:dt-Phi-proof}
    &\overset{\substack{\eqref{eq:HH} \\ \eqref{eq:partial-vphi}}}{=} \int _M (\fr_{ij} + \varepsilon \delta _{ij}) \vt ^{j} \pdv{t} \vt ^{i} 
    = 
    \int _M (\fr_{\varepsilon } )_{ij} \vt ^{j} \pdv{t} \vt ^{i} 
    \intertext{and using $\partial_{t}\vt^{i}=\tau(\vS;\fr_{\varepsilon},h)^{i} = \Delta_h \vt^i + (\Gamma_{\varepsilon } )^i_{jk}\Braket{\dd \vt^j , \dd \vt^k}_h$ by \eqref{eq:proof-tension-locally} (with the metric $\fr_{\veps}$ in place of $\fr$)
    }
    &= \int_M (\fr_{\varepsilon })_{ij} \vt^{j} \Delta _h \vt ^{i} +  \int _M (\fr_{\varepsilon} )_{ij} (\Gamma_{\varepsilon})^{i}_{kl} \la \dd \vt^{k}, \dd \vt ^{l} \ra_ h \vt^{j}.  
    \intertext{
            By Prop.~\ref{prop:Gamma-Eps} we have $(\fr_{\varepsilon} )_{ij} (\Gamma_{\varepsilon})^{i}_{kl} = \Gamma_{klj}$ and $\Gamma_{klj} = \Gamma_{jkl}$ for any order of the indices. Thus
    } \label{eq:dt-Phi-b-proof}
    \dv{t} \Phi(\vS) &= 
    \int_M (\fr_{\varepsilon})_{ij} \vt^{j} \Delta _h \vt ^{i} +  \int _M \Gamma_{ijk} \vt^{i} \la \dd \vt^{j} ,\dd \vt ^{k}\ra_h.  
\end{align}
\end{subequations}
Regarding the first integral on the right-hand side, we apply partial integration and use again the volume measure $\omega_h = \sqrt{ |h| } \dd x$ explicitly
\begin{subequations}
\begin{align}
    \int_M (\fr_{\varepsilon})_{ij} \vt^{j} \Delta _h \vt ^{i} 
    &\overset{\eqref{eq:int-M}}{=}
    \int_M (\fr_{\varepsilon})_{ij} \vt^{i} \Delta _h \vt ^{j} \sqrt{ |h| } \dd x
    \overset{\eqref{eq:delta-EH}}{=} \int_M (\fr_\varepsilon )_{ij} \vt^i \partial_\nu \qty(h^{\mu \nu} \sqrt{|h|}\partial_\mu \vt^j) \dd x \\
    &= 
    -\int_M 
    \partial_\nu (\fr_\varepsilon )_{ij} \vt^i  h^{\mu \nu} \partial_\mu \vt^j \sqrt{|h|}\dd x
    - \int_ M 
    (\fr_\varepsilon )_{ij} \partial_\nu  \vt^i  h^{\mu \nu} \partial_\mu \vt^j \sqrt{|h|} \dd x. 
\end{align}
\end{subequations}
Taking into account the chain rule $\partial_{\nu}(\fr_{\veps})_{ij} = \partial_\nu \fr_{ij} = \pd_k \fr_{ij} \partial_\nu\vt^k$, we obtain
\begin{subequations}
\begin{align}
\dv{t} \Phi(\vS)
&= - \int_M \pd_k \fr_{ij} \la \dd \vt^{j} , \dd \vt^{k} \ra_h \vt^{i}
- \int_M (\fr_\varepsilon)_{ij} \la \dd \vt^{j} , \dd \vt ^{i} \ra_h
+ \int _M \Gamma_{ijk} \vt^{i} \la \dd \vt^{j} , \dd \vt ^{k} \ra_h
\intertext{and using the relation $\pd_k \fr_{ij} = 2\Gamma_{kij} = 2\Gamma_{ijk}$}
\dv{t} \Phi(\vS) 
&= - \int_M (\fr_\varepsilon)_{ij} \la \dd \vt^{j} , \dd \vt ^{i} \ra_h
- \frac{1}{2} \int_M \pd_k \fr_{ij} \la \dd \vt^{j} , \dd \vt^{k} \ra_h \vt^{i}
\\
&= -\int_{M}\Big(
\fr_{ij}+\veps\delta_{ij} + \frac{1}{2}\vt^{k}\pd_k \fr_{ij}\Big) \la \dd \vt^{i} , \dd \vt ^{j} \ra_h
\\
\label{eq:dt-Phi-proof-B}
&\overset{\eqref{eq:def-mcB-theta}}{=}
-\int_{M}\big(\mc{B}(\vt)+\veps\II\big)_{ij} \la \dd \vt^{i} , \dd \vt ^{j} \ra_h.
    \end{align}
\end{subequations}
The integrand has the form $\tr(AB)$ where $A,B$ are symmetric positive semi-definite matrices ($\varepsilon$ was chosen so that this is true). Invoking the lower bound $\lambda_{\min}(B) \tr A \leq \tr(AB)$ implied by a trace inequality \cite[p.~341,~H.1.h]{Marshall:2011vn} and taking into account \eqref{eq:Lyapunov-def-beta} gives
\begin{equation}\label{eq:proof-dtPhiqt}
    \dv{t} \Phi(\vS) \leq -\beta  \int_M \delta _{ij}\la \dd \vt^{i} , \dd \vt^{j} \ra_h.
\end{equation}
By virtue of Lemma \ref{lem:LB-spectrum} and \eqref{eq:Delta-h-spectrum}, we expand the functions $\vt^{i}$ in the orthonormal basis $(\phi_{n})_{n\in\N_{0}}$ provided by the Laplacian $\Delta_{h}$,
\begin{equation}\label{eq:vti-Laplacian-decomposition}
\vt^{i} = \sum_{n\geq 0} a^{i}_{n}\phi_{n},\qquad
a^{i}_{n} = \int_{M}\vt^{i}\phi_{n}
\end{equation}
to obtain
\begin{subequations}\label{eq:dt-Phi-proof-laplacian-adjoint}
\begin{align}
\int_M \la \dd \vt^{i} , \dd \vt^{j} \ra_h
&\overset{\eqref{eq:laplace-adjoint}}{=} -\int_{M}\vt^{i}\Delta_{h}\vt^{j}
= -\sum_{n_{1},n_{2}\geq 0} a_{n_{1}}^{i} a_{n_{2}}^{j}\lambda_{n_{2}}\underbrace{\int_{M}\phi_{n_{1}}\phi_{n_{2}}}_{=\delta_{n_{1} n_{2}}}
\\
&= -\sum_{n\geq 0} a^{i}_{n} a^{j}_{n}\lambda_{n}
\overset{\lambda_{0}=0}{=} -\sum_{n\geq 1} a^{i}_{n} a^{j}_{n}\lambda_{n}.
\end{align}
\end{subequations}
Returning to \eqref{eq:proof-dtPhiqt}, we thus have
\begin{subequations}
\begin{align}\label{eq:proof-dotPhi}
\dv{t} \Phi(\vS) 
&\leq \beta\sum_{n\geq 1}(\delta_{ij} a^{i}_{n} a^{j}_{n})\lambda_{n}
 = \beta\sum_{n\geq 1}\|a_{n}\|^{2}\lambda_{n}
 \intertext{with}
a_{n} &= (a_{n}^{1},\dotsc, a_{n}^{\cc})^{\T}.
\end{align}
\end{subequations}
The eigenvalues $\lambda_n$ of the Laplacian $\Delta_{h}$ depend on the state $\vS$ via the coupling $h = \mc{O}(\vS)$. We can however bound the metric uniformly from below by $C(\mc{O})\II, c > 0$ due to \eqref{eq:ass-uniform-boundedness}. This allows to give a uniform bound on $\lambda_n$ as follows.
For any $n \geq 1$, invoke the identity \cite[Eq.~(46)]{chavelEigenvaluesRiemannianGeometry1984}
\begin{equation}
    \abs{\lambda_n} = \int_M h^{\mu \nu} \partial_\mu \phi_n \partial_\nu \phi_n
\end{equation}  
with $\int_M \phi_n = 0$, which follows from the fact that eigenfunctions of the Laplacian are either constant or have mean 0 and the constant eigenfunctions are associated with $\lambda_0 = 0$. Since we know that $h^{-1} \succeq \frac{1}{C(\mc{O})} \II$ we can apply a trace inequality as in \eqref{eq:proof-dtPhiqt} 
\begin{equation}
    \abs{\lambda_n} \geq \frac{1}{C(\mc{O})} \int_M \delta^{\mu \nu} \partial_\mu \phi_n 
    \partial_\nu \phi_n = \frac{1}{C(\mc{O})} \norm{\dd{\phi_{n}}}_{\delta}^{2}.
\end{equation}   
Then, the Poincaré lemma \cite[Cor.~A.1.1]{jostRiemannianGeometryGeometric2017} guarantees the existence of a constant $C(M) > 0$ such that
\begin{equation}
    \norm{\phi_n}_{L^{2}(M)} \leq C(M) \norm{\dd \phi_n}_{\delta}
\end{equation}
which implies
\begin{equation}
    \abs{\lambda_n} 
    \geq 
    \frac{1}{C(\mc{O})} \norm{\dd \phi_n}_{\delta}^{2}
    \geq 
    \frac{1}{C(M)^{2} C(\mc{O})}
    \norm{\phi_n}^{2}_{L^{2}(M)}. 
\end{equation}
By normalization of the eigenfunctions $\norm{\phi_n}_{L^{2}(M)} = 1$ we 
we thus obtain $\lambda_{n}< -\frac{1}{C(M) C(\mc{O})}$. From \eqref{eq:proof-dotPhi}, we finally infer
\begin{equation}
    \dv{t} \Phi(\vq) \leq 
    -\frac{\beta}{C(M)^{2} C(\mc{O})}
    \sum_{n\geq 1}\|a_{n}\|^{2}. 
\end{equation}
Because $\|a_{n}\|$ measures the non-constant part of $\vq$ if $n>0$, it follows that $\Phi(\vS)$ monotonically decays as long as $\sum_{n\geq 1}\|a_{n}\|^{2} > 0$, and $\sum_{n\geq 1}\|a_{n}\|^{2} = 0$ holds if and only if $\vS$ is constant.
\end{proof}
\begin{remark}[sigma flow: existence and convergence]\label{rem:existence}
Proposition \ref{prop:lyapunov} shows that the convex functional $\Phi$ given by \eqref{eq:def-lyapunov} is monotonically decreasing as long as $\vq_{t}$ is not constant. After discretizing the domain manifold $M$ which is required for numerical experiments, this characterizes the sigma flow as proper geometric diffusion process, i.e.~$\lim_{t\to\infty} \vq_{t}$ is constant. However, to rigorously  show  existence and global convergence in the general case, a weak set-up with a feasible set of $\vq_{t}$ containing $C^{\infty}(M,\S_c)$ as dense subspace would have to be considered, as discussed in Section \ref{sec:Harmonic-Maps-Gradient-Flows}.
\end{remark}

\begin{remark}[harmonic maps into spheres]
    A similar conclusion could have been drawn along a different line of reasoning, when considering that harmonic maps into the sphere orthant must be constant, as implied by the general theory of harmonic maps into spheres \cite{solomonHarmonicMapsSpheres1985}. The only point to note is  that $\S_c$ is isometric to the positive orthant of a sphere. 
\end{remark}

\subsection{Sigma-\texorpdfstring{$\alpha$}{ɑ} Flow}\label{sec:alpha-SigmaFlow}
The sigma flow system \eqref{eq:sigmaflow} involves the metric connection $\Gamma = \,^{0}\Gamma $ of the Fisher-Rao metric. In this section, we consider the extension to the family of $\alpha$-connections $\,^{\alpha}\Gamma$ from information geometry given by \eqref{eq:def-alpha-Gamma}. As a consequence, the tension field \eqref{eq:tension_field} given explicitly by \eqref{eq:tension-f-explicitly} will take the form (cf.~also \eqref{eq:proof-tension-locally})
\begin{equation}\label{eq:tau-alpha-sigma}
\,^{\alpha}\tau(\vS_t;\fr,h_{t})^{i} = \Delta_{h_{t}}\vt^{i} + \,^{\alpha}\Gamma^{i}_{jk}\la \dd\vt^{j},\dd\vt^{k}\ra_{h_{t}},\qquad i\in[\cc].
\end{equation}

\begin{definition}[sigma-$\alpha$ flow]\label{def:sigma-alpha-flow}
    Let $\mfk{P} \in C^{\infty}(M,\S_c)$ and $\alpha \in \R, T>0$ be given. The sigma-$\alpha$ flow is generated by the tension field \eqref{eq:tau-alpha-sigma} and the system 
    \begin{equation}
        \label{eq:asgmflw}
        \tag{$\sigma$-$\alpha$ flow}
        \vS \in C^{\infty }(M_{T},\S_c) \qq{s.t.} 
        \begin{dcases}
            \partial_t \vS_t = \,^{\alpha}\tau (\vS_t;\fr,h_t), \quad & \forall t \in (0,T), \\
            h_t = \mc{O}( \vS_t) ,\quad & \forall t \in (0,T),\\ 
            \vq_0  = \mfk{P} ,
        \end{dcases}
    \end{equation}
    where $\mc{O}$ satisfies the uniform positive definiteness condition \eqref{eq:ass-uniform-boundedness}. 
\end{definition}
The following proposition generalizes Proposition \ref{prop:sigmaflow-coordinates} accordingly. It reveals, in particular, that the sigma-$\alpha$ flow combines two linear flows corresponding to the two extreme cases of the $\alpha$-connections, viz.~the case $\alpha=1$,
\begin{equation}
    \partial _t \vt ^{i} = \Delta _{h_t} \vt ^{i},\qquad i\in[\cc]
\end{equation}
and the case $\alpha=-1$,
\begin{equation}
    \partial _t \vp_i = \Delta _{h_t} \vp_i,\qquad i\in[\cc].
\end{equation}

\begin{proposition}[sigma-$\alpha$ flow in coordinates]
    The first equation of the system \eqref{eq:asgmflw} is given with respect to the $\theta $ coordinates by
    \begin{equation}\label{eq:alpha-sigma-coordinates-e}
        \partial _t \vt^{i} = \frac{1+\alpha }{2} \Delta _{h_t} \vt^{i} + \frac{1-\alpha }{2} \fr^{ij} \Delta _{h_t} \vp_j,\qquad i\in[\cc]
    \end{equation}
    and with respect to the $p$ coordinates by
    \begin{equation}\label{eq:alpha-sigma-coordinates-m}
        \partial _t \vp_i = \frac{1 + \alpha }{2} \fr_{ij} \Delta _{h_t} \vt ^{j} + \frac{1-\alpha }{2} \Delta _{h_t} \vp_i,\qquad i\in[\cc].
    \end{equation}
\end{proposition}
\begin{proof}
By \eqref{eq:def-alpha-Gamma}, Eq.~\eqref{eq:tau-alpha-sigma} reads
    \begin{equation}\label{eq:tau-alpha-sigma-rhs-proof}
        \,^{\alpha}\tau(\vS_t;\fr,h_t)^{i} 
        = \Delta_{h_t} \vt^{i} + (1-\alpha)\,\Gamma ^{i}_{jk} \la \dd \vt ^{j} , \dd \vt ^{k} \ra_{h_t}
    \end{equation}
    and we show below equivalence to \eqref{eq:alpha-sigma-coordinates-e}. Equation \eqref{eq:alpha-sigma-coordinates-m} then follows directly using 
    \begin{equation}\label{eq:proof-partial-pi}
        \partial _\mu  \vp_i = \fr_{ij} \partial _\mu \vt ^{j}.
    \end{equation}
    Using
    \begin{equation}\label{eq:proof-Gamma-ijk}
        \Gamma ^{i}_{jk} = \frac{1}{2}\fr^{il} \pd_l \fr_{jk}
    \end{equation}
    we transform the right-hand side of \eqref{eq:alpha-sigma-coordinates-e} into \eqref{eq:tau-alpha-sigma-rhs-proof}.
    \begin{subequations}
    \begin{align}
        \frac{1 + \alpha }{2} \Delta _{h_t} \vt ^{i} + \frac{1-\alpha }{2} \fr^{ij} \Delta _{h_t} \vp_j 
        &\overset{\eqref{eq:delta-EH}}{=} 
        \frac{1 + \alpha }{2} \Delta _{h_t} \vt ^{i} 
        + 
        \frac{1-\alpha }{2} \fr^{ij} \frac{1}{\sqrt{|h_t|} } \partial _\mu ( 
            h_t^{\mu \nu } \sqrt{|h_t|}  \partial _\nu  \vp_j
        ) \\ 
        &\overset{\eqref{eq:proof-partial-pi}}{=} 
        \frac{1 + \alpha }{2} \Delta _{h_t} \vt ^{i} 
        + 
        \frac{1-\alpha }{2} \fr^{ij} \frac{1}{\sqrt{|h_t|} } \partial _\mu  (h_t^{\mu \nu } \sqrt{|h_t|}  \fr_{jk} \partial _\nu  \vt^k)  \\ 
        &= 
        \frac{1 + \alpha }{2} \Delta _{h_t} \vt ^{i} 
        + 
        \frac{1-\alpha }{2} \underbrace{\fr^{ij} \fr_{jk}}_{\delta_{ik}} \frac{1}{\sqrt{|h_t|} } \partial _\mu  (h_t^{\mu \nu } \sqrt{|h_t|}   \partial _\nu  \vt^k)  \\
        & \qquad + 
        \frac{1-\alpha }{2} \fr^{ij} \underbrace{\pd_l\fr_{jk} \partial _\mu \vt ^{l}}_{\partial_{\mu}\fr_{jk}}   \frac{1}{\sqrt{|h_t|} } h_t^{\mu \nu } \sqrt{|h_t|}   \partial _\nu  \vt^k  \\
        &= 
        \qty(\frac{1 + \alpha }{2} + \frac{1-\alpha }{2}) \Delta _{h_t} \vt ^{i} 
        + 
        (1 - \alpha)\frac{1}{2} \fr^{ij} \pd_l\fr_{jk} \la \dd   \vt ^{l}  , \dd   \vt^ k \ra_{h_t}  \\
        &\overset{\eqref{eq:proof-Gamma-ijk}}{=} 
        \Delta _{h_t} \vt ^{j} 
        + 
        (1 - \alpha) \Gamma ^{i}_{lk} \la \dd \vt ^{l}  , \dd    \vt^ k \ra_{h_t} .
    \end{align}
    \end{subequations}
\end{proof}

\begin{remark}[regularized metric $\fr_{\veps}$]\label{rem:fr-eps-entropic-flow}
    Convergence of these flows to constant solutions under the assumption of Proposition \ref{prop:lyapunov}, and with the reservation concerning the general case expressed as Remark \ref{rem:existence}, can be shown by minor adaption of the arguments. We omit the details but we note that, if the regularized metric $\fr_{\veps}$ given by \eqref{eq:def-fr-eps} is to be used instead of $\fr$, this entails the replacements
    \begin{equation}
        \fr \to \fr_\varepsilon , \quad  \Gamma \to \Gamma_\varepsilon,
        \qq{and}
        \HH(\vp) \to \HH(\vp) + \frac{\varepsilon}{2} \delta _{ij} \vt^i \vt ^j,
    \end{equation}
with $\fr_\varepsilon, \Gamma_\varepsilon$ defined by \eqref{eq:def-fr-eps}, \eqref{eq:def-Gamma-veps}.
\end{remark}

\subsection{Entropic Potential and Convergence to the Boundary}\label{sec:Entropic-Potential}
The geometric diffusion equations introduced so far produce constant solutions in the infinite time limit. This is at odds with the goal of achieving a labeling of observed data at every point $x\in M$, that is an assignment of a definite label. 
We modify the sigma flow system \eqref{eq:sigmaflow} to achieve such labelings by including a term that drives the flow to the boundary of the target manifold $(\S_c,\fr)$.

\begin{definition}[entropic harmonic energy]\label{def:entropic-harmony-energy}
    Let $m^2 > 0$ be fixed. We define the \textit{entropic harmonic energy} as 
    \begin{equation}\label{eq:entropic-harmony-energy}
        E_{\HH}: C^{\infty}(M,\S_c) \to  \R, \qquad 
        \mfk{P} \mapsto \int _M \qty(\frac{1}{2}\norm{\dd \mfk{P}}^2 - m^2\HH(\mfk{P})),
    \end{equation}
with $\HH$ given by \eqref{eq:HH}.
\end{definition}
As a consequence of including the entropy term, the expressions \eqref{eq:proof-tension-locally} and \eqref{eq:tau-vp-alternative} for the tension field in $\theta $ coordinates change to
\begin{subequations}\label{eq:tau-sigma-HH}
\begin{align}
    \tau_{\HH}(\vS_t;\fr,h)^{i} &= \Delta _h \vt ^{i} +  \Gamma ^{i}_{jk} \la \dd  \vt ^{j}  , \dd \vt ^{k} \ra_h + m^2 \vt ^{i}
    \\
    &= \frac{1}{2}(\Delta_h \vt^i + \fr^{ij} \Delta _h \vp_i  ) + m^2 \vt ^{i},\qquad i\in[\cc].
    \end{align}
\end{subequations}
See Remark \ref{rem:fr-eps-entropic-flow} for minor modifications if the metric $\fr$ is replaced by the $\varepsilon$-regularized metric $\fr_{\veps}$.

\begin{remark}[potentials in physics]
\label{rem:potentials-in-physics}
The modification of the sigma flow according to Definition \ref{def:entropic-harmony-energy} is reminiscent of adding a potential function to models of physical systems. In the present case, the potential $-\HH$ is concave, contrary to most of the common cases in physics. 
However, scenarios where locally concave potentials appear also in physics include the \textit{Higgs potential} \cite[Ch.~8]{hamiltonMathematicalGaugeTheory2017} or the \textit{Landau-Ginzburg potential} \cite[Ch.~11]{jostRiemannianGeometryGeometric2017}, where the potential has the shape of a `mexican hat' depicted by  Figure \ref{fig:mex-hat}. 
\begin{figure}[t]
\centering
\begin{subfigure}{0.3\textwidth}
\centering
    \includegraphics{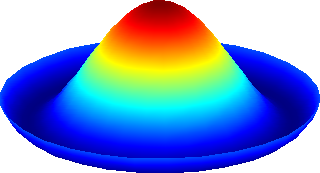}
    \caption{'Mexican hat' potential from the Ginzburg-Landau functional, which plays a major role for diffuse $\&$ threshold dynamics, cf.~\cite{Bertozzi:2016aa}.}
\end{subfigure}
\hfill
\begin{subfigure}{0.3\textwidth}
    \centering
    \includegraphics{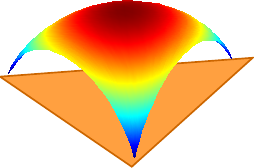}
        \caption{Entropic potential over the simplex $-\HH$, which is concave. This shape is desired because the PDE system should generate a flow towards the extreme points of the simplex.}
\end{subfigure}
\hfill
\begin{subfigure}{0.3\textwidth}
    \centering
    \includegraphics{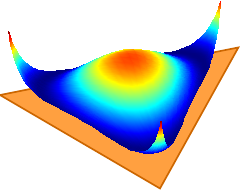}
        \caption{Adaptation of the entropic potential over the simplex to the 'mexican hat'-like potential $(0.5 - (\HH / \log(c) +1))^{2}$. Away from the barycentric point, the potential is now convex and consequently extreme points of the simplex are no longer stable stationary points.}
\end{subfigure}
\caption{Mexican hat potential and entropic potential on the simplex (colored orange).} 
\label{fig:mex-hat}
\end{figure}
A similar shape could be also achieved in the sigma flow setting by the replacement $-\HH \to (0.5 - (\HH / \log c + 1))^2$ producing a potential that is concave around the origin but convex when approaching the boundary of the simplex. 
    
The shape of the above-mentioned `mexican hat' potential may remind some readers of the `mexican hat' shaped convolution masks for edge detection in image data, generated by Laplacian-of-Gaussian operators, which have a long history in early computer vision and as physiological models of simple cells \cite{Marr:1980aa}. Our class of models introduced in this paper is PDE-based, however, rather than based on convolution followed by thresholding. Specifically, the metric $h_{t}$ on $M$ is coupled to the evolving state $\vS$, which may be used for -- in comparison to basic convolution and thresholding: sophisticated -- edge detection, as in \eqref{eq:PM-ht}, for instance. For a detailed study of the connection between models based on PDEs and on convolution, respectively, we refer to \cite{Buades:2006ue}, and to \cite{Bertozzi:2016aa} for advanced approaches combining diffusion and threshold dynamics. 
\end{remark}

\begin{definition}[entropic sigma flow]\label{def:entropic-sigma-flow}
    Let $\mfk{P} \in C^{\infty }(M,\S_c) ,\varepsilon > 0, m^{2} > 0$ be fixed. The regularized \textit{entropic sigma flow system} is given by 
    \begin{equation}
        \label{eq:hsigmaflow}
        \tag{$\sigma $-$\HH$-flow}
        \vS \in C^{\infty }(M_T,\S_c) \qq{s.t.} 
        \begin{dcases}
            \partial_t \vS_t = \tau_{\HH}(\vS_t; \fr_{\veps},h_t), \quad & \forall t \in (0,T), \\
            h_t = \mc{O}(\vS_t) ,\quad & \forall t \in (0,T),\\ 
            \vS_0  = \mfk{P} ,
        \end{dcases}
    \end{equation}
    where $\mc{O}$ satisfies the uniform positive definiteness condition \eqref{eq:ass-uniform-boundedness}.  
\end{definition}

\begin{theorem}[entropic sigma flow: convergence]\label{thm:entropic-sigma-flow-convergence}
    Assume that $\vS \in C^\infty (M_\infty , \S_c)$ is a solution of \eqref{eq:hsigmaflow} existing for all time and let $m^2,\veps > 0$ be fixed but arbitrary. Let $t>0$ be fixed and $\vt$ denote the $\theta$ coordinate expression of $\vS_{t}$, and let $\vt^i = \sum_{n\geq 0} a^i_n\phi_n$ be the decomposition of $\vt$ into eigenfunctions of the Laplacian $\Delta_{h_t}$, analogous to \eqref{eq:vti-Laplacian-decomposition}. 
    Define the set of low frequencies as
    \begin{equation}\label{eq:def-aleph}
        \aleph = \qty{n \in \N_0 \colon c_2\lambda_n + \varepsilon(\lambda_n + m^2) > 0}
    \end{equation} 
    with $c_{2}$ from \eqref{eq:bounds-mcB} and assume
    \begin{equation}\label{eq:cond-aleph}
        a^i_n = 0 \qq{for all} n \notin \aleph, \quad i \in [\cc]. 
    \end{equation} 
    If condition \eqref{eq:cond-aleph} holds for all $t>0$, then the $L^{2}(M,\S_c)$-norm of $\vS_t$ is unbounded as a function of time.
\end{theorem}
It is clear that the set $\aleph$ is never empty; it always contains 0. For $\veps$ fixed, $\aleph$ can be enlarged by increasing $m^2$. 
\begin{proof}
To alleviate notation, we drop the subscript $t$ in this proof for all time dependent quantities.
Consider the functional $\Phi : C^{\infty }(M,\S_c) \to \R$ as in Proposition \ref{prop:lyapunov}. The proof consists in showing that $\Phi(\vS)$ is strictly increasing in time.
Hence we compute $\dv{t} \Phi(\vS)$ in the $\theta$-coordinate system, where $\vS$ satisfies the differential equation \eqref{eq:hsigmaflow}
\begin{equation}
    \pdv{t} \vS = \tau_{\HH}(\vS;\fr,h), \quad h = \mc{O}(\vS)
\end{equation}
which evaluates in $\theta$-coordinates to
\begin{equation}
    \pdv{t} \vt^{i}
    \overset{\eqref{eq:tau-sigma-HH}}{=}
    \Delta_{h} \vt ^{i} +  \Gamma ^{i}_{jk} \la \dd  \vt ^{j}  , \dd \vt ^{k} \ra_{h} + m^2 \vt ^{i}. 
\end{equation}
Similarly to Eq.~\eqref{eq:dt-Phi-proof} one derives
\begin{equation}
    \dv{t} \Phi(\vS) = 
    \int_{M} (\pd_i \HH + \varepsilon \delta _{ij}\vt^{j})\pdv{t} \vt^{i} 
    = 
    \int_{M} (\fr_{\varepsilon })\vt^{j}\pdv{t} \vt^{i}
\end{equation}
yielding
\begin{equation}
    \dv{t} \Phi(\vS) = 
    \int_M (\fr_{\varepsilon })_{ij}\vt^{j}\big(
    \Delta_{h} \vt ^{i} +  \Gamma ^{i}_{jk} \la \dd  \vt ^{j}  , \dd \vt ^{k} \ra_{h} + m^2 \vt ^{i}\big). 
\end{equation}
The first two terms can then be treated as in the proof of Prop.~\ref{prop:lyapunov} where the equation
\begin{equation}\label{eq:Phi-to-mcB}
    \int_{M} (\fr_{\varepsilon })\vt^{j}\big(
    \Delta_{h} \vt ^{i} +  \Gamma ^{i}_{jk} \la \dd  \vt ^{j}  , \dd \vt ^{k} \ra_{h} \big)
    \overset{\eqref{eq:dt-Phi-proof-B}}{=}
    - \int_{M} (\mc{B} + \veps \II)_{ij} 
    \la \dd \vt ^{i} , \vt ^j \ra_{h}  
\end{equation} 
is shown with $\mc B$ defined as in Lemma \ref{lem:b}. The last expression can be bounded by 
\begin{equation}\label{eq:lower-bound-mcb}
    - \int_M (\mc{B} + \veps \II)_{ij} 
    \la \dd \vt ^{i} , \dd \vt ^j \ra_{h}  
    \geq
    - (c_2 + \varepsilon ) \delta_{ij}\int _M 
    \la \dd \vt ^{i} ,\dd \vt ^j \ra_{h}  
\end{equation}
since $c_2 \II$ is an upper bound to $\mc B$ by \eqref{eq:bounds-mcB} and due to a trace inequality argument similar to the one used in Eq.~\eqref{eq:proof-dtPhiqt}. 
The last expression can be further simplified when expressed in terms of the expansion coefficients \eqref{eq:vti-Laplacian-decomposition}
\begin{equation}\label{eq:lower-bound-dirichlet-norm}
    - \int_M
    \la \dd \vt ^{i} , \dd \vt ^j \ra_{h}  
    \overset{\eqref{eq:dt-Phi-proof-laplacian-adjoint}}{=} 
        \sum_{n \geq 0} 
    a_n^i a_n^j \lambda_n 
    \overset{\eqref{eq:cond-aleph}}{=}
    \sum_{n\in\aleph} 
    a_n^i a_n^j \lambda_n.
\end{equation} 
    Putting together the results from above yields
    \begin{subequations}\label{eq:lower-bound-first}
        \begin{align}
            \int_{M} (\fr_{\varepsilon })\vt^{j}\big(
                \Delta_{h} \vt ^{i} +  \Gamma ^{i}_{jk} \la \dd  \vt ^{j}  , \dd \vt ^{k} \ra_{h} \big)
            & 
            \overset{\eqref{eq:Phi-to-mcB}}{=} 
            - \int_M (\mc{B} + \veps \II)_{ij} 
            \la \dd \vt ^{i} , \dd \vt ^j \ra_{h} \\
            & \overset{\eqref{eq:lower-bound-mcb}}{\geq}
            - (c_2 + \varepsilon ) \delta_{ij}\int _M 
            \la \dd \vt ^{i} ,\dd \vt ^j \ra_{h} \\
            & \overset{\eqref{eq:lower-bound-dirichlet-norm}}{=}
            (c_2 + \varepsilon ) \delta_{ij}
            \sum_{n\in\aleph} 
            a_n^i a_n^j \lambda_n.
        \end{align}
    \end{subequations}
A for the second term, we invoke the positive definiteness of $\fr \succeq 0$ and obtain a lower bound
\begin{equation}
    m^2\int_M (\fr_{\varepsilon })_{ij}\vt^{j}
        \vt ^{i}
    = 
    m^2\int_M (\fr + \veps \II)_{ij}\vt^{j}
    \vt ^{i}
    \geq
    m^{2}\varepsilon \delta_{ij}  \int_M \vt^{j}
    \vt ^{i}. 
\end{equation}
Furthermore, substituting the expansion coefficients gives
\begin{equation}
    \int_M \vt^{j}
    \vt ^{i}
    =
    \sum_{n\geq 0} a_n^{i } a_n^j 
    \overset{\eqref{eq:cond-aleph}}{=}
    \sum_{n\in \aleph} a_n^{i } a_n^j 
\end{equation}
and hence the lower bound
\begin{equation}\label{eq:lower-bound-second}
    m^2\int_M (\fr_{\varepsilon })_{ij}\vt^{j}
    \vt ^{i} \geq \varepsilon m^2  
    \delta _{ij}
    \sum_{n\in \aleph} a_n^{i } a_n^j
    = \veps m^{2} \sum_{n \in \aleph}\|a_{n}\|^{2}.
\end{equation}
Combining the bounds \eqref{eq:lower-bound-first}, \eqref{eq:lower-bound-second}, we thus obtained
\begin{subequations}
\begin{align}
    \dv{t} \Phi (\vS)
    &=
    \int (\fr_{\varepsilon })\vt^{j}\big(
    \Delta_{h} \vt ^{i} +  \Gamma ^{i}_{jk} \la \dd  \vt ^{j}  , \dd \vt ^{k} \ra_{h} \big)
    +
    m^2\int_M (\fr_{\varepsilon })_{ij}\vt^{j}
    \vt ^{i} 
    \\
    &\geq 
    \sum_{n\in \aleph} 
    ((c_{2} + \varepsilon )\lambda_n + \varepsilon m^2) \norm{a_n}^2 > 0,
    \end{align}
\end{subequations}
because for all $n\in \aleph$ by \eqref{eq:def-aleph}, we have $(c_{2} + \varepsilon )\lambda_n + \varepsilon m^2 > 0$. 
This shows that $\Phi$ is monotonously increasing in time. Since $\HH$ is a bounded from above by $0$, however, we conclude that the $L^{2}$-norm of coordinate functions $\vt$ diverges as a function of time.    
\end{proof}

\begin{remark}[relevance for convergence in practice]
The interpretation of Theorem \ref{thm:entropic-sigma-flow-convergence} is that the entropic sigma flow converges to the simplex boundary $\partial \S_c$ (in the $L^{2}$ sense) if only low frequency modes are present in the state $\vS$.  
This can be expected to hold due to the diffusion part of the tension field \eqref{eq:tau-sigma-HH}. Numerical experiments substantiate this result in Section \ref{sec:Experiments}. 
\end{remark}

\subsection{Comparison to the Continuum Limit of the S Flow}
We compare the entropic harmonic energy functional \eqref{eq:entropic-harmony-energy} and the S flow functional on the continuous domain, as described by equation \eqref{eq:cont-domain-s-flow}. In order to establish a common ground between the two models, we set $M = X \subset \R^{2}$ an open domain which is a Riemannian manifold with the Euclidean metric $h=\delta$ induced by $\R^{2}$. The S flow functional $E_S$ considers $\vP : X \to \S_c\subset\R^{c}$ as a mapping to $\R^{c}$ and evaluates to   
\begin{equation}
    E_S(\vP) = \frac{1}{2}\int_X (\norm{\dd \vP}^{2} - \norm{\vP}^{2})\dd[2]{x}. 
\end{equation}
The harmonic energy is given in $\theta$-coordinates as
\begin{equation}
    E_{\HH}(\vP) = \int_X \qty (\frac{1}{2} \fr_{ij} \la \dd \vt ^{i}, \dd \vt^{j} \ra_{\delta}  - \HH(\vP)) \dd[2]{x}. 
\end{equation}
This last expression can be simplified by employing the sphere map defined next.
\begin{definition}[{sphere map \cite{Kass:1989aa,amariMethodsInformationGeometry2000}}]
    The sphere map 
    \begin{equation}
        \label{eq:sphere-map}
        \Lambda : \S_c \to 2S^{\cc}; \quad 
        p \mapsto 2 \sqrt{p}
    \end{equation}
    is a smooth isometric immersion that identifies the simplex $\S_c$ with the positive orthant of the sphere $2S^{\cc}$ of radius 2, equipped with the round metric \cite{astromImageLabelingAssignment2017,ay2017information}. 
\end{definition}

\begin{proposition}[spherical representation]
    Let $\vP : X \to \S_c$ be a smooth map with $\theta$-coordinate functions $\vt^{i}$ and let $\mfk s := \Lambda \circ \vP : X \to S^{\cc}$ the mapping of $\vP$ to the 2-sphere. Then one has
    \begin{equation}
        \label{eq:def-vs}
        \int_X \fr_{ij}\la \dd \vt^{i} , \dd \vt^{j} \ra_{\delta} \dd[2]{x}
        = 
        \int_X \norm{ \dd \mfk s}^{2} \dd[2]{x}
    \end{equation}
\end{proposition}

\begin{proof}
    The sphere metric is induced by the Euclidean metric of the ambient $\R^{c}$. Accordingly, the Riemannian norm of a vector $v \in \R^{c}$ tangent to $2 S^{\cc}$ is given by $\delta_{ij} v^{i}v^{j}$. 
    For a vector $w \in \R^{c}$ tangent to $\S_c$, with coordinate representation $\mathbf{w}\in \R^{\cc}$, its Fisher-Rao norm at a point $p \in \S_c$ is given by 
    \begin{equation}
        \frac{\delta_{ij} w^{i} w^{j}}{p_i} = \fr(p)_{ij} \mathbf{w}^{i} \mathbf{w}^{j}.  
    \end{equation}
    The isometry relation yields
    \begin{equation}\label{eq:proof-wi-vi}
        \fr(p)_{ij} \mathbf{w}^{i} \mathbf{w}^{j}
        = 
        \delta_{ij} v^{i}v^{j}  \qq{for}
        v^{i} = \pdv{\Lambda^{i}}{p ^{j}} w^{j}.  
    \end{equation}
    Specifically, if $\mathbf{w} = \partial_\mu \vt$, where $\vt$ is the coordinate vector associated to $\vp$, then $v = \partial_\mu \mfk s$ by the chain rule. Hence
    \begin{subequations}
    \begin{align}
        \int_X \fr_{ij}\la \dd \vt^{i} , \dd \vt^{j} \ra_{\delta} \dd[2]{x}
        &=
        \int_X \fr_{ij}\partial_{\mu}\vt^{i} \partial_{\nu} \vt^{j} \delta^{\mu \nu} \dd[2]{x} 
        = \int_X \fr_{ij}\mbf{w}^{i}  \mbf{w}^{j} \delta^{\mu \nu} \dd[2]{x}
        \\
        &\overset{\eqref{eq:proof-wi-vi}}{=} \int_X \delta_{ij} v^{i} v^{j} \delta^{\mu \nu} \dd[2]{x}
        = \int_X \delta_{ij}  \partial_\mu \mfk s^{i} \partial_\nu \mfk{s}^{j} \delta^{\mu \nu} \dd[2]{x} \\
        &=
        \int_X \norm{\dd \mfk s}^{2} \dd[2]{x}.
    \end{align}
    \end{subequations}
\end{proof}

\begin{remark}[\textbf{generalized S flows}]
    All the steps above still work if we replace $(X,\delta)$ with another Riemannian manifold $(M,h)$. This paves the way for extensions of the S flow to more general base manifolds.  
\end{remark}

\begin{corollary}[entropic harmonic energy on the sphere]
    Let $\vP :M \to \S_c$ be a smooth map and $\mfk s = \Lambda \circ \vp : M \to S^{\cc}$ the associated sphere-valued map. Then the entropic harmonic energy takes the form
    \begin{equation}
        E_{\HH}(\vP) 
        = 
        \int_X \qty (\frac{1}{2} \fr_{ij} \la \dd \vt ^{i}, \dd \vt^{j} \ra_{\delta}  - \HH(\vP)) \dd[2]{x} 
        =
        \int_X \qty (\frac{1}{2} \norm{\dd \mfk s}^{2} - \HH(\vP)) \dd[2]{x}.
    \end{equation}
\end{corollary}

\vspace{0.1cm}
We are now in a position to compare the energy $E_{\HH}$ with the functional \eqref{eq:cont-domain-s-flow}
\begin{equation}
    E_S(\vP) = \frac{1}{2}\int_X (\norm{\dd \vP}^{2} - \norm{\vP}^{2})\dd[2]{x} 
\end{equation} 
governing the S flow (recall Section \ref{sec:s-flows}).  
The two functionals formulate similar goals in that, by minimization, both enforce smoothness of the function $\vp$ while also penalizing configurations that are close to the barycenter in $\S_c$ which is the point $\theta = 0$. The S flow energy pushes to the boundary by maximizing the purity $\norm{p}^{2}$, while the sigma flow energy maximize $\HH$ which means to minimize the entropy $H$. While the S flow energy enforces smoothness by reducing the $L^{2}$-norm of the gradient of $\vP$, the sigma flow reduces the magnitude of the gradient of the corresponding sphere-valued map $\mfk{s}$. 

Thus, both processes induced by minimizing the respective functionals achieve the same goal, namely to generate smooth mappings to the boundary of the simplex, yet in a slightly different manner.

\subsection{Tangent Space Parametrization}
In this section, we generalize the tangent space parametrization of assignment flows from \cite{zeilmannGeometricNumericalIntegration2020} to entropic sigma flows \eqref{eq:hsigmaflow} and an arbitrary $\alpha$-connection, i.e.~for the entropic extension of sigma-$\alpha$ flows \eqref{eq:asgmflw} (Def.~\ref{def:sigma-alpha-flow}). Such parametrizations are essential for numerical computation. We first introduce a convenient representation of the flow extending the tension field \eqref{eq:tension-coordinates} and some further notation.
\begin{definition}[entropic sigma-$\alpha$ flow]\label{def:entropic-sigma-alpha-flow}
    Let $\vP\in C^{\infty}(M,\S_c)$ and $\alpha \in \R,m^{2}\in \R_>$ be given. The entropic sigma-$\alpha$ flow in $\theta$-coordinates is the system
    \begin{equation}
        \label{eq:ah-sigmaflow}
        \tag{$\sigma$-$\HH$-$\alpha$ flow}
        \vS \in C^{\infty }(M_{T},\S_c) \qq{s.t.} 
         \quad 
        \begin{dcases}
            \partial_t \vt^{i} &= 
            \Delta_{h_t} \vt^{i} 
        + \frac{1-\alpha}{2}\Big(  
            \la \dd (\vt^{i} - \Psi) , \dd (\vt ^{ i } - \Psi) \ra_{h_t} 
            \\ &\hspace{3cm}
            -
            \la \dd  \Psi , \dd \Psi \ra_{h_t} 
        \Big)
        + m^{2} \vt^{i}
            \\
            h_t &= \mc{O}(\vS_t) \\ 
            \vS_0  &= \vP,
        \end{dcases}
    \end{equation}
    $\forall t \in (0,T),\,i \in [\cc]$, 
    where $\mc{O}$ satisfies the uniform positive definiteness condition \eqref{eq:ass-uniform-boundedness} and $\vt^{i}$ are the coordinate functions of $\vS_t$ and 
\begin{equation}
    \Psi := \psi \circ \vS_t : M \to \R. 
\end{equation}
\end{definition}

For any vector in $\theta \in \R^{\cc}$, we denote by $(0,\theta) \in \R^{c}$ the vector with 0 prepended as first component. Furthermore, the natural logarithm $\log$ as well as the exponential function apply componentwise to vectors, that is for a vectors $p \in \R^{c}_>$ and $v\in\R^{c}$, one has $(\log p)_{i}=\log p_{i}$ and $(e^{v})_{i}=e^{v_{i}}$, for every $i\in[c]$. We set $\textbf{1}_c:=(1,\dotsc,1)^{\T}\in\R^{c}$. 
The tangent space to $\S_c$ is the linear subspace
\begin{equation}\label{eq:def-T0}
    T_0 \coloneqq \Big\{ v \in \R^{c} \colon \sum_{i \in [c]} v_i = 0 \Big\}
\end{equation} 
with linear orthogonal projection 
\begin{equation}\label{eq:pi-0}
    \Pi_0 : \R^{c} \to T_0, \qquad v \mapsto v - \Big(\frac{1}{c} \sum_{i \in [c]} v_i \Big) \textbf{1}_c.
\end{equation}
The mapping 
\begin{equation}
    \label{eq:softmax}
    \sm : T_0 \to \S_c, \qquad v \mapsto \mathrm{softmax}(v) := \frac{e^{v}}{\la  \textbf{1}_c, e^{v}\ra}
\end{equation}
defines a smooth diffeomorphism between $T_0$ and $\S_c$ with inverse given by
\begin{equation}\label{eq:sm-1}
    \sm^{-1}: \S_c \to T_0, \qquad p \mapsto \Pi_0 \log p.
\end{equation}
Thus, we can uniquely parametrize elements $p$ of $\S_c$ by tangent vectors $v$ through the relation $p = \sm(v)$. We consider furthermore the \textit{replicator mapping} (as special case of the replicator tensor \eqref{eq:replicator-tensor})
\begin{equation}
    R : \S_c \to \R^{c \times c}, \qquad 
    p \mapsto R_p = [p_i(\delta_{ij} - p_j)]_{i,j \in [\cc]_{0}}. 
\end{equation}
It has the properties 
\begin{equation}\label{eq:replicator-props}
    R_p \Pi_0 = R_p, \quad \forall  p \in \S_c 
    \qquad\text{and}\qquad
    \dd\sm (v) = R_{\sm(v)}. 
\end{equation}
The maps $R$ and $\Pi_0$ can be extended to a maps $R : C^{\infty}(M,\S_c) \to C^{\infty}(M,\R^{c \times c})$ and $\Pi_0 : C^{\infty}(M,\R^{c}) \to C^{\infty}(M,T_0)$ by post-composition.

We specify the relations between the tangent space parametrization and $\theta$-coordinates.
\begin{proposition}[tangent coordinates and $\theta$-coordinates]\label{prop:tangent-space-parametrization}
For $p \in \S_c$ with coordinate representation (recall \eqref{eq:def-kappa-chart})
\begin{equation}
\theta = (\theta^1, ...,\theta^\cc) = \chi_{e}(p)
\end{equation}
and tangent space parameters 
\begin{equation}\label{eq:v-by-sm-1}
v = (v^0,\dots,v^{\cc}) = \sm ^{-1}(p) \in \R^{c}, 
\end{equation}
one has the relations
\begin{equation}\label{eq:v-to-theta}
    v^i = \theta^{i} + v^{0},\quad i\in[\cc] \qquad\text{where}\qquad
    v^{0} = - \frac{1}{c} \sum_{j \in [\cc]} \theta^{j}. 
\end{equation}
These relations are summarized by the equations
\begin{subequations}\label{eq:tangent-theta}
\begin{align}
v &= \Pi_0 (0,\theta), \\
(0,\theta) &= v - v^{0} \mathbf{1}_c. 
\end{align}
\end{subequations}
\end{proposition}
\begin{proof}
    Equations \eqref{eq:v-by-sm-1}, \eqref{eq:sm-1} and \eqref{eq:pi-0} yield 
    \begin{equation}
        v^{i} = \log p_i - \frac{1}{c} \sum_{j\in[\cc]_{0}} \log p_j
        = \log\frac{p_{i}}{p_{0}} - \frac{1}{c} \sum_{j\in[\cc]_{0}} \log\frac{p_j}{p_{0}}
        \overset{\eqref{eq:def-kappa-chart}}{=}
        \theta^{i}-\frac{1}{c}\sum_{j\in[\cc]}\theta^{j},\qquad i\in[\cc].
    \end{equation}
    The expression for $v^{0}$ follows from $v\in T_{0}$ and \eqref{eq:def-T0}: 
\begin{equation}
    v^{0}=-\sum_{i\in[\cc]} v^{i}=\frac{\cc}{c}\sum_{j\in[\cc]}\theta^{j}-\sum_{i\in[\cc]}\theta^{i}
    = \frac{\cc-c}{c}\sum_{j\in[\cc]}\theta^{j}
    = -\frac{1}{c}\sum_{j\in[\cc]}\theta^{j}.
\end{equation}    
\end{proof}
    We extend the bilinear pairing on the cotangent space $\la \cdot, \cdot \ra_h : \Omega^{1}(M) \otimes \Omega^{1}(M) \to C^{\infty}(M)$ to a bilinear mapping on vector valued forms $\la \cdot, \cdot \ra_h : (\Omega^{1}(M) \otimes \R^{c}) \otimes (\Omega^{1}(M) \otimes \R^{c}) \to C^{\infty}(M,\R^{c})$. For a smooth vector-valued function $f : M \to \R^{c}$ with $\dd f \in \Omega^{1}(M) \otimes \R^{c}$, the pairing is locally given by
    \begin{equation}
        \label{eq:vector-metric-term}
        (\la \dd f , \dd f \ra_h)(x) = \qty(\partial_\mu f^{i}(x) \partial_\nu f^{i}(x) h^{\mu \nu}(x))_{i \in [\cc]_0} \quad x \in M.
    \end{equation}
    With this notation, we describe the tangent space parametrization of the sigma flow model as follows. 
\begin{proposition}[tangent space representation of sigma flow]\label{prop:sigma-flow-tangent-representation}
    Let $T >0$ be given and $\vq : M_T \to \S_c$ be a solution to the entropic sigma-$\alpha$ flow system \eqref{eq:ah-sigmaflow} for an initial condition $\vP$, the mass parameter $m^{2} >0$ and some fixed $\alpha \in \R$. 
    Let $\vv_{t} = \sm^{-1}(\vq)$ denote the tangent space representation of $\vS$. Then $\mfk v$ satisfies the PDE
    \begin{equation}\label{eq:def-ts-parametrization}
        \pdv{t} \mfk v_{t} = \Pi_0 \qty(\Delta_{h_t} \mfk{v}_{t} + \frac{1-\alpha}{2}\la \dd \log \sm(\vv_{t}),\dd \log \sm(\vv_{t}) \ra_{h_t} + m^{2} \vv_{t}), \quad 
        h_t = \mc{O}(\sm(\vv_{t})), 
    \end{equation}
    where $\Pi_0$ acts on the function to its right by post-composition and the second term in the parenthesis is the application of the bilinear pairing $\la\cdot,\cdot\ra_h$ \eqref{eq:vector-metric-term} to the vector valued one-form $\dd \log \sm(\vv_{t}) \in \Omega^{1}(M) \otimes \R^{c}$.
    Furthermore, the function $\vS : M_T \to \S_c \subset \R^{c}$ satisfies the PDE 
    \begin{equation}
        \pdv{t} \vS_t = R_{\vS_t} \qty(\Delta_{h_t} \log  \vS_t + \frac{1-\alpha}{2}\la \dd \log \vS_t,\dd \log \vS_{t} \ra_{h_{t}} + m^{2} \log \vq_t), \quad 
        h_t = \mc{O}(\vS_t).
    \end{equation}
    where the matrix function $R_{\vS_t}$ acts on the function to its right by pointwise matrix-vector multiplication and the second term in the parenthesis is the application of the bilinear pairing $\la\cdot,\cdot\ra_{h_t}$ \eqref{eq:vector-metric-term} to the vector valued one-form $\dd \log \vq_t \in \Omega^{1}(M) \otimes \R^{c}$.
\end{proposition}

\begin{proof}
    In the following, we simplify notation and do not indicate the time dependencies $t\mapsto \vv_{t},t\mapsto h_{t}$ and  $t\mapsto \vt_{t}$. From relation \eqref{eq:tangent-theta} follows
    \begin{equation}\label{eq:dt-vv}
        \pdv{t} \vv
        = \Pi_0 \pdv{t} (0,\vt).
    \end{equation}
    Then, due to \eqref{eq:ah-sigmaflow}
    \begin{equation}
        \pdv{t} \vt^{i} = \Delta_{h} \vt^{i} 
        + \frac{1-\alpha}{2}\qty( 
            \la \dd (\vt^{i} - \Psi) , \dd (\vt ^{ i } - \Psi) \ra_{h} 
            -
            \la \dd  \Psi , \dd \Psi \ra_{h} 
        )
        + m^{2} \vt^{i}
    \end{equation}
    We denote the ambient coordinates of $\vS$ by $(\vp_0,\dots,\vp_\cc)$. 
    By virtue of \eqref{eq:v-to-theta}, \eqref{eq:def-theta-p} and \eqref{eq:def-psi-partition} we have
    \begin{equation}
        \vS = (\vp_0,\dots,\vp_\cc), \quad
        (\vt^{i} - \Psi) = \log \vp_{i} ,\quad 
        - \Psi = \log \vp_0, \quad 
        \vt ^{i} = \vv ^{i} - \vv^{0}.
    \end{equation}
    This yields
    \begin{equation}
        \dv{t}\vt^{i} = \Delta_{h} (\vv^{i} - \vv^{0})
        + \frac{1-\alpha}{2}\qty( 
            \la \dd \log \vp_{i} , \dd \log \vp_{i} \ra_{h} 
            -
            \la \dd \log \vp_{0} , \dd \log \vp_{0} \ra_{h} 
        )
        + m^{2} (\vv^{i} - \vv^{0})
    \end{equation}
    and in turn
    \begin{equation}\label{eq:dt-vt-first-comp}
        \dv{t}\vt^{i} = \Delta_{h} \vv^{i} 
        + \frac{1-\alpha}{2}
            \la \dd \log \vp_{i} , \dd \log \vp_{i} \ra_{h} 
        + m^{2} \vv^{i}
        - \Delta_{h} \vv^{0} 
        - \frac{1-\alpha}{2} \la \dd \log \vp_{0} , \dd \log \vp_{0} \ra_{h} 
        - m^{2} \vv^{0}.
    \end{equation}
    The projection $\Pi_0$ is linear, fulfills $\Pi_0 \textbf{1}_{c} = 0$ and allows to write
    \begin{align}
        \Pi_0 \pdv{t}(0, \vt)
        =
        \Pi_0 \qty(\pdv{t}(0,\vt) 
        + 
        \big[\Delta_{h} \vv^{0} 
        + \frac{1-\alpha}{2} \la \dd \log \vp_{0} , \dd \log \vp_{0} \ra_{h} 
        + m^{2} \vv^{0} \big] \textbf{1}_{c} ), 
    \end{align}
    which by \eqref{eq:dt-vv} and \eqref{eq:dt-vt-first-comp} amounts to
    \begin{align}
        \pdv{t}\vv 
        =   
        \Pi_0 \qty(\Delta_{h} \mfk v + \frac{1-\alpha}{2}\la \dd \log \sm(\vv),\dd \log \sm(\vv) \ra_{h} + m^{2} \vv).
    \end{align}
    Using \eqref{eq:replicator-props} and the defining relation $\sm(\vv) = \vq$, the PDE governing $\vq$ reads 
    \begin{equation}
        \pdv{t} \vq = R_\vq \dot{\vv} = 
        R_\vq \Pi_0 \qty(\Delta_{h} \mfk v + \frac{1-\alpha}{2}\la \dd \log \sm(\vv),\dd \log \sm(\vv) \ra_{h} + m^{2} \vv). 
    \end{equation} 
    Using $R_\vq \Pi_0 = R_\vq$ (cf. \eqref{eq:replicator-props}) and $\vq = \Pi_0 \log \vv$ finally leads to
    \begin{equation}
        \pdv{t} \vq = R_\vq \pdv{t} \vv = 
        R_\vq  \qty(\Delta_{h} \log \vq + \frac{1-\alpha}{2}\la \dd \log \vq,\dd \log \vq \ra_{h} + m^{2} \log \vq).
    \end{equation}
\end{proof}
\begin{remark}[advantage of tangent space parametrization]
    Regarding nonlinear flow integration on the probability simplex $(\S_c,\fr)$, the tangent parametrization enables better conditioned numerical computation than the exponential $\theta$-coordinates.
\end{remark}


\section{Experiments and Comparison}\label{sec:Experiments}
This section presents few proof-of-concept experimental results that illustrate the sigma flow model. 
\begin{itemize}
\item
Section \ref{sec:Implementation} describes a discretization of the approach (Definition \ref{def:entropic-sigma-alpha-flow}) based on the tangent space parametrization (Proposition \ref{prop:sigma-flow-tangent-representation}) and implementation details. 
\item
Section \ref{sec:Comparison-S-Flow} compares the resulting semi-discrete sigma flow model and the discretized S flow model from Section \ref{sec:s-flows}. 
\item
The convergence behavior of the numerically integrated sigma flow with and without mass parameter is reported in Section \ref{sec:1-synthetic-benchmark-for-convergence-behavior}, which illustrates Proposition \ref{prop:lyapunov} and Theorem \ref{thm:entropic-sigma-flow-convergence}.
\end{itemize}

Finally, we focus on the expressivity of the sigma flow model and on learning a mapping from data to the Riemannian metric $h$ of the domain manifold $M$. 
\begin{itemize}
\item
Section \ref{sec:2-universal-approximation-behavior-of-sigma-flow} demonstrates that, for any given image, a metric exists which generates the image just out of noise. This raises the question to what extent such metric-valued mappings generalize to an entire \textit{class} of images. 
\item
We demonstrate empirically in Section \ref{sec:3-learning-a-mapping-mco-from-simple-patterns-and-generalizing-to-complex-data} for computer generated labelings of \textit{unseen noisy} 2D \textit{random} polygonal regions, that this is indeed possible using a metric-valued mapping parametrized by a small neural net, even when then input is corrupted with a high level of noise. Details of the implementation are listed in Section \ref{sec:implementation-details}.
\end{itemize}
This finding sheds light on the intriguing problem of the generalization of this generative approach, which is just based on predicting a section of a $2\times 2$ positive definite tensor field, to classes of \textit{real} images.  
A corresponding thorough investigation is beyond the scope of this paper, however.

We therefore touch only briefly on this subject and conclude by illustrating how much the aforementioned map, trained on \textit{synthetic} random Voronoi partitions, \textit{fails} to generalize to \textit{unconstrained real} images. Although the random polygonal scene scenario, that was used for learning, considerably differs from real images, it turns out that real image structure can be recovered remarkably well. Enhancing the metric prediction map which parametrizes the Laplace-Beltrami operator in order to close this gap, is left for future work.

\subsection{Implementation}\label{sec:Implementation}
Numerical computations are based on the semi-discrete problem: the data manifold $M$ is discretized but the time dimension is kept continuous. We assume that $M$ is covered by a single global coordinate chart which is the case, e.g., for the torus $M=\TT^{2}$. 

The equation to be discretized is the sigma flow model with respect to $\alpha$-connections and entropic potential with weight $m^2 \geq 0$ in ambient coordinates, i.e.~the initial value PDE problem
\begin{equation}\label{eq:exp-sigma-model-continuous}
    \vq \in C^{\infty}(M_T,N), \qq{s.t.} \quad
    \begin{dcases}
        \pdv{t} \vq &= R_{\vq_t} \qty(\Delta_{h_t} \log  \vq_t + \frac{1-\alpha}{2}\la \dd \log \vq,\dd \log \vq \ra_h + m^{2} \log \vq),  \\
        h_t &= \mc{O} (\vq_t), \\ 
        \vq_0 &= \vP,
    \end{dcases}
\end{equation} 
$\forall  t \in (0,T)$, 
for a given $\vP \in C^{\infty}(M,\S_c)$ and some chosen fixed time period $[0,T],\, T>0$.  

Spatial discretization entails to replace paths $[0,T] \to C^{\infty}(M,\S_c)$ by paths taking values in the assignment manifold, denoted by
\begin{equation}
    \label{eq:def-fS}
\fq\colon [0,T] \to \mc{W}^N_c,
\end{equation}
where $N$ is the number of points used to discretize $M$. Due to our assumption that $M$ is covered by a single coordinate chart, we can discretize $M$ by a regular grid. $\log \vq$ is discretized by $t \mapsto \log \fq(t)$ where $\log$ acts on all the entries of the matrix $\fq(t)$ separately.
The metric $h \in \Gamma\big(T^{\ast}(M_T)\otimes T^{\ast}( M_T)\big)$ (recall the notation \eqref{eq:def-MT}) is represented by a matrix function 
\begin{equation}
\label{eq:def-fh}
\fh : [0,T] \to \R^{N \times 2 \times 2}.
\end{equation}
We denote the first dimension slices of the tensor $\fh \in \R^{N \times 2 \times 2}$ by 
\begin{equation}\label{eq:ha-2x2}
\fh(t)_a \in \R^{2 \times 2},\quad a \in [N]. 
\end{equation}
Discretization of the operator $\mc{O}$ yields the function 
\begin{equation}\label{eq:mcO-discrete}
\msf O: \mc{W}^N_c \to \R^{N \times 2 \times 2}
\end{equation}
to be further specified below. 
The Laplace-Beltrami operator 
\begin{equation}
    \Delta _h = \frac{1}{\sqrt{\abs{h}} } \partial_{\mu } \qty(\sqrt{\abs{h}} h^{\mu \nu } \partial _\nu )
\end{equation}
is the composition of the differential operator $\partial_{\mu }(\sqrt{\abs{h}} h^{\mu \nu } \partial _\nu )$ with the multiplication of the scalar function $1 / \sqrt{\abs{h}}$, which can be discretized individually. 
As for the differential operator,  we adopt the discretization of \cite[Sec.~3.4.2]{Weickert:1998aa} which yields a sparse matrix $\msf{E}_\fh \in \R^{N \times N}$. Discretizing the scalar function $1 / \sqrt{\abs{h}}:M \to \R$ yields a diagonal matrix $\msf{Q}_\fh \in \R^{N \times N}$ whose entries are given by
\begin{equation}
    (\msf{Q}_\fh)_{aa} = \frac{1}{\sqrt{\abs{\fh_a}}} . 
\end{equation}  
As a result, the discretized Laplace-Beltrami operator is given by the matrix
\begin{equation}
    \label{eq:def-fL}
    \fL_\fh = \msf{Q}_\fh \msf{E}_\fh \in \R^{N \times N}.
\end{equation} 
In view of the second term of the PDE \eqref{eq:exp-sigma-model-continuous}, we describe now a discretization of the operation
\begin{equation}
    \label{eq:bilin-cotan}
    \la \dd \log \vq, \dd \log \vq \ra_h = h^{\mu \nu} \partial_\mu (\log \vq) \partial _\nu (\log \vq) :M \to \R^{c}. 
\end{equation} 
The derivative operators $\partial_\mu$ are discretized by stencil operators that generate matrices $\msf{D}_\mu   \in \R^{N \times N}, \mu \in [\dim M]$. 
For instance, a $3\times 3$ stencil for estimating the partial derivative $\partial_{1}$ in $x$-direction on a 2d grid is
\begin{equation}
    \frac{1}{8}
    \begin{pmatrix}
        -1 & 0 & 1 \\ 
        -2 & 0 & 2 \\
        -1 & 0 & 1
    \end{pmatrix}
    .
\end{equation}
The discrete derivative operators $\msf{D}_\mu$ act on the matrices $\log \fq \in \R^{N \times c}$ by matrix multiplication from the left. As a result, the discretized version of \eqref{eq:bilin-cotan} is given by the matrix 
\begin{equation}
    \label{eq:def-fD}
    \la \msf{D} \log \fq, \msf{D} \log \fq \ra_\fh \in \R^{N \times c}, \quad 
    \big(\la \msf{D} \log \fq, \msf{D} \log \fq \ra_\fh\big)_a
    = 
    \fh_{a}^{\mu \nu} (\msf{D}_\mu \log \fq)_a (\msf{D}_\nu \log \fq)_a, \quad a\in [N].  
\end{equation}
We summarize these discretization rules in the following definition.
\begin{definition}[semi-discrete sigma flow]
    Let $\msf{P} \in \mc{W}^N_c,\alpha \in \R, m^{2} \in \R_>$ be given. The \textit{semi-discrete sigma flow model} is the initial value ODE problem
    \begin{equation}\label{eq:semidiscrete-sigma-model}
        \msf S : [0,T] \to \mc{W}^N_c \qq{s.t.}
        \begin{dcases}
            \dot{\msf{S}} &= R_\fq \qty(\fL_{\fh} \log \fq + \frac{1-\alpha}{2} \la \fD \log \fq , \fD \log \fq \ra_{\fh} + m^{2} \log \fq), \\
            \msf h(t) &= \msf{O} \big(\fq(t)\big), \qquad \forall  t \in [0,T], \\ 
            \msf{S}(0) &= \msf{P}.
        \end{dcases}
    \end{equation}
    The first equation is understood as equality of functions $[0,T] \to T\mc{W}^N_c$, whereas $\msf O : \mc{W}^{N}_c \to \R^{N \times 2 \times 2}$ maps to the open cone of symmetric positive-definite matrices and satisfies the uniform positive definiteness condition \eqref{eq:ass-uniform-boundedness}. 
\end{definition}
The dynamical system \eqref{eq:semidiscrete-sigma-model} evolves on the assignment manifold. 
The papers \cite{zeilmannGeometricNumericalIntegration2020} and \cite{Boll:2024aa} have established numerical methods for the geometric numerical integration of such systems that can also be employed here.

\subsection{Comparison to the Discrete S Flow Model}\label{sec:Comparison-S-Flow}
We compare the discrete S flow model from Section \ref{sec:s-flows} and the semi-discrete sigma flow model \eqref{eq:semidiscrete-sigma-model} with $m^{2}=1,\alpha=1$. This makes sense, because the S flow is intimately related to the e-geometry of the probability simplex (c.f. \cite{amariMethodsInformationGeometry2000}) which is realized by $\alpha=1$. The discrete S flow ODE is given by 
\begin{equation}
    \dot{\fq} = R_{\fq}(\fL_{\Omega} \fq + \fq),
\end{equation}   
whereas the sigma flow ODE with $\alpha =1 ,m^{2} =1$ reads
\begin{equation}
    \dot{\fq} = R_{\fq}(\fL_{\fh} \log \fq +  \log \fq),
\end{equation} 
which reveals the similarity of these two models. Both systems couple replicator equations as the original assignment flow model, yet with different fitness functions. To be specific, the S flow has its Laplacian operator parametrized by a weight matrix $\Omega$, whereas the sigma flow is parametrized by a spatially discretized Riemannian metric $\fh$. Furthermore, the fitness function of the sigma flow acts purely on the logarithms of $\fq$, whereas the S flow acts on $\fq$ directly.

\subsection{Relation with the Transformer Network Architecture}
\label{sec:transformer}

The sigma model has been introduced from a \textit{optimize-then-discretize} viewpoint: Minimizing various generalized harmonic energies (e.g., \eqref{eq:def-EH}, \eqref{eq:def-EB}, \eqref{eq:EH-simplex}, \eqref{eq:entropic-harmony-energy}) generate variants of the sigma model approach. Subsequently, after discretization (Section \ref{sec:Implementation}), the differential equations are numerically integrated to determine the sigma flow.

A viable alternative is the antipodal \textit{discretize-then-optimize} viewpoint that first focuses on a discretization of the harmonic energy functional. We only provide here a brief account of the ideas and refer to \cite{gaster2018computing,wu2020computing} and \cite{desbrun2005discrete}.
 
First, assume that the Riemannian manifold $(M,h)$ has been approximated by a \textit{simplicial complex}, which is plausible at least in the case of image data, see \cite{Gu:2002aa} for instance. This discretization may be represented in terms of a graph $G=(V,E)$ with $N=|V|$ vertices and a weight matrix $\Omega \in \R^{N \times N}$ satisfying
\begin{itemize}
    \item non-negativity: $\Omega_{ab} \geq 0,\; a,b\in V$, 
    \item symmetry: $\Omega_{ab} = \Omega_{ba},\;\forall a,b\in V$,
    \item support on edges: $\Omega_{ab} > 0 \;\gdw\; ab\in E$. 
\end{itemize}
Section \ref{sec:Implementation} provides a basic example: A regular grid is used to discretize a torus $\TT^{2}$, and a local $3\times 3$ symmetric stencil is used to discretize a differential operator of the form $u \mapsto \mathrm{div}(D \nabla u)$. This realizes an elementary simplicial complex where the graph is given by the grid and the weight matrix is encoded by the stencils weights. These weights matrix fulfills the three criteria for a weight matrix of \cite[pp.~76]{Weickert:1998aa}. 

In this set-up, maps $\mfk{P}: M \to \S_c$ are discretized by assignment matrices $\msf{P}\in \mc{W}^{N}_c$, and the harmonic energy functional is succinctly discretized to 
\begin{equation}
    \msf{E} : \mc{W}^{N}_c \to \R, \qquad \msf{P} \mapsto \sum_{a,b \in [N]^{2}} \Omega_{ab} d_{\fr}(\msf{P}_a,\msf{P}_b),
\end{equation}
where $d_\fr$ denotes the distance function on $\S_c$ induced by the Fisher-Rao metric. The variation of this discrete harmonic energy is given by the \textit{discrete tension field} \cite{gaster2018computing}
\begin{equation}
    \Upsilon(\msf{s}, \Omega) \in \R^{N \times c}, \qquad \Upsilon(\msf{s},\Omega)_a = \frac{1}{\mu_a}\sum_{b \in [N]} \Omega_{ab} \exp_{\msf{P}_a}^{-1}(\msf{P}_b) \in T_0,
\end{equation} 
where $\mu_a$ is the vertex weight induced by the simplicial complex; see \cite{desbrun2005discrete} for details. 
This expression can be simplified when we isometrically identify the simplex $\S_c$ with the intersection of the positive orthant and the sphere $S^{c-1}$ (\cite[Lemma 1]{astromImageLabelingAssignment2017}) and replace $\msf{P}$ by a matrix $\msf{s} \in \R^{N \times c}$ whose columns are unit vectors. The discrete tension field takes the explicit form
\begin{equation}
    \Upsilon(\msf{s},\Omega)_a = \frac{1}{\mu_a}\sum_{b \in [N]} \Omega_{ab} \frac{\mathrm{arccos}(\la \msf{s}_a,\msf{s}_b \ra)}{1 - \la \msf{s}_a,\msf{s}_b \ra}
    (\msf{s}_b - \msf{s}_a \la \msf{s}_a,\msf{s}_b \ra),
\end{equation}   
Using this representation, we can formulate an alternative discrete sigma flow equation with fixed weight matrix of the form
\begin{equation}
    \label{eq:mesh_sigma}
    \dot{\msf{s}}(t) = \Upsilon(\msf{s}(t),\Omega), \qq{where} \msf{s}_a(t) \in S^{c-1} \qq{and} \msf{s}_a(t) \geq 0, \quad \forall a \in [N],\; t > 0. 
\end{equation}
To generalize, we can make the weight matrix $\Omega$ dependent on the state $\msf{s}(t)$ and the time $t$, i.e.~we consider time-variant weight matrices. As a result, we obtain a dynamical system evolving on a product of unit spheres, whose solution can be approximated by the geometric Euler scheme (assuming no constraints are violated) 
\begin{equation}
    \msf{s}(t_{k+1}) = \Pi\Big(\msf{s}(t_k) + \delta \Upsilon\big(\msf{s}(t_k),\Omega(\msf{s}(t_k) , t_k)\big) \Big),
\end{equation}  
where $\Pi: \R^{N\times c}\to (S^{c-1})^{N},\; x \mapsto x / \norm{x}$ is the standard projection and $\delta$ is the step size for numerical integration. We analyze this update step in the extreme case $\la \msf{s}_a, \msf{s}_b \ra = 0$ corresponding to the largest possible distance on the part of the sphere contained in the positive orthant. The equations then simplify to
\begin{equation}
    \msf{s}(t_{k+1}) = \Pi\Big(\msf{s}(t_k) + \frac{\delta \pi}{2\mu} \Omega (\msf{s}(t_t),t_k) 
    \msf{s}(t_k) \Big) \qq{(equality of matrices)}. 
\end{equation}
This expression shares many similarities with an \textit{attention block} of the \textit{Transformer} network architecture \cite{vaswani2017attention}. As presented in \cite[pp.~5]{geshkovski2023mathematical}, a simplified version of the attention block can be abstractly understood as the update rule
\begin{equation}
    \msf{s}(t_{k+1}) = \Pi\Big(\msf{s}(t_k) +  \mathrm{SelfAttention}\big(\msf{s}(t_k)\big) 
    \msf{s}(t_k) \Big), 
\end{equation}   
where the so-called \textit{dot product self-attention mechanism}
\begin{equation}
    \mathrm{SelfAttention}\big(\msf{s}(t_t)\big) 
    =
    \mathrm{sm}\Big(\big\la\msf{s}(t_k) , \msf{s}(t_k) \big\ra\Big)
\end{equation}
induces a valid symmetric weight matrix on a graph with $N$ vertices that is fully connected. 
    This simplified version of the attention block ignores fully connected layers usually employed in the transformer architecture. 
The comparison shows however that this version of the discrete sigma flow can be interpreted as a generalized attention block, since the formulation \eqref{eq:mesh_sigma} realizes the simplified attention block as a special case on a fully connected graph. 
Furthermore, the choice 
\begin{equation}
    \Omega (\msf{s}(t_t),t_k) = \mathrm{EdgeMask}\big(\mathrm{SelfAttention}\big(\msf{s}(t_t)\big)\big) 
\end{equation}
would be a valid weight matrix if the entries of the self-attention block not attending to connected vertices are masked out. 
    On the other hand, the formulation \eqref{eq:mesh_sigma} contains only weights on graph edges, because the discretization scheme is only supposed to achieve the first order approximation of the smooth structure. In order to approximate the smooth structures to higher precision, accounting for the contribution of points in two-hop neighborhoods etc. would be required. This produces a fully supported weight matrix in the limit of highest order approximation with a fixed number of discretization points.

We leave the investigation of mutual connections and implications for research from the continuous \textit{optimize-then-discretize} perspective, as adopted in this paper, for future work.

\subsection{Implementation Details}\label{sec:implementation-details}
\textbf{Code.}
The experiments from the Sections \ref{sec:1-synthetic-benchmark-for-convergence-behavior}, \ref{sec:2-universal-approximation-behavior-of-sigma-flow}, \ref{sec:3-learning-a-mapping-mco-from-simple-patterns-and-generalizing-to-complex-data} are documented in code at 
\begin{center}
    \url{https://github.com/IPA-HD/sigmamodels}
\end{center}
The repo provides all data needed to reproduce the figures and the training process.

\begin{wrapfigure}{r}{.5\textwidth}
\includegraphics[width=.24\textwidth]{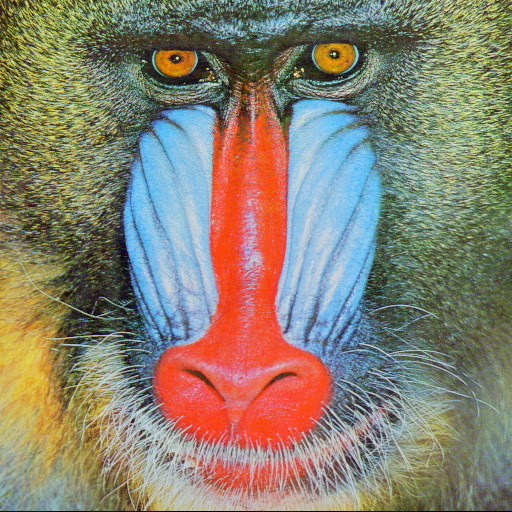}
\includegraphics[width=.24\textwidth]{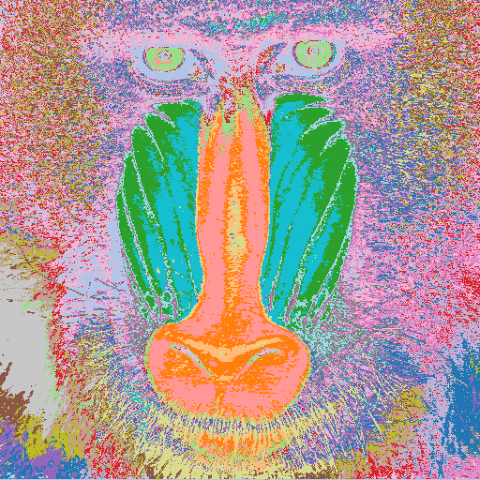}\\[2mm]
\includegraphics[width=.24\textwidth]{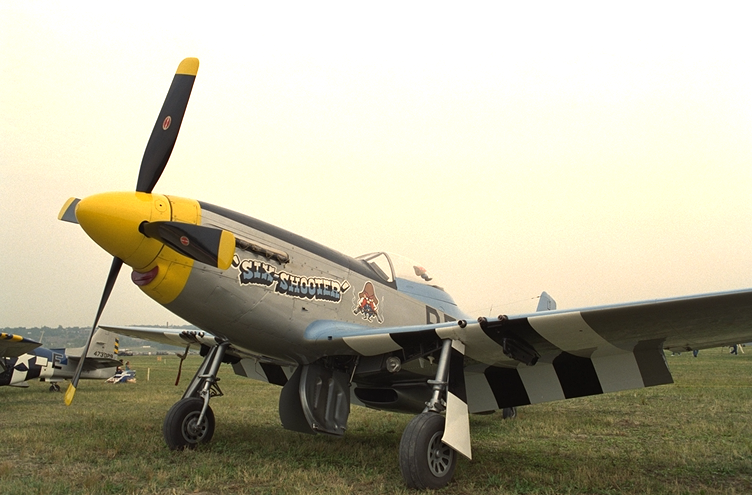}
\includegraphics[width=.24\textwidth]{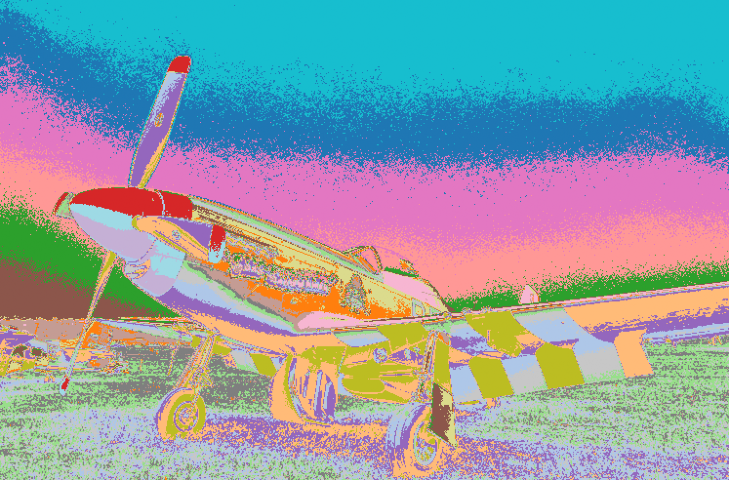}
\caption{Images of the mandrill and the airplane with corresponding labelings, used for training and evaluating the expressivity of the sigma flow method.}
\label{fig:img1}
\end{wrapfigure}

\textbf{Data.} 
We used the images displayed in Figures \ref{fig:img1} and \ref{fig:img2} to test the sigmaflow architecture for various image labeling tasks. The label configurations are displayed using a color code. The figures display the RGB data of the images we used as well as labelings of these images produced by means of $k$-means clustering of the pixel values in CIELAB color space. The image of an airplane is the sample kodim20 from the Kodak database \url{https://r0k.us/graphics/kodak/}. The image of a mandrill was taken from the USC-SIPI Image Database \url{https://sipi.usc.edu/database/database.php}. The images of the buildings taken from different viewpoints and displayed in Figure \ref{fig:img2}, were taken from the Zurich Buildings Database \url{https://vision.ee.ethz.ch/datsets.html}. The original RGB values were modified via a SLIC segmentation method, to simplify the textures in the data and suppress high spatial frequencies. 
\begin{figure}
\includegraphics[width=\textwidth]{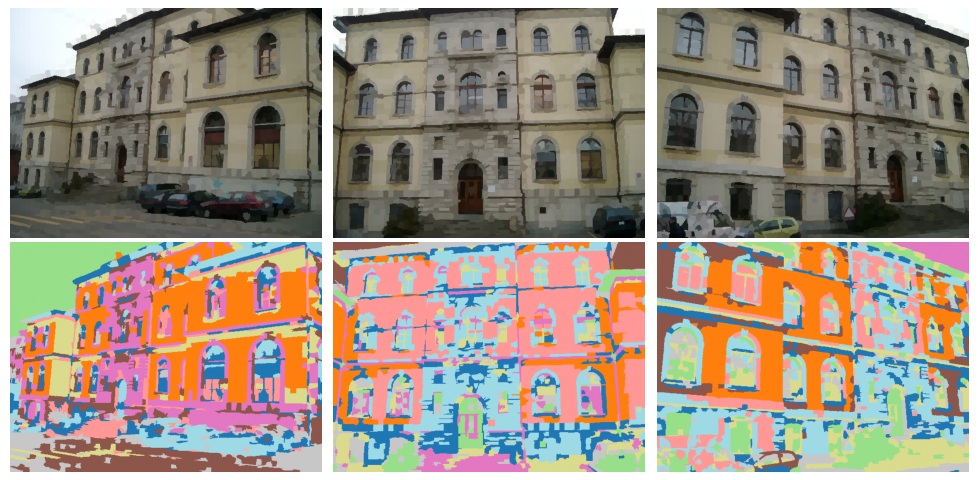}
\caption{Simplified images of buildings with corresponding labelings, used for training the sigma flow approach and an UNet in order to compare expressivity and generalization to novel viewpoints.}
\label{fig:img2}
\end{figure}

\vspace{0.25cm}
\textbf{Parametrization of the metric $\fh$.} 
We specify the parametrization of the positive definite $2 \times 2$ matrices $\fh(t)_a$ (cf.~\eqref{eq:ha-2x2}). For simplicity, we drop the subscript $a$ indicating the dependency on the grid points that are used to discretize the domain manifold $M$. We reuse below the symbol $a$ with a different meaning confined to this section. 

The three degrees of freedom of the metric tensor are given by $x,y,z \in \R$. The inverse metric tensor bounded from below by $0.1 \II_2$ has the form
\begin{equation}
    \fh^{-1} = \frac{1}{v} \mqty(
        a & b \\ 
        b & c 
    ),
\end{equation} 
where the parameters $a,b,c$ must satisfy the relation $a+c>0$ and $ac-b^2 > 0$. We parametrized $a,b,c,v$ in terms of functions $\lambda(x),\alpha(y)$ and $v(z)$ in terms of
\begin{equation}
    a = \lambda + \sin^{2}(\alpha) \Delta, 
    \qquad 
    b = \lambda \sin(\alpha)\cos(\alpha),
    \qquad
    c = \lambda + \cos^{2}(\alpha) \Delta, 
    \qquad 
    \Delta = \frac{1}{\lambda} - \lambda.
\end{equation} 
with $\lambda(x) > 0.1, \alpha(z) \in \R$ and $1/v(z) > 1$. This parametrization guarantees that the resulting matrix $\fh^{-1}$ is symmetric positive definite and its smallest eigenvalue is not smaller than 0.01. The functions $\lambda, \alpha, v$ were varied for the different experiments.

For the experiments reported in Section \ref{sec:2-universal-approximation-behavior-of-sigma-flow}, we applied this function to a parameter matrix of size $\R^{N \times 3}$, described below as part of the neural network architecture, along the last dimension to obtain a metric. For the parameter functions, we chose for the simple label configuration
\begin{equation}
    \lambda(x) = \tilde\sigma(x) + 0.5, \quad \cos^2(\alpha(y)) = \sigma(y), \quad 1/v(z) = 0.5*\sigma(z) + 0.5 
\end{equation}
where $\sigma(x) = e^x / (1+e^x)$ the sigmoid and $\tilde\sigma(x) = \log(1 + e^x)$ the softplus function. For the more complex label configuration we used
\begin{equation}
    \lambda(x) = \tilde\sigma(x) + 1, \quad \cos^2(\alpha(y)) = \sigma(y), \quad 1/v(z) = \sigma(z) + 0.1 
\end{equation}

For the experiments reported in Section \ref{sec:3-learning-a-mapping-mco-from-simple-patterns-and-generalizing-to-complex-data}, we trained a neural network mapping $\R^{N \times d}$ to $\R^{N\times 3}$. The outputs of the network were then interpreted as parameters $x,y,z$ of a metric tensor, which we then transformed with the prescription to transform them to a positive symmetric matrix via the functions
\begin{equation}
    \lambda(x) = 1 - 0.9*\tanh(|x|) , \quad \alpha(y) = \frac{\pi}{2}\tanh(y), \quad v(z) = 1 - 0.9*\tanh(|z|).
\end{equation}

\vspace{0.25cm}
\textbf{Numerics.} For assessing the convergence behavior of the sigma flow (Section \ref{sec:1-synthetic-benchmark-for-convergence-behavior}), we numerically solved the ODE problem corresponding to the semi-discrete sigma flow \eqref{eq:semidiscrete-sigma-model} using a geometric Runge-Kutta scheme of order five with adaptive step size implemented via the Python libraries Diffrax \cite{kidger2021on} and Jax \cite{jax2018github}, and with the standard 5-point stencil for discretizing the flat Laplacian $\Delta$ on $\TT^{2}$. 

For the computationally more expensive experiments reported in Sections \ref{sec:2-universal-approximation-behavior-of-sigma-flow} and \ref{sec:3-learning-a-mapping-mco-from-simple-patterns-and-generalizing-to-complex-data}, we used the simplest geometric Runge-Kutta scheme, i.e.~geometric Euler integration, to speed up the optimization. For both experiments, we used the Adabelief first-order optimizer. 

Regarding the simple label configurations considered in Section \ref{sec:2-universal-approximation-behavior-of-sigma-flow}, we optimized over 2000 steps with learning rate 0.01, using a cosine decay learning rate schedule, end time $T=3$ and step size 0.2, and mass parameter $m^2 = 1$. The metric $\fh$ was parametrized by a single $512\times 512 \times 3$ tensor. For the experiment reported in Section \ref{sec:3-learning-a-mapping-mco-from-simple-patterns-and-generalizing-to-complex-data}, we used the end time $T=2$ with step size $0.5$ and a \textit{learnable mass parameter $m^2$} initialized to 1. We trained both the UNet model and the sigma flow model with the Adabelief optimizer and the learning rate $10^{-4}$ taking 15000 training steps, with a batch size of 2.

\vspace{0.25cm}
\textbf{Noisy labelings.} 
The corrupted labelings for the experiments in Section \ref{sec:3-learning-a-mapping-mco-from-simple-patterns-and-generalizing-to-complex-data} were created by first applying label smoothing with smoothing factor 0.8, followed by applying the logarithm, addition of noise sampled from the standard normal distribution $\mc{N}(0,0.2)$, followed by a linear projection to the unit cube and applying softmax. For visualization, we rounded pixelwise to the label with highest probability and applied a color code.

\vspace{0.25cm}
\textbf{Neural networks.} 
In the experiments of Section \ref{sec:3-learning-a-mapping-mco-from-simple-patterns-and-generalizing-to-complex-data}, we employed two neural network configurations.

We trained an UNet\footnote{
    We took the implementation of the UNet from \url{https://github.com/Benjamin-Walker/diffrax-diffusion-models}
} \cite{ronneberger2015u} as a naive baseline (in the sense that we did not tune or modify this architecture) for comparison to our sigma flow architecture. We chose an UNet as a generic reference model, since it is known to work well for segmentation tasks, while being flexible enough to accept different image resolutions. The reference configuration operated at hidden dimension 16, with two downsampling, two upsampling and two middle blocks, comprising overall 270.716 parameters.

\begin{wrapfigure}[17]{r}{0.3\textwidth}
\centering
\includegraphics[width=.3\textwidth]{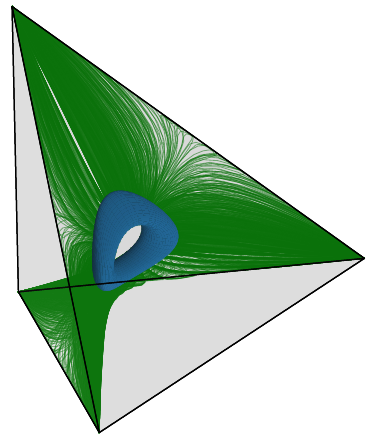}
\caption{Visualization of the torus inside $\S_4$ using the chart $\chi_{m}$ and the flow lines of points on the torus under the sigma flow time evolution.}
\label{fig:torus-a}
\end{wrapfigure}

We compared this baseline model with a parametrized version of the sigma flow architecture. The sigma flow was parameterized with a neural network containing a single convolutional layer with kernel size 15 $\times$ 15 and 64 filters padded so as to preserve the number of vertices in the discretized tori. Following the convolutional layer, we applied a multilayer perceptron with one hidden layer of dimension 64 and output dimension 3 at each pixel separately. Using the 3 dimensional output vectors, the metric $\fh^{-1}$ was constructed by applying the function $\R^{3} \to \R^{2 \times 2}$ described above. Alltogether, the sigma flow model contained 306.948 trainable parameters.

\subsection{Convergence Behavior}\label{sec:1-synthetic-benchmark-for-convergence-behavior}
We demonstrate the convergence behavior of the sigma flow model using the following set-up. The torus $M = \mathbb T^{2}$ serves as base manifold, with the standard flat metric $h = \delta$ and the simplex $\S_{4}$ as the target manifold. 
We used an initial configuration $\vP$ given by 
\begin{equation}
    \vP : \TT^{2} \to \S_{4} , \qquad 
    (\xi_{1},\xi_{2}) \mapsto 
    \mathrm{sm}
     \begin{pmatrix}
        x = 0.2(3+\cos\xi_{1} )\cos \xi_{2}  \\
        y = 0.2(3 + \cos \xi_{1}) \sin \xi_{2}  \\
        z = 0.2 \sin \xi_{1}  \\ 
        x + y + z
    \end{pmatrix},
\end{equation}
where $\mathrm{sm}$ denotes the softmax function $\R^{4} \to \S_4$ defined by \eqref{eq:softmax}. Figure \ref{fig:torus-a} shows a plot of $\chi_{m} \circ \mfk{P}$; see the blue submanifold from which the flow lines emanate. 

\begin{figure}
    \begin{subfigure}{.45\textwidth}
        \centering
        \includegraphics[width=\textwidth]{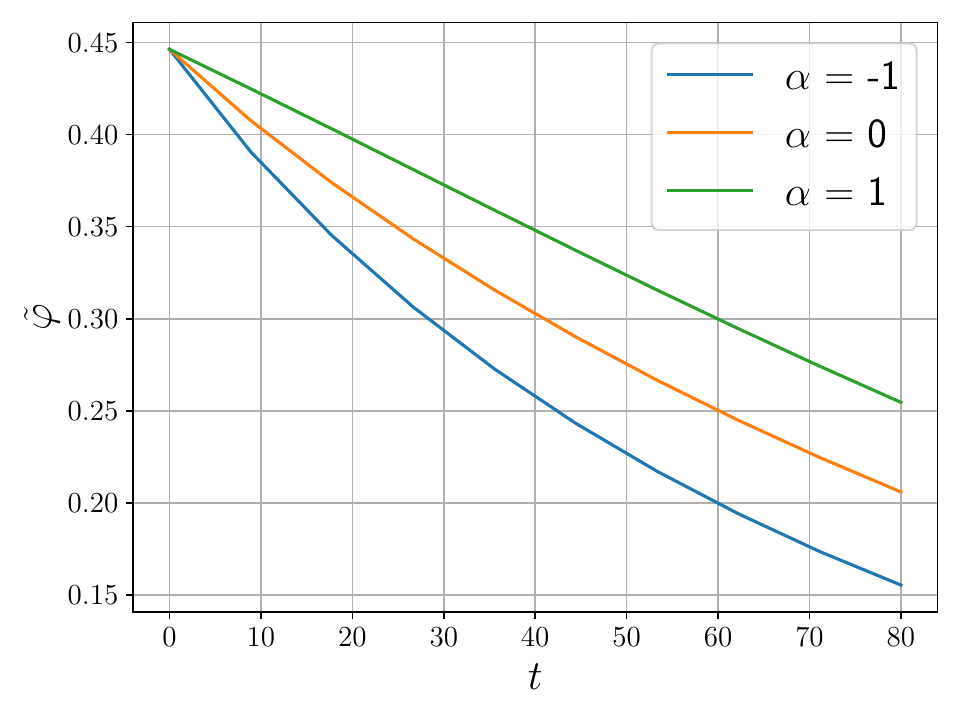}
        \caption{Speed of convergence to the center for different values of $\alpha$ and $m=0$. $\tilde{\varphi}$ is $\varphi$ normalized to the range $[0,1]$.}
        \label{fig:torus-e}
    \end{subfigure}
    \hfill
    \begin{subfigure}{.45\textwidth}
        \centering
        \includegraphics[width=\textwidth]{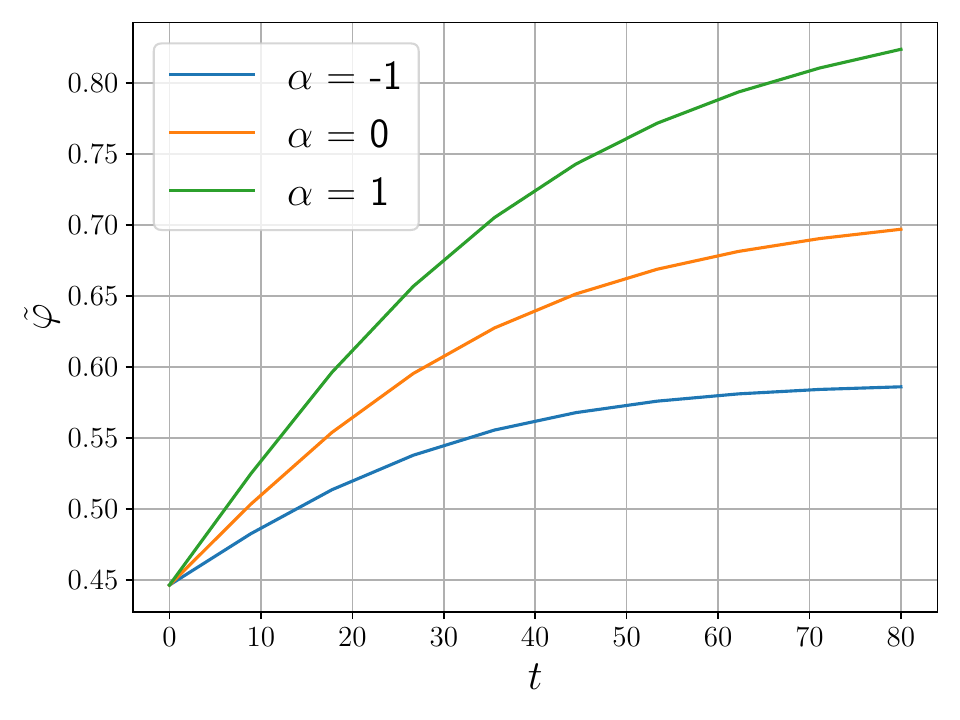}
        \caption{Speed of convergence to the boundary for different values of $\alpha$ and $m=0.1$. $\tilde{\varphi}$ is $\varphi$ normalized to the range $[0,1]$.}
        \label{fig:torus-f}
    \end{subfigure}
\caption{Illustration of the convergence behavior of the sigma flow for a low dimensional toy example.}
\label{fig:torus}
\end{figure}

We first examined with the solution to the PDE system 
\begin{equation}
\pdv{t}\mfk{S}_{t} = \tau(\mfk{S}_{t},\fr,\delta), \quad \mfk{S}_{0} = \mfk{P},
\end{equation}
corresponding to the special case $m^{2}=0$ of the sigma model, which was discretized and numerically solved as explained in Section \ref{sec:Implementation}. The torus $\TT^{2} = S^{1} \times S^{1}$ was discretized by a regular grid with periodic boundary conditions. 

$\mfk{S}_{t}$ converges to a point for large integration times which corroborates the statement of Proposition \ref{prop:lyapunov}. We point out that we never had to resort to the regularized metric \eqref{eq:def-fr-eps}, that is the flow computation converged for $\varepsilon = 0$.  
We repeated the experiment with $m^{2} = 1$. The trajectories emanating from the initial toroidal surface are shown by Figure \ref{fig:torus-a}. All integral curves of the flow lie approximately on the \textit{Wright manifold} \cite[Sec.~18.8]{hofbauer1998evolutionary}, a 2 dimensional submanifold of $\S_4$. This illustrates that solutions to the sigma flow equations are capable of preserving non-trivial relations during time evolution.
Although no metric regularization $\varepsilon$ was used, convergence to the boundary of the simplex was always observed. Varying $\alpha$ did not change the convergence behavior, but the speed of convergence as shown by the Figures \ref{fig:torus-e} and \ref{fig:torus-f}, respectively. 

\newpage
\subsection{Expressivity of Sigma Flow Model with a Constant Metric}\label{sec:2-universal-approximation-behavior-of-sigma-flow}
\begin{wrapfigure}{r}{.5\textwidth}
    \centering
    \includegraphics[width=.4\textwidth]{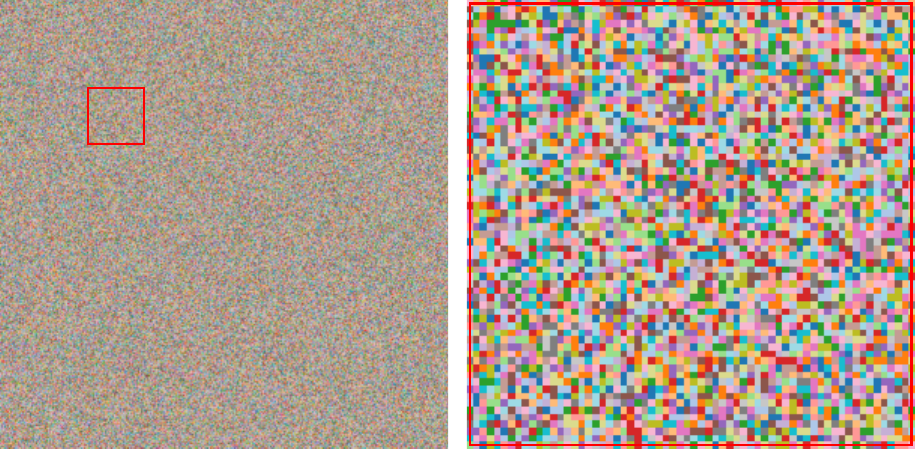}
    \caption{Initial configuration $\vP$ (left) with close-up view of the patch marked with red (right).}
    \label{fig:capacity-noise}
\end{wrapfigure}

We studied the capacity of the entropic sigma flow \eqref{eq:hsigmaflow} 
\begin{equation}
\pdv{t} \mfk{S}= \tau_{\tilde{\varphi}}(\mfk{S},\fr,h),\qquad
\mfk{S}_{0} = \mfk{P}
\end{equation}
to parametrize a labeling
\begin{equation}
    \label{eq:def-vL}
\mfk{L}: \TT^{2}\to[20]
\end{equation}
using a \textit{fixed} metric $h$ on $M=\TT^{2}$, as a basis for predicting such mappings in terms of a corresponding metric for novel given data (Section \ref{sec:3-learning-a-mapping-mco-from-simple-patterns-and-generalizing-to-complex-data}).

To this end, we solved the PDE-constrained optimization problem
\begin{equation}\label{eq:h-learning-problem}
\min_{h} \int_{\TT ^{2}}\mrm{KL}[\mfk{S}_{T}(x): \mfk{L}(x)] \, \dd^{2} x \qquad \text{s.t.} \qquad
\mfk{S} \in C^{\infty}(M_{T},\S_{c}), \quad \pdv{t} \mfk{S}_{t} = \tau_{\tilde{\varphi}}(\mfk{S}_{t},\fr,h), \quad 
\mfk{S}_{0} = \mfk{P},
\end{equation}
for a fixed initial configuration $\vP$ (see Figure~\ref{fig:capacity-noise}). The optimization was performed using gradient descent and backpropagation. We repeated the optimization for two different target configurations $\mfk{L}$ depicted by Figure~\ref{fig:targets}. The label configurations are visualized with a color code.

Figure~\ref{fig:capacity-prog} shows the 
resulting state $\vS_T$ at different stages of the optimization, 
demonstrating that the desired configuration has been gradually achieved. For the simple label configuration, we observed a nearly exact reproduction of the target pattern: only three pixels were labeled incorrectly. For the complex label configuration, we measured $4.4\%$ mislabeled pixels. This is however not visible to the naked eye, since all the incorrectly labeled pixels are located in image regions with high spatial frequency, cf.~Figure \ref{fig:bin-mask}, where one cannot identify and recognize the individual pixels.

Overall, these results validate that the sigma flow architecture is expressive enough to encode labelings comprising a broad range of spatial scales and local configurations of label transitions.

\begin{figure}[t]
    \centering
    \begin{subfigure}[t]{0.6\textwidth}
    \includegraphics[width=.49\textwidth]{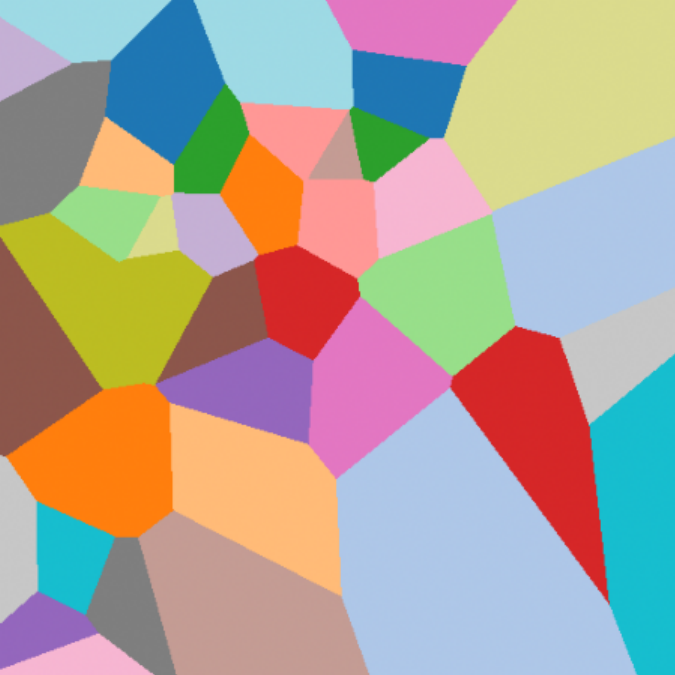}
    \hfill
    \includegraphics[width=.49\textwidth]{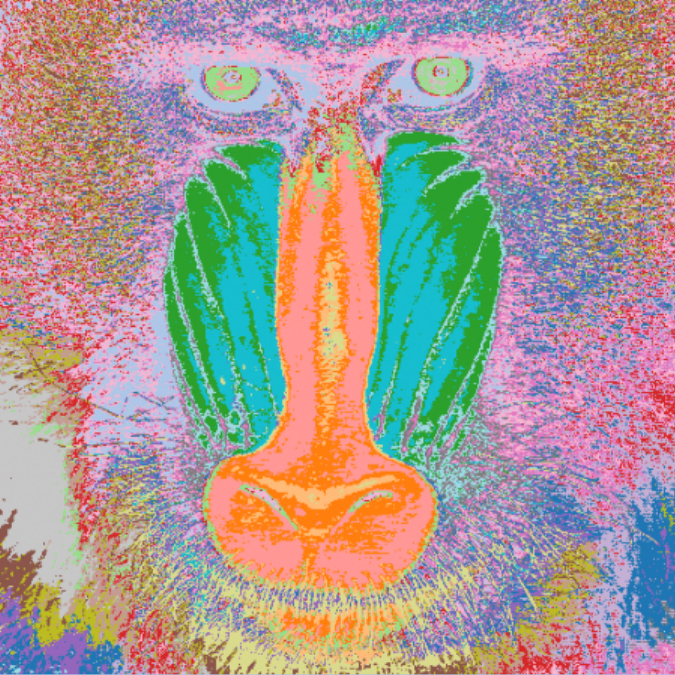}
    \caption{
        \footnotesize{Target configurations to be generated by the sigma flow architecture comprising a constant metric tensors, separately learned for each target configuration. The result for the simple label configuration, consisting of large-scale polyhedral patches, is nearly perfect; only 3 pixels are mislabeled after the optimization procedure. For the complex labeling of the mandrill image, the final labeling is slightly erroneous, with 4.4 \% mislabeled pixels located in the highly textured regions. This errors, therefore, are not visible.
        }
    }
    \label{fig:targets}
    \end{subfigure}
    \hspace{2mm}
    \begin{subfigure}[t]{0.3\textwidth}
    \centering
    \includegraphics[width=.98\textwidth]{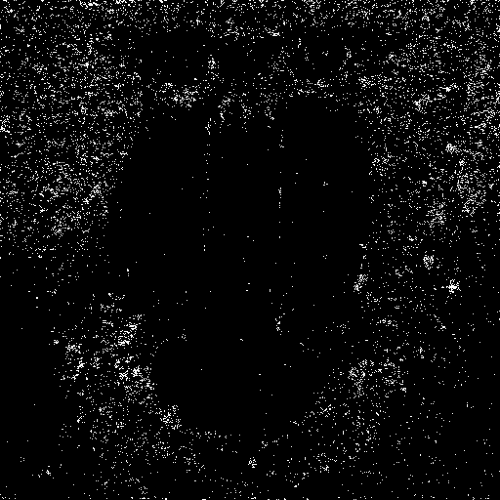}
    \caption{
        \footnotesize{Binary mask of mislabeled pixels in the mandrill example. Errors only occur in the regions of high spatial frequency, such that the final labeling is indistinguishable by the human eye.}
    }
    \label{fig:bin-mask}
    \end{subfigure}\\[2mm]
    \begin{subfigure}[t]{\textwidth}
    \includegraphics[width=\textwidth]{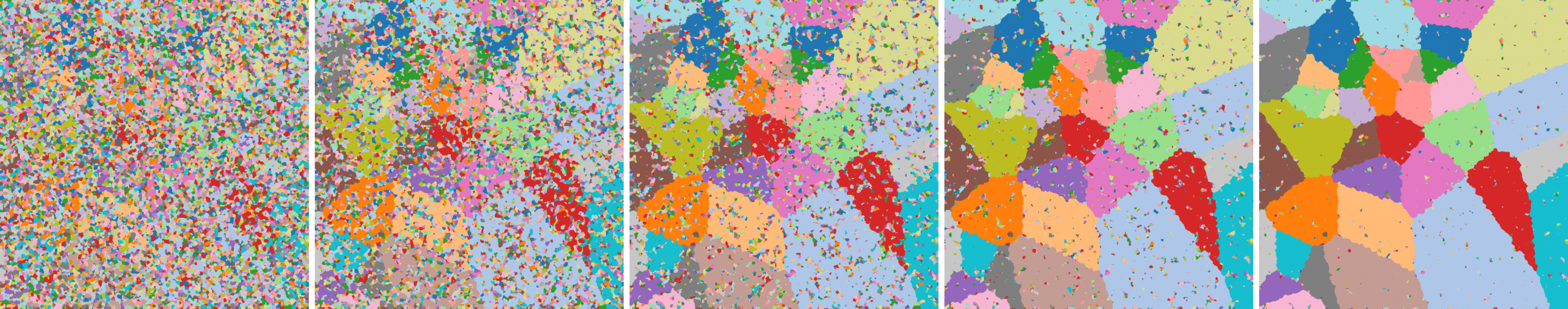}
    \caption{Evolution of $\vS_T$ after subsequent 100 gradient descent steps. This illustrates how the sigma flow approach learns to represent labeled regions at larger spatial scales and with arbitrary orientations in terms of a constant metric and geometric diffusion.
    \\[4mm]}
    \label{fig:capacity-prog}
    \end{subfigure}
    \caption{Validation of the expressivity of the sigma flow architecture: representation of labelings after learning a constant metric that comprise different spatial scales and local label transitions.}
\end{figure}

\subsection{Learning Prediction of Labelings in Terms of a Time-Variant Metrics}
\label{sec:3-learning-a-mapping-mco-from-simple-patterns-and-generalizing-to-complex-data}
The experiments reported in this section concern the \textit{learned} operator 
\begin{equation}\label{eq:mcO-to-be-learned}
    \mc{O} : C^{\infty}(\TT^{2} \times [0,T],\S_{20}) \to C^{\infty}(\TT^{2},\R^{2 \times 2})
\end{equation}
which parametrizes all variants of the sigma model in terms of the metric $h_{t} = \mc{O}(\vS_t,t)$ (Figure \ref{fig:top-level-view}). Here the dependency of the operator $\mc{O}$ on both the state $\vP : \TT^{2} \to \S_{20}$ and the time $t \in [0,T]$ was taken into account. 

We selected a set of training data 
\begin{equation}
    \label{eq:ts}
\mrm{TS} = \{ \mfk{L} : \TT^{2} \to [20] \}
\end{equation}
consisting of pointwise labeled (segmented) tori with a \text{random} polygonal spatial structure. A color-coded sample from the training set is depicted by Figure \ref{fig:trainset-patches}. We repeated the same experiment with a different training set, consisting of images patches of labeled buildings from Figure \ref{fig:img2}, to provide also a test with more complex label configurations. See Figure \ref{fig:house-train} for a sample from this second training set. Here only patches of images taken from two out of three different viewpoints were used for training. The third unseen view was used for testing.
The objective of training was to learn a \textit{prediction map} \eqref{eq:mcO-to-be-learned} for the recovery of the ground truth labeling from a severely corrupted input signal. We denote by $\mathrm{Corr}$ the distribution of corruptions applied to the data. Figure~\ref{fig:learning-res} shows an example of a corrupted sample from the training set in the leftmost panel. 

We used a small neural network to define and parametrize a class of operators 
\begin{equation}\label{eq:nn-function-class}
\mc{C} \subset C^{\infty}(\TT^{2} \times [0,T],\S_{20}) \to C^{\infty}(\TT^{2},\R^{2 \times 2}). 
\end{equation}
The learning task was formulated as the PDE-constrained optimization problem
\begin{equation}\label{eq:Corr-objective}
    \min_{\mc{O} \in \mc{C}} \sum_{\mfk{L} \in \mathrm{TS}} 
        \mathbb{E}_{\sigma \sim \mrm{Corr}}\Bigg[\int_{\TT ^{2}} \mathrm{KL}[\vS_T(x):\mfk{L}(x)] \, \dd^{2}x \Bigg]
        \;\qq{s.t.}\;
        \pdv{t} \vS_t = \tau_{\tilde{\varphi}}\big(\vS_t,\fr, \mc{O}(\vS_t,t)\big), 
        \quad \vS_0 = 
        \sigma(\mfk{L}).
\end{equation} 
The learned operator $\mc{O}$ is supposed to accomplish the following. If $\vS_0$ is a corrupted initial configuration, then $\vS_T$ should recover the uncorrupted configuration. In particular, this should hold for \textit{unseen} and \textit{independently sampled random} test data from the same image class. The optimization problem was solved through backpropagation and gradient descent. As a baseline for comparison, we trained an UNet of roughly similar model size for the same task. Results of the test phase are shown by Figure \ref{fig:learning-res} for the Voronoi labelings and by Figure \ref{fig:house-res} for the labeling of buildings. Both models achieved the restoration task to a reasonable extent.\footnote{
    Note that training in the two scenarios result in two different models. Each model can only solve one of the two cases and therefore we trained four models: one UNet for Voronoi labelings and another UNet for the labelings of building. Accordingly, we trained two sigma flow models for these two scenarios.
}

\begin{wrapfigure}{r}{.6\textwidth}
    \centering
    \begin{subfigure}{.6\textwidth}
        \includegraphics[width=\textwidth]{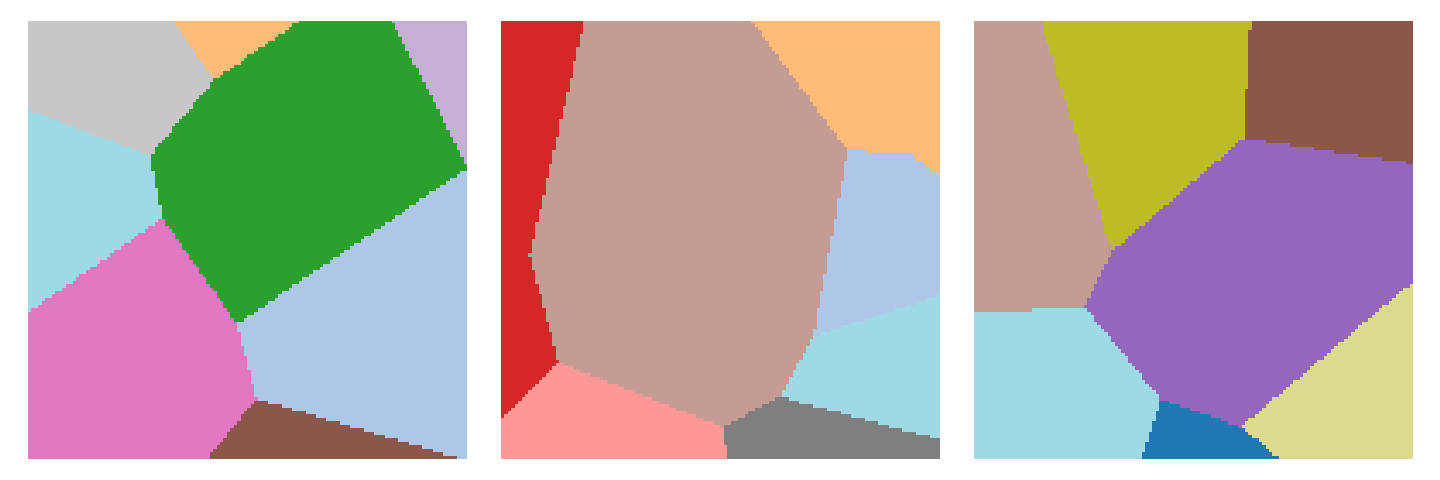}
        \caption{Samples from the first training set containing labeled images of size 128$\times$128 pixels, with randomly distributed Voronoi regions.\\[2mm]}
        \label{fig:trainset-patches}
    \end{subfigure}
    \begin{subfigure}{.6\textwidth}
        \includegraphics[width=\textwidth]{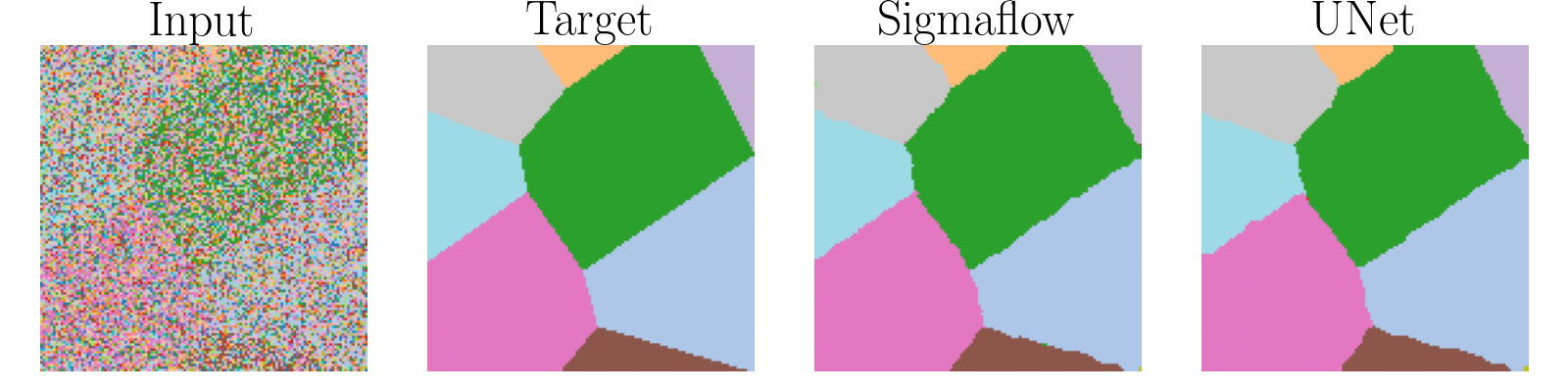}
        \caption{
            Training results for the sigma flow and the UNet model.
        }
        \label{fig:learning-res}
    \end{subfigure}
    \begin{subfigure}{.6\textwidth}
        \includegraphics[width=\textwidth]{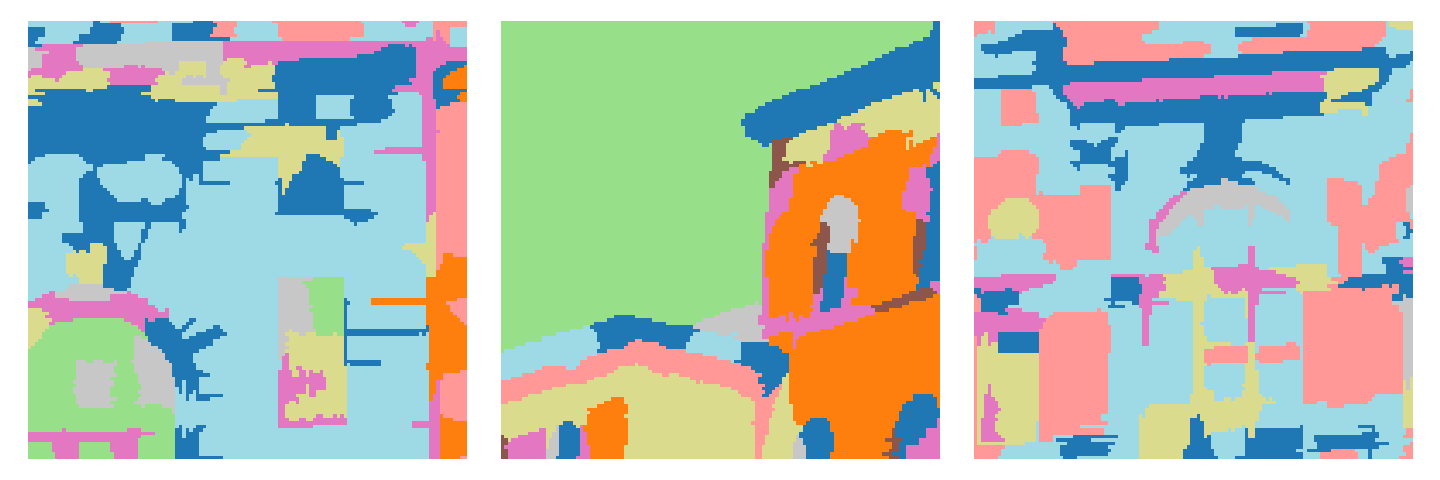}
        \caption{Samples of the training set for the images of building, each of size 128$\times$128 pixels, extracted as patch from the larger labeled images from Figure \ref{fig:img2}. Only images taken from two out of the three different 3D viewpoints were used for training. \\[2mm]}
        \label{fig:house-train}
    \end{subfigure}    
    \begin{subfigure}{.6\textwidth}
        \includegraphics[width=\textwidth]{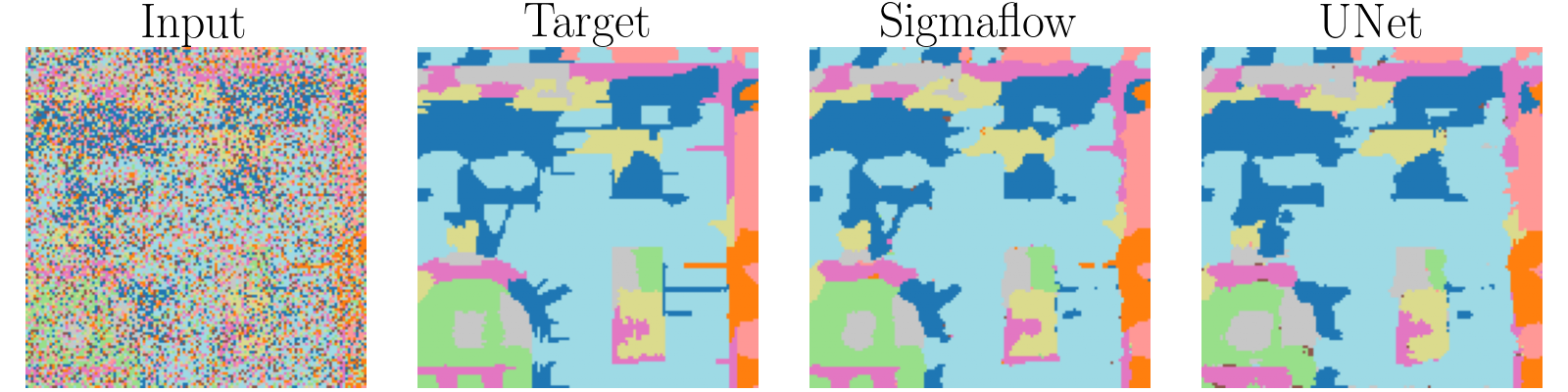}
        \caption{Training results for the sigma flow and the UNet model.}
        \label{fig:house-res}
    \end{subfigure}
    \caption{Training setup and results for sigma flow and UNet architectures.}
    \label{fig:training-data}
\end{wrapfigure}
We then tested the generalization capability of both models. As shown by Figure \ref{fig:learning-gen}, we confronted the models that were trained on labeled Voronoi-structured  \textit{patches}, with larger Voronoi-structured images of size $512\times 512$ pixels and similar spatial scales. Additionally, we applied a different normalization to the corrupted data: during training, the data were normalized to the unit cube, whereas normalization to the unit sphere was used in the test phase. 

We observed degradation of labeling results for both models (Figure \ref{fig:test-sigma-UNet}). The sigma flow produced fewer artifacts in this test case, however. Erroneous predictions of the UNet model were mainly caused by the changed normalization of the inputs. Using larger images as inputs, rather than patches for training, only slightly decreased the prediction performance of both models, but the change of normalization affected notably the UNet model. Training produced about 99\% percent correctly labeled pixels for both models, which remained at this level for both models when they were tested on the larger images. Regarding normalization, the performance of the sigma flow model stayed at 99\% after changing the normalization, whereas the performance of the UNet dropped down to 92\%.

We then tested both models on data shown by Figure \ref{fig:learning-air} comprising different spatial scales, quite dissimilar to the training data depicted by Figure \ref{fig:training-data}. We changed again the normalization from the unit cube to the unit sphere and observed the same effect on the prediction performance of both models, as reported above. When using the unseen airplane picture with unit cube normalization as input for both models, the sigma flow dropped to 79\% correctly labeled pixels and the UNet model to 76\%. After changing the normalization to the unit sphere, the sigma flow dropped to 77\% and the UNet to 68\%. 

Finally, we tested the prediction performance of both models in the second scenario with images of buildings. Figure \ref{fig:gen-house} shows how the models generalized to larger inputs and different normalization. The training rate of correct labelings for both models is about 94\% correctly labeled pixels. When confronted with the whole image instead of a smaller patch like during training, both models kept this performance level. After changing the normalization, the sigma flow model kept its performance, whereas the prediction rate of the UNet dropped to 84\%. Figure \ref{fig:gen-house-new} shows how the models performed on unseen inputs which resemble the training data, but differ due to a changed 3D viewpoint. 
The prediction rate of the sigma flow remained stable at 94\% whereas the performance of the UNet dropped to 91\%. After changing the normalization from the unit cube to the unit sphere, the performance of the UNet further dropped down to 85\%. Figure \ref{fig:gen-house-new} shows the corresponding predictions of both models.

Summing up, we conclude that in this scenario, the sigma flow architecture produced a more labeling predictor than the UNet baseline model, mainly due to high sensitivity of the UNet model to input normalization. We argue that the geometric diffusion process underlying the sigma flow architecture helps to prevent such instabilities, since the diffusion always operates in terms of relative quantities and hence provides some form of slack for changes of global scale. Furthermore, the diffusion based modelling paradigm of the sigma flow equips the architecture with a robust prior for dissipative noise removal, which can be leveraged in problems such as the one presented in these experiments.

\begin{figure}
    \begin{subfigure}{\textwidth}
        \centering
        \includegraphics[width=\textwidth]{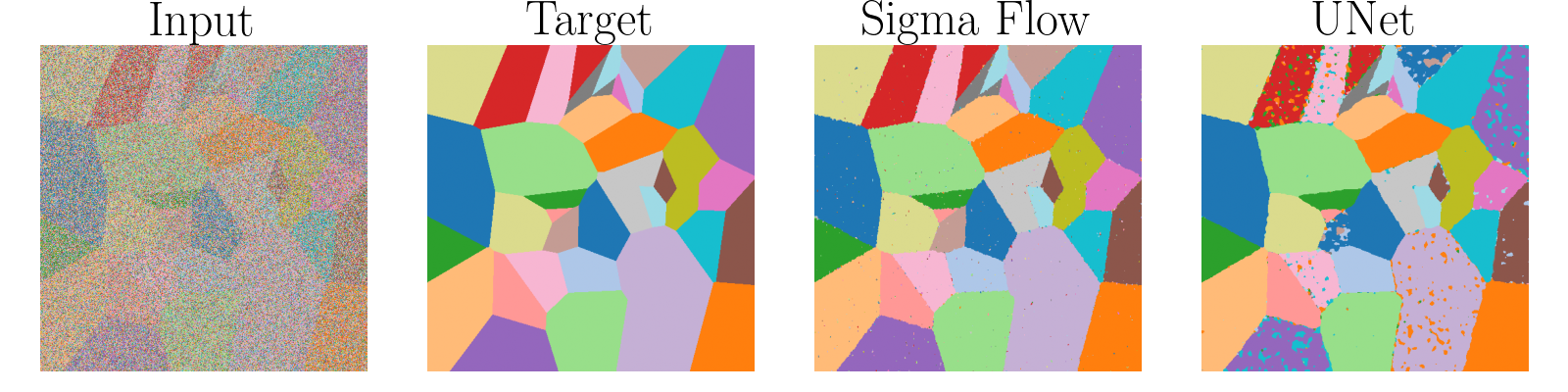}
        \caption{
            Results of the sigma flow model and the UNet for unseen test data of larger size 512$\times$512, with the same level of noise but a different input normalization in comparison to the training phase (see text). \textit{Left to Right:} Perturbed input data, target labeling (ground truth), result of the sigma flow model, result of the UNet model. The performance of prediction degrades in comparison to training performance. The sigma flow model with a learned metric achieves a better error rate, however.
        }
        \label{fig:learning-gen}
    \end{subfigure}
    \vspace{2mm}
    \begin{subfigure}{\textwidth}
        \centering
        \includegraphics[width=\textwidth]{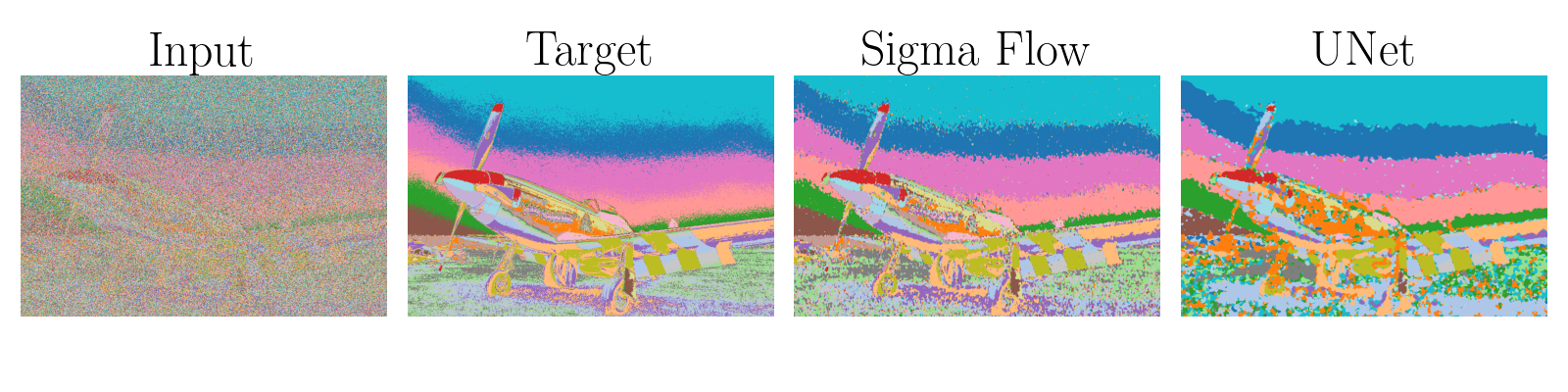}
        \caption{
            Same set-up as (a), yet now using a dissimilar unseen test image in comparison to the training data. \textit{Left to Right:} Perturbed input data, target labeling (ground truth), result of the sigma flow model, result of the UNet model. We observe the same trend as in Figure \ref{fig:learning-gen}: both models degrade, but the sigma flow model achieves a better error rate.
        }
        \label{fig:learning-air}
    \end{subfigure}
    \begin{subfigure}{\textwidth}
        \includegraphics[width=\textwidth]{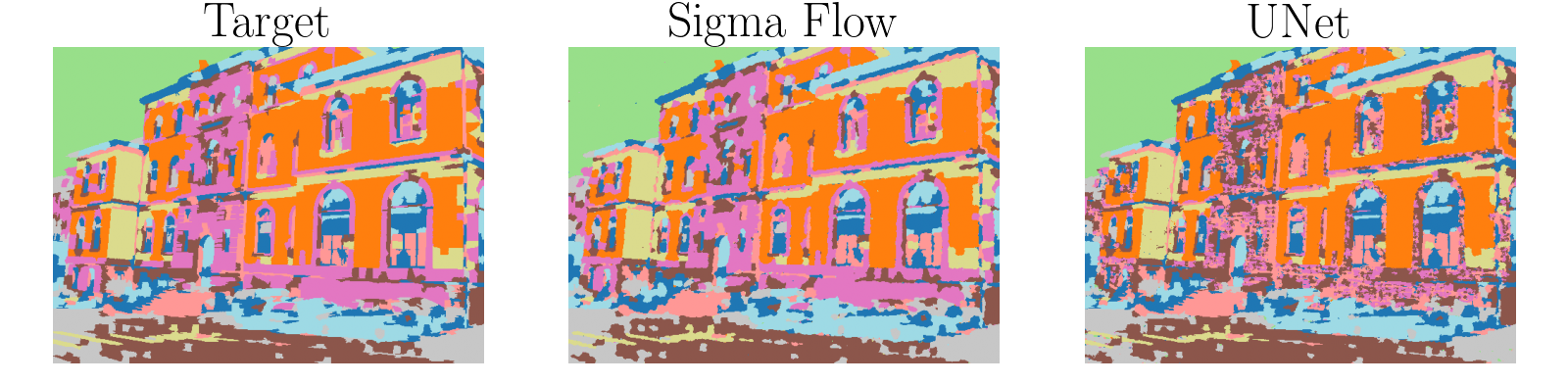}
        \caption{
            Test evaluation of the two models for the building scenario using larger input images from the training set (rather than patches in the training phase) and different input normalization (see text). \textit{Left to Right:} Perturbed input data, target labeling (ground truth), result of the sigma flow model, result of the UNet model.\\[2mm]
        }
        \label{fig:gen-house}
    \end{subfigure}
    \begin{subfigure}{\textwidth}
        \includegraphics[width=\textwidth]{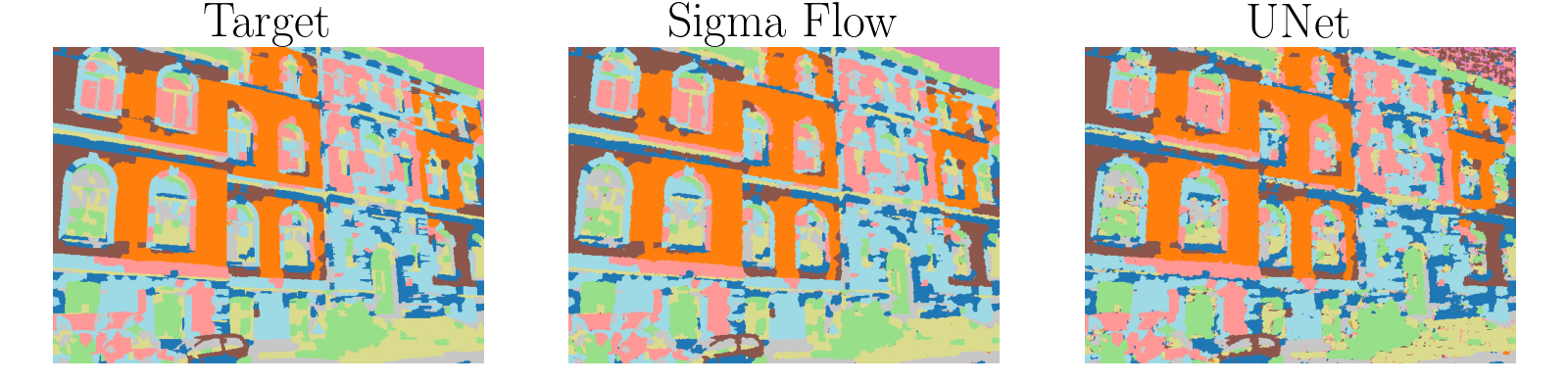}
        \caption{
            Same set-up as (c), yet now using an unseen image taken from a different 3D viewpoint. \textit{Left to Right:} Perturbed input data, target labeling (ground truth), result of the sigma flow model, result of the UNet model. Regarding the error rate, the sigma flow model is on par with or better than the UNet model.
        }
        \label{fig:gen-house-new}
    \end{subfigure}
    \caption{Test evaluation for sigma flow and UNet architectures.}
    \label{fig:test-sigma-UNet}
\end{figure}

\section{Conclusion}\label{sec:Conclusion}
\subsection{Summary}
\paragraph{Sigma flow model.} This paper introduced the sigma flow model for image and metric data labeling on graphs. The model is based on a generalized harmonic energy as objective function between a Riemannian domain and the target manifold, respectively. Geometric integration of the Riemannian gradient flow optimizes the mapping. The flow is called `sigma flow' because the approach resembles the mathematical structure of sigma models known in quantum field and string theory.

Specific choices of the domain and target manifold and the generalized harmonic energy yield variants of the sigma flow approach. We specifically focused on image data on a two-dimensional domain manifold and the probability simplex equipped with the Riemannian Fisher-Rao metric. This variant of the sigma flow model combines the Laplace-Beltrami approach to image denoising and enhancing introduced by Sochen, Kimmel and Malladi about 25 years ago \cite{sochenGeneralFrameworkLow1998} and the assignment flow apporach to metric data labeling developed by the authors. We proved that this sigma flow model constitutes a proper nonlinear geometric diffusion approach such that its variant based on the generalized entropic harmonic energy constitutes a proper labeling approach.

\paragraph{Two-stage parametrization of structured prediction via sigma flows.}
A remarkable feature of our geometric approach is the chain of self-referring time-variant parametrizations of large-scale structured prediction in terms of sigma flows, as sketched by Figure \ref{fig:top-level-view} and repeated here for the reader's convenience:
\begin{equation}\label{eq:sigma-flow-parametrization}
\dotsb\; \mfk{S}_{t}\;\xlongrightarrow{\mc{O}}\; h_{t}\;\xlongrightarrow{\Delta_{h_{t}}}\; \tau(\mfk{S}_{t},g,h_{t}) \;\xlongrightarrow{\int \dd t}\; \mfk{S}_{t}\;\dotsb 
\end{equation}
The tension field $\tau$ which governs the evolution of the state $t\mapsto \mfk{S}_{t}$ is parametrized by the Laplace-Beltrami operator $\Delta_{h_{t}}$ that itself is parametrized by the Riemannian metric $h_{t}$ of the domain manifold. We showed that by making this metric dependent in terms of a mapping $\mc{O}$ on both given data and the evolving state, our sigma flow model covers a range of established nonlinear PDE models of mathematical image analysis. In particular, the mapping $\mc{O}$ can be parametrized by a neural network whose parameters can be conveniently learned from data, due to the inherent smoothness of our geometric approach and the robust numerics used for geometric integration of the sigma flow.

\paragraph{Expressivity of sigma flows and learning the generator from data.}
We demonstrated the remarkable expressivity of the sigma flow: any image structure can be generated from pure white noise by choosing properly the domain metric $h$, i.e.~the  parametrization of the generator $\Delta_{h}$ (Laplace Beltrami operator) of the generator $\tau(\Delta_{h})$ (tension field) of the sigma flow. This suggests to study the sigma flow model from the viewpoint of machine learning, since the aforementioned succinct mathematical representation of structured prediction by nonlinear sigma flows should enable strong task-specific adaptivity by using a compact set of parameters learned from data. 

We briefly demonstrated this property using 
a fairly small neural network for the parametrization of the mapping $\mc{O}$, that generates the domain metric $h_{t}$ as `seed' of the sigma flow, according to \eqref{eq:sigma-flow-parametrization}. As proof of concept, we showed empirically that learning in this way the generation of a field of $2\times 2$ metric tensors (i.e.~the evolving discretized metric $h_{t}$) enables to cope with a class of random polygonal scenes, even when contaminated with a  high level of noise. Furthermore, applying this trained model directly to more general labelings of real images deteriorates prediction performance, but does not cause it to break down. 

\subsection{Further Work}
The last statement raises the question: how rich have parametrizations of the time-variant metric $h_{t}$ to be in order to generalize the labeling prediction performance to classes of real images? We point out that this specific question differs from the general problem of the role of the parameters of general black-box deep networks, since 
\begin{enumerate}[(i)]
\item
a major part of the prediction performance is already `hard-wired' by the geometric structure of sigma flows and by geometric integration, which generates the `layers' of a corresponding deep network;
\item
the parametrization only concerns the `generator of the generator' of sigma flows, as discussed above in connection with \eqref{eq:sigma-flow-parametrization}.
\end{enumerate}
These issues regarding the feasibibility of low-dimensional parametrizations should also stimulate probabilistic extensions towards \textit{generative models} for \textit{discrete} structured predictions, along the lines recently worked out using assignment flows \cite{Boll:2024ac}.

Another aspect of our work concerns the appearance of ingredients of the transformer network architecture, as outlined in Section \ref{sec:transformer}, by adopting a `discretize first -- then optimize' approach. Since we adopted the antipodal `optimize first -- then discretize' approach which better conforms to the established mathematical structures of PDE-based data analysis and mathematical physics, we hope that our work contributes to the field of geometric deep learning and to promoting geometric design principles in other areas of scientific machine learning.

\appendix
\section{Appendix -- Glossary of Symbols}

\subsubsection*{Normal Font symbols} 
\begin{tabular}[H]{ll}
$M$ &
smooth closed oriented Riemannian manifold, image domain
\\
$N$ &
smooth manifold \textit{and} number of vertices of a grid graph
\\
$h$ &
Riemannian metric on $M$
\\
$\abs{h}$ &
determinant of $h$
\\
$g$ &
Riemannian metric on $N$
\\
$\Gamma$ &
global sections of vector bundles on $M$
\\
$\delta$ &
Euclidean metric \eqref{eq:color-g-delta-metric} \textit{and} Kronecker delta
\\
$f$ &
smooth function $M \to \R$
\\
$M_{T}$ &
time cylinder $(0,T)\times M$ of $M$ of length $T$ \eqref{eq:def-MT}
\\
$\omega_{h}$ &
volume form of $h$ on $M$ \eqref{eq:int-M-a}
\\
$\dd$ & (de-Rahm) differential on $M$ \eqref{eq:df-space}
\\
$\Omega$ &
weight matrix for S flow \eqref{eq:omega-s} \textit{and} domain in $\R^{2}$ for image functions \eqref{eq:AnisDiff}
\\
$X$ &
open domain in $\R^{2}$ \eqref{eq:cont-domain-s-flow}
\\
$u$ &
time dependent gray value image $u:\Omega_{T} \to \R$ \eqref{eq:def-ut}
\\
$f_{t}$ & 
for $f:M_T \to \R$ and $t \in (0,T)\colon$ $f_{t} = f(t,\cdot)$ is the constant time slice of $f$
\\
$\TT^{2}$ & 
 two-dimensional torus $S^{1} \times S^{1}$
\\
$\tau(f,g,h)$ &
tension field \eqref{eq:tension_field}
\\
$E_{H}$ &
harmonic energy \eqref{eq:def-EH-general}
\\
$\Delta_{h}$ &
Laplace-Beltrami operator \eqref{eq:delta-EH}
\\
$\phi_{n}$ &
eigenfunctions of $\Delta_{h}$ \eqref{eq:Delta-h-spectrum}
\\
$\lambda_{n}$ &
eigenvalues of $\Delta_{h}$ \eqref{eq:Delta-h-spectrum} 
\\
$a_n$ &
expansion coefficients of $\vt$ in the basis $\phi_n$ \eqref{eq:vti-Laplacian-decomposition}
\\
$\kappa$ &
warp factor for anisotropic diffusion \eqref{eq:WAD-system}
\\
$B$ &
scalar function for the generalized diffusion ansatz (Perona-Malik, TV) \eqref{eq:def-EB} 
\\
$\lambda$ &
conformal factor \eqref{eq:conformal-transformation} 
\\
$C(\mc{O})$ &
uniform boundedness of $\mc{O}$
\\
$V$ &
diffusion tensor (in the anisotropic diffusion setting) \eqref{eq:AnisDiff}
\\
$c_{1},c_{2}$ &
bounds for the matrix function $\mc{B}$ \eqref{eq:bounds-mcB}
\\
$m^{2}$ &
mass parameter \eqref{def:entropic-harmony-energy}
\\
$c$ &
number of labels (categories, classes)
\\
$\triangle_{c}$ &
probability simplex in $\R^{c}$ \eqref{eq:def-Delta-p}
\\
$\S_{c}$ &
relative interior of the simplex $\triangle_{c}$ \eqref{eq:def-Sc}
\\
$\chi_{e},\chi_{m}$ &
coordinate charts on $\S_{c}$ \eqref{eq:def-kappa-chart} 
\\
$p_{i}$ &
$m$-coordinates on $\S_{c}$ \eqref{eq:def-theta-p}
\\
$\theta_{i}$ &
$e$-coordinates on $\S_{c}$ \eqref{eq:def-theta-p}
\\
$\cc$ & 
dimension $\cc = c-1$ of $\S_{c}$; index range of local $m$- and $e$-coordinates \eqref{eq:def-kappa-chart}
\\
$T_{0}$ &
tangent space to $\S_{c}$ \eqref{eq:def-T0}
\\
$v$ & tangent vector $v\in T_{0}$ 
\\
$\Pi_{0}$ &
orthogonal tangent space projection \eqref{eq:pi-0}
\\
$R$ &
replicator operator \eqref{eq:replicator-tensor}
\\
$\mathrm{sm}$ &
softmax function \eqref{eq:softmax}
\\
$\alpha$ &
parameter for $\alpha$-connections on $\S_c$, see \eqref{eq:def-alpha-Gamma} 
\\
$\Gamma$ &
Christoffel symbols on $\S_{c}$ \eqref{eq:gamma0} \textit{and} sections of vector bundles 
\\
$\varphi$ &
negative entropy \eqref{eq:HH}
\\
$H$ &
entropy $H = -\varphi$ \eqref{eq:entropy-eps}
\\
$\tilde \varphi$ &
shifted negative entropy \eqref{eq:HH}
\\
$\psi$ & convex (Legendre-Fenchel) conjugate $\psi=\varphi^{\ast}$ of $\varphi$ \eqref{eq:def-psi-partition}
\end{tabular}

\noindent
\begin{tabular}{ll}
$E_{S}$ &
continuous domain S flow energy \eqref{eq:cont-domain-s-flow} 
\\
$E_{\tilde \varphi}$ &
entropic harmonic energy \eqref{def:entropic-harmony-energy}
\\
$\Lambda$ &
sphere map \eqref{eq:sphere-map}
\\
$\mathrm{KL}$ &
Kullback-Leibler divergence for probability vectors \eqref{eq:HH} \textit{and} 
\\ & \hspace{0.5cm}
Kullback-Leibler divergence extended to assignment matrices \eqref{eq:h-learning-problem} 
\\
$\mathrm{TS}$ &
test set used in \eqref{eq:ts} 
\\
$\mathrm{Corr}$ &
probability distribution used for corrupting labelings; cf.~\eqref{eq:Corr-objective}
\end{tabular}

\subsubsection*{Symbols for smooth quantities}
\begin{tabular}{ll}
$\fr$ &
Fisher Rao metric \eqref{eq:fr_metric}
\\
$\veps$ &
parameter regularizing the Fisher-Rao-metric \eqref{eq:def-fr-eps} 
\\
$\mathfrak{d}$ &
partial derivative along $\theta$-coordinates on $\S_{c}$ \eqref{eq:def-pdi}
\\
$\mathfrak{S}$ &
smooth function $M_{T} \to \S_{c}$, smooth analog of assignment matrix $\mathsf{S}$ \eqref{eq:Bf-simplex}, fraktur font capital letter $S$
\\
$\mathfrak{P}$ &
smooth function $\mathfrak{P}\colon M\to\S_{c}$, starting point of sigma flow, smooth analog of $\mathsf{P}$ \eqref{eq:Bf-simplex}, fraktur font capital letter $P$
\\
$\vp$ &
$m$-coordinate expression $\vp = \chi_{m} \circ \mathfrak{S}_{t}$ of $\mathfrak{S}_{t}$ \eqref{eq:def-vt-vp}
\\
$\vt$ & 
$e$-coordinate expression $\vt= \chi_{e} \circ \mathfrak{S}_{t}$ of $\mathfrak{S}_{t}$ \eqref{eq:def-vt-vp}
\\
$\mathfrak{s}$ & 
$\mathfrak{S}_{t}$ mapped $\mathfrak{s}=\Lambda \circ \mathfrak{S}_{r}$ to the sphere \eqref{eq:def-vs}
\\
$\mathfrak{v}$ & 
tangent space representation $\mathfrak{v}= \mathrm{sm}^{-1} \circ \mathfrak{S}_{t}$  of $\mathfrak{S}_{t}$ (Prop.~\ref{prop:sigma-flow-tangent-representation})
\\
$\Psi$ & 
composition $\Psi= \psi \circ \mathfrak{S}_{t}$ of the Hessian potential of \eqref{eq:tension-coordinates}
\\
$\mfk{L}$ &
label configuration $\TT^{2} \to [20]$ used in experiments \eqref{eq:def-vL}
\end{tabular}

\subsubsection*{Symbols for discretized quantities}

\begin{tabular}{ll}
$\mathsf{S}$ &
assignment matrix, spatial discretization of $\mathfrak{S}$ \eqref{eq:def-fS}
\\
$\mathsf{L}_{\Omega}$ &
$\Omega$-induced discrete Laplacian \eqref{eq:Omega-L}
\\
$\mathsf{P}$ &
assignment matrix,  starting point for S flow \eqref{eq:def-s-flow-0}, \eqref{eq:def-s-flow}, spatial discretization of $\mathfrak{P}$
\\
$\mathsf{h}$ &
discretization of $h$ \eqref{eq:def-fh}
\\
$\mathsf{L}_{\mathsf{h}}$ &
discretization of $\Delta_{h}$ \eqref{eq:def-fL}
\\
$\mathsf{D}$ &
discretization of $\partial_{\mu}$ \eqref{eq:def-fD}
\\
$\mathsf{O}$ &
discretization of $\mc{O}$ \eqref{eq:mcO-discrete}
\end{tabular}

\subsubsection*{Calligraphic Symbols}
\begin{tabular}{ll}
$\mc{H}$ &
Gaussian mean curvature 
\\
$\mc{N}$ &
surface normal vector 
\\
$\mc{W}_{c}^{N}$ &
assignment manifold \eqref{eq:def-mcW}
\\
$\mc{G}$ &
graph $\mc{G}=(\mc{V},\mc{E})$
\\
$\mc{V}$ &
graph vertices 
\\
$\mc{E}$ &
graph edges
\\
$\mc{J}$ &
S flow potential \eqref{eq:def-JJ}
\\
$\mc{O}$ &
operator mapping states to metrics \eqref{eq:AnisDiff}, \eqref{eq:WAD-system}, \eqref{eq:WAD-LB}, \eqref{eq:Bf-simplex}, \eqref{eq:mcO-to-be-learned}, \eqref{eq:Corr-objective}
\\ & \hspace{0.5cm}
\eqref{eq:sigmaflow}, \eqref{eq:asgmflw}, \eqref{eq:hsigmaflow}, \eqref{eq:ah-sigmaflow}
\\
$\mc{C}$ &
class of functions parametrized by a neural network \eqref{eq:nn-function-class}
\\
$\mc{B}$ &
matrix function \eqref{eq:def-mcB-theta}
\\
$\mc{F}$ &
metric space
\end{tabular}

\begin{center}
    We deliberately overload some symbols when there is no danger of confusion due to the context.
\end{center}
\label{appendix}

\newpage
\bibliographystyle{amsalpha}
\bibliography{SigmaModels}

\newcommand{\etalchar}[1]{$^{#1}$}
\providecommand{\bysame}{\leavevmode\hbox to3em{\hrulefill}\thinspace}
\providecommand{\MR}{\relax\ifhmode\unskip\space\fi MR }
\providecommand{\MRhref}[2]{%
  \href{http://www.ams.org/mathscinet-getitem?mr=#1}{#2}
}
\providecommand{\href}[2]{#2}
\begin{thebibliography}{vGGOB14}

\bibitem[AGLM93]{Alvarez:1993aa}
L.~Alvarez, F.~Guichard, P.~L. Lions, and J.~M. Morel, \emph{{Axioms and
  Fundamental Equations of Image Processing}}, Arch. Rat. Mech. Anal.
  \textbf{123} (1993), 199--257.

\bibitem[AJLS17]{ay2017information}
N.~Ay, J.~Jost, H.~V. L{\^e}, and L.~Schwachh{\"o}fer, \emph{{Information
  Geometry}}, Springer, 2017.

\bibitem[AN00]{amariMethodsInformationGeometry2000}
S.-I. Amari and H.~Nagaoka, \emph{{Methods of Information Geometry}}, {American
  Mathematical Society}, 2000.

\bibitem[{\AA}PSS17]{astromImageLabelingAssignment2017}
F.~{\AA}str{\"o}m, S.~Petra, B.~Schmitzer, and C.~Schn{\"o}rr, \emph{{Image
  Labeling by Assignment}}, Journal of Mathematical Imaging and Vision
  \textbf{58} (2017), no.~2, 211--238.

\bibitem[Bat11]{batardHeatEquationsVector2011}
T.~Batard, \emph{{Heat Equations on Vector Bundles -- Application to Color
  Image Regularization}}, Journal of Mathematical Imaging and Vision
  \textbf{41} (2011), no.~1, 59--85.

\bibitem[BBL{\etalchar{+}}17]{Bronstein:2017aa}
M.~M. Bronstein, J.~Bruna, Y.~LeCun, A.~Szlam, and P.~Vandergheynst,
  \emph{{Geometric Deep Learning: Going beyond Euclidean data}}, IEEE Signal
  Processing Magazine \textbf{34} (2017), no.~4, 18--42.

\bibitem[BCA{\etalchar{+}}24]{Boll:2024aa}
B.~Boll, J.~Cassel, P.~Albers, S.~Petra, and C.~Schn\"{o}rr, \emph{{A Geometric
  Embedding Approach to Multiple Games and Multiple Populations}}, preprint
  arXiv:2401.05918 (2024).

\bibitem[BCG{\etalchar{+}}18]{Boscain:2018aa}
U.~V. Boscain, R.~Chertovskih, J.-P. Gauthier, D.~Prandi, and A.~Remizov,
  \emph{{Highly Corrupted Image Inpainting Through Hypoelliptic Diffusion}}, J.
  Math. Imag. Vision \textbf{60} (2018), no.~8, 1231--1245.

\bibitem[BCM06]{Buades:2006ue}
A.~Buades, B.~Coll, and J.-M. Morel, \emph{{Neighborhood Filter and PDEs}},
  Numer. Math. \textbf{105} (2006), 1--34.

\bibitem[BCM10]{Buades:2010aa}
A.~Buades, B.~Coll, and J.~M. Morel, \emph{{Image Denoising Methods. A New
  Nonlocal Principle}}, SIAM Review \textbf{52} (2010), no.~1, 113--147.

\bibitem[BF16]{Bertozzi:2016aa}
A.L. Bertozzi and A.~Flenner, \emph{{Diffuse Interface Models on Graphs for
  Classification of High Dimensional Data}}, {SIAM Review} \textbf{58} (2016),
  no.~2, 293--328.

\bibitem[BFH{\etalchar{+}}18]{jax2018github}
James Bradbury, Roy Frostig, Peter Hawkins, Matthew~James Johnson, Chris Leary,
  Dougal Maclaurin, George Necula, Adam Paszke, Jake Vander{P}las, Skye
  Wanderman-{M}ilne, and Qiao Zhang, \emph{{JAX}: composable transformations of
  {P}ython+{N}um{P}y programs}, 2018.

\bibitem[BGAPS25]{Boll:2024ac}
B.~Boll, D.~Gonzalez-Alvarado, S.~Petra, and C.~Schn\"{o}rr, \emph{{Generative
  Assignment Flows for Representing and Learning Joint Distributions of
  Discrete Data}}, J. Math. Imaging Vision \textbf{67} (2025), no.~34.

\bibitem[BKP10]{Bredies:2010aa}
K.~Bredies, K.~Kunisch, and T.~Pock, \emph{{Total Generalized Variation}}, SIAM
  Journal on Imaging Sciences \textbf{3} (2010), no.~3, 492--526.

\bibitem[BS14]{batardClassGeneralizedLaplacians2014}
T.~Batard and N.~Sochen, \emph{{A Class of Generalized Laplacians on Vector
  Bundles Devoted to Multi-Channel Image Processing}}, Journal of Mathematical
  Imaging and Vision \textbf{48} (2014), no.~3, 517--543.

\bibitem[BZPS23]{Boll:2023ab}
B.~Boll, A.~Zeilmann, S.~Petra, and C.~Schn\"{o}rr, \emph{{Self-Certifying
  Classification by Linearized Deep Assignment}}, PAMM: Proc. Appl. Math. Mech.
  \textbf{23} (2023), no.~1, e202200169.

\bibitem[CCN15]{Caselles:2015ws}
V.~Caselles, A.~Chambolle, and M.~Novaga, \emph{{Total Variation in Imaging}},
  {Handbook of Mathematical Methods in Imaging} (O.~Scherzer, ed.), Springer,
  2015, pp.~1455--1499.

\bibitem[CCP12]{Chambolle:2012aa}
A.~Chambolle, D.~Cremers, and T.~Pock, \emph{{A Convex Approach to Minimal
  Partitions}}, SIAM J. Imag. Sci. \textbf{5} (2012), no.~4, 1113--1158.

\bibitem[CFSS16]{Citti:2016aa}
G.~Citti, B.~Franceschiello, G.~Sanguinetti, and A.~Sarti,
  \emph{{Sub-Riemannian Mean Curvature Flow for Image Processing}}, SIAM J.
  Imag. Sci. \textbf{9} (2016), no.~1, 212--237.

\bibitem[Cha84]{chavelEigenvaluesRiemannianGeometry1984}
I.~Chavel, \emph{{Eigenvalues in Riemannian Geometry}}, Pure and Applied
  Mathematics, no. 115, Academic Press, Orlando, 1984.

\bibitem[Cha04]{chambolleAlgorithmTotalVariation2004}
A.~Chambolle, \emph{{An Algorithm for Total Variation Minimization and
  Applications}}, Journal of Mathematical Imaging and Vision \textbf{20}
  (2004), no.~1, 89--97.

\bibitem[CRBD18]{chen2018neural}
R.~T.~Q. Chen, Y.~Rubanova, J.~Bettencourt, and D.~K. Duvenaud, \emph{{Neural
  Ordinary Differential Equations}}, NIPS (2018).

\bibitem[CRE{\etalchar{+}}21]{chamberlain2021beltrami}
B.~Chamberlain, J.~Rowbottom, D.~Eynard, F.~Di~Giovanni, X.~Dong, and
  M.~Bronstein, \emph{{Beltrami Flow and Neural Diffusion on Graphs}}, Advances
  in Neural Information Processing Systems \textbf{34} (2021), 1594--1609.

\bibitem[DHLM05]{desbrun2005discrete}
M.~Desbrun, A.~N. Hirani, M.~Leok, and J.~E. Marsden, \emph{{Discrete Exterior
  Calculus}}, preprint arXiv:math/0508341 (2005).

\bibitem[DMSC16]{Duran:2016aa}
J.~Duran, M.~Moeller, C.~Sbert, and D.~Cremers, \emph{{Collaborative Total
  Variation: A General Framework for Vectorial TV Models}}, SIAM J. Imaging
  Sciences \textbf{9} (2016), no.~1, 116--151.

\bibitem[EKV23]{Elad:2023aa}
M.~Elad, B.~Kawar, and G.~Vaksman, \emph{{Image Denoising: The Deep Learning
  Revolution and Beyond--A Survey Paper}}, SIAM J. Imaging Sciences \textbf{16}
  (2023), no.~3, 1033--1855.

\bibitem[FM19]{Feehan:2019aa}
P.~M.~N. Feehan and M.~Maridakis, \emph{{Lojasiewicz--Simon Gradient
  Inequalities for the Harmonic Map Energy Function}}, preprint
  arXiv:1903.01953v1 (2019).

\bibitem[FMN16]{Fefferman:2016aa}
C.~Fefferman, S.~Mitter, and H.~Narayanan, \emph{{Testing the Manifold
  Hypothesis}}, J. Amer. Math. Soc. \textbf{29} (2016), no.~4, 983--1049.

\bibitem[Gar13]{Garcke:2013aa}
H.~Garcke, \emph{{Curvature Driven Interface Evolution}}, Jahresberichte der
  Deutschen Mathematiker-Vereinigung \textbf{115} (2013), no.~2, 63--100.

\bibitem[GLM20]{gaster2018computing}
J.~Gaster, B.~Loustau, and L.~Monsaingeon, \emph{{Computing Discrete
  Equivariant Harmonic Maps}}, preprint arXiv:1810.11932 (2020).

\bibitem[GLPR23]{geshkovski2023mathematical}
B.~Geshkovski, C.~Letrouit, Y.~Polyanskiy, and P.~Rigollet, \emph{{A
  Mathematical Perspective on Transformers}}, preprint arXiv:2312.10794 (2023).

\bibitem[GY02]{Gu:2002aa}
X.~Gu and S.-T. Yau, \emph{{Computing Conformal Structures of Surfaces}},
  Commun. Information and Systems \textbf{2} (2002), no.~2, 121--146.

\bibitem[Ham17]{hamiltonMathematicalGaugeTheory2017}
M.~J.~D. Hamilton, \emph{{Mathematical Gauge Theory}}, Springer, 2017.

\bibitem[HJ13]{Horn:2013aa}
R.~A. Horn and C.~R. Johnson, \emph{{Matrix Analysis}}, 2nd ed., Cambridge
  University Press, 2013.

\bibitem[HJLZ19]{Han:2019aa}
X.~Han, J.~Jost, L.~Liu, and L.~Zhao, \emph{{Global Existence of the Harmonic
  Map Heat Flow into Lorentzian Manifolds}}, J. Math. Pures Appl. \textbf{130}
  (2019), 130--156.

\bibitem[HKK{\etalchar{+}}03]{Hori:2003aa}
K.~Hori, S.~Katz, A.~Klemm, R.~Pandharipande, R.~Thomas, C.~Vafa, R.~Vakil, and
  E.~Zaslov, \emph{{Mirror Symmetry}}, Amer. Math. Soc., 2003.

\bibitem[HKW77]{Hildebrandt:1977aa}
S.~Hildebrandt, H.~Kaul, and K.-O. Widman, \emph{{An Existence Theorem for
  Harmonic Mappings of Riemannian Manifolds}}, Acta Math. \textbf{138} (1977),
  no.~1, 1--16.

\bibitem[HS98]{hofbauer1998evolutionary}
J.~Hofbauer and K.~Sigmund, \emph{{Evolutionary Games and Population
  Dynamics}}, Cambridge University Press, 1998.

\bibitem[HSLG23]{Han:2023aa}
A.~Han, D.~Shi, L.~Lin, and J.~Gao, \emph{{From Continuous Dynamics to Graph
  Neural Networks: Neural Diffusion and Beyond}}, preprint arXiv:2310.10121
  (2023).

\bibitem[HSPS21]{Huhnerbein:2021th}
R.~H{\"{u}}hnerbein, F.~Savarino, S.~Petra, and C.~Schn\"{o}rr, \emph{{Learning
  Adaptive Regularization for Image Labeling Using Geometric Assignment}}, J.
  Math. Imaging Vision \textbf{63} (2021), 186--215.

\bibitem[Hua06]{Huang:2006ab}
S.-Z. Huang, \emph{{Gradient Inequalities}}, Amer. Math. Soc., 2006.

\bibitem[HW08]{Helein:2008aa}
F.~H{\'e}lein and J.~C. Wood, \emph{{Harmonic Maps}}, Handbook of Global
  Analysis, Elsevier, 2008, pp.~417--491.

\bibitem[Jos97]{Jost:1997ab}
J.~Jost, \emph{{Nonpositive Curvature: Geometric and Analytic Aspects}},
  Springer Basel AG, 1997.

\bibitem[Jos17]{jostRiemannianGeometryGeometric2017}
\bysame, \emph{{Riemannian Geometry and Geometric Analysis}}, Springer, 2017.

\bibitem[JS09]{Jost:2009aa}
J.~Jost and F.~M. Simsir, \emph{{Affine Harmonic Maps}}, Analysis \textbf{29}
  (2009), no.~2, 185--197.

\bibitem[Kas89]{Kass:1989aa}
R.~E. Kass, \emph{{The Geometry of Asymptotic Inference}}, Statistical Science
  \textbf{4} (1989), no.~3, 188--219.

\bibitem[Kic08]{Kichenassamy:2008aa}
S.~Kichenassamy, \emph{{The Perona-Malik Method as an Edge Pruning Algorithm}},
  J. Math. Imaging Vision \textbf{30} (2008), 209--219.

\bibitem[Kid21]{kidger2021on}
P.~Kidger, \emph{{O}n {N}eural {D}ifferential {E}quations}, Ph.D. thesis,
  University of Oxford, 2021.

\bibitem[KMS00]{kimmel2000images}
R.~Kimmel, R.~Malladi, and N.~Sochen, \emph{{Images as Embedded Maps and
  Minimal Surfaces: Movies, Color, Texture, and Volumetric Medical Images}},
  International Journal of Computer Vision \textbf{39} (2000), 111--129.

\bibitem[Lim20]{lim2020hodge}
L.-H. Lim, \emph{{Hodge Laplacians on Graphs}}, {SIAM Review} \textbf{62}
  (2020), no.~3, 685--715.

\bibitem[LRMU15]{Lefkimmiatis:2015aa}
S.~Lefkimmiatis, A.~Roussos, P.~Maragos, and M.~Unser, \emph{{Structure Tensor
  Total Variation}}, SIAM J. Imaging Sciences \textbf{8} (2015), no.~2,
  1090--1122.

\bibitem[LS11]{Lellmann:2011aa}
J.~Lellmann and C.~Schn{\"{o}}rr, \emph{{Continuous Multiclass Labeling
  Approaches and Algorithms}}, SIAM J. Imag. Sci. \textbf{4} (2011), no.~4,
  1049--1096.

\bibitem[MBC15]{Morel:2015aa}
J.-M. Morel, A.~Buades, and T.~Coll, \emph{{Local Smoothing Neighborhood
  Filters}}, {Handbook of Mathematical Methods in Imaging} (O.~Scherzer, ed.),
  Springer, 2015, pp.~1599--1643.

\bibitem[MBO94]{Merriman:1994aa}
B.~Merriman, J.K. Bence, and S.J. Osher, \emph{{Motion of Multiple Junctions: A
  Level Set Approach}}, J. Comp. Physics \textbf{112} (1994), no.~2, 334--363.

\bibitem[MC98]{Manduchi:1998aa}
R.~Manduchi and T.~Carlo, \emph{{Bilateral Filtering for Gray and Color
  Images}}, Sixth International Conference on Computer Vision (ICCV), IEEE,
  1998.

\bibitem[MH80]{Marr:1980aa}
D.~Marr and E.~Hildreth, \emph{{Theory of Edge Detection}}, Proc. Royal Soc.~B
  \textbf{207} (1980), 187--217.

\bibitem[Mil13]{Milanfar:2013ab}
P.~Milanfar, \emph{{A Tour of Modern Image Filtering}}, IEEE Signal Proc. Mag.
  \textbf{30} (2013), no.~1, 106--128.

\bibitem[MOA11]{Marshall:2011vn}
A.~W. Marshall, I.~Olkin, and B.~C. Arnold, \emph{{Inequalities: Theory of
  Majorization and Its Applications}}, 2nd ed., Springer, 2011.

\bibitem[Nis02]{Nishikawa:2002aa}
S.~Nishikawa, \emph{{Variational Problems in Geometry}}, Amer. Math. Soc.,
  2002.

\bibitem[PC17]{Pang:2017aa}
J.~Pang and G.~Cheung, \emph{{Graph Laplacian Regularization for Image
  Denoising: Analysis in the Continuous Domain}}, IEEE Trans. Image Proc.
  \textbf{26} (2017), no.~4, 1770--1785.

\bibitem[PM90]{peronaScalespaceEdgeDetection1990}
P.~Perona and J.~Malik, \emph{{Scale-Space and Edge Detection Using Anisotropic
  Diffusion}}, IEEE Transactions on Pattern Analysis and Machine Intelligence
  \textbf{12} (1990), no.~7, 629--639.

\bibitem[Pol81]{polyakovQUANTUMGEOMETRYBOSONIC1981}
A.~M. Polyakov, \emph{{Quantum Geometry of Bosonic Strings}}, Physics Letters B
  \textbf{103} (1981), no.~3, 207--210.

\bibitem[Pol98]{Polchinski:1998rq}
J.~Polchinski, \emph{{String Theory: An Introduction to the Bosonic String}},
  vol.~1, Cambridge University Press, 1998.

\bibitem[Pro16]{Provenzi:2016aa}
E.~Provenzi, \emph{{A Differential Geometry Model for the Perceived Colors
  Space}}, Int. J. Geom. Meth. Modern Physics \textbf{13} (2016), no.~8,
  1630008.

\bibitem[Res74]{Resnikoff:1974aa}
H.L. Resnikoff, \emph{{Differential Geometry and Color Perception}}, J. Math.
  Biology \textbf{1} (1974), 97--131.

\bibitem[RFB15]{ronneberger2015u}
Olaf Ronneberger, Philipp Fischer, and Thomas Brox, \emph{U-net: Convolutional
  networks for biomedical image segmentation}, International Conference on
  Medical Image Computing and Computer-Assisted Intervention, Springer, 2015,
  pp.~234--241.

\bibitem[ROF92]{rudinNonlinearTotalVariation1992}
L.~Rudin, S.~Osher, and E.~Fatemi, \emph{{Nonlinear Total Variation Based Noise
  Removal Algorithms}}, Physica D \textbf{60} (1992), 259--268.

\bibitem[Ros97]{Rosenberg:1997aa}
S.~Rosenberg, \emph{{The Laplacian on a Riemannian Manifold}}, Cambridge Univ.
  Press, 1997.

\bibitem[SBS23]{Sitenko:2023aa}
D.~Sitenko, B.~Boll, and C.~Schn\"{o}rr, \emph{{A Nonlocal Graph-PDE and
  Higher-Order Geometric Integration for Image Labeling}}, SIAM J. Imaging
  Sciences \textbf{16} (2023), no.~1, 501--567.

\bibitem[SCB{\etalchar{+}}23]{Schwarz:2023ab}
J.~Schwarz, J.~Cassel, B.~Boll, M.~G\"{a}rttner, P.~Albers, and C.~Schn\"{o}rr,
  \emph{{Quantum State Assignment Flows}}, Entropy \textbf{25} (2023), no.~9,
  1253.

\bibitem[Sch15]{Scherzer:2015aa}
O.~Scherzer (ed.), \emph{{Handbook of Mathematical Methods in Imaging}},
  Springer, 2015.

\bibitem[Sch20]{schnorrAssignmentFlows2020}
C.~Schn{\"o}rr, \emph{{Assignment Flows}}, Handbook of {{Variational Methods}}
  for {{Nonlinear Geometric Data}} (P.~Grohs, M.~Holler, and A.~Weinmann,
  eds.), Springer, 2020, pp.~235--260.

\bibitem[Shi07]{shimaGeometryHessianStructures2007}
H.~Shima, \emph{{The Geometry of Hessian Structures}}, World Scientific, 2007.

\bibitem[Sim83]{Simon:1983ab}
L.~Simon, \emph{{Asymptotics for a Class of Non-Linear Evolution Equations,
  with Applications to Geometric Problems}}, Ann. Mathematics \textbf{118}
  (1983), no.~3, 525--571.

\bibitem[SKM98]{sochenGeneralFrameworkLow1998}
N.~Sochen, R.~Kimmel, and R.~Malladi, \emph{{A General Framework for Low Level
  Vision}}, IEEE transactions on image processing: a publication of the IEEE
  Signal Processing Society \textbf{7} (1998), no.~3, 310--318.

\bibitem[Sol85]{solomonHarmonicMapsSpheres1985}
B.~Solomon, \emph{{Harmonic Maps to Spheres}}, Journal of Differential Geometry
  \textbf{21} (1985), no.~2, 151--162.

\bibitem[SPBD23]{smetsPDEbasedGroupEquivariant2023}
B.~Smets, J.~Portegies, E.~Bekkers, and R.~Duits, \emph{{PDE-based Group
  Equivariant Convolutional Neural Networks}}, Journal of Mathematical Imaging
  and Vision \textbf{65} (2023), no.~1, 209--239.

\bibitem[SPSOD21]{Smets:2021tw}
B.~M.~N. Smets, J.~Portegies, E.~St-Onge, and R.~Duits, \emph{{Total Variation
  and Mean Curvature PDEs on the Homogeneous Space of Positions and
  Orientations}}, J. Math. Imaging Vision \textbf{63} (2021), 237--262.

\bibitem[SS21]{Savarino:2021wt}
F.~Savarino and C.~Schn\"{o}rr, \emph{{Continuous-Domain Assignment Flows}},
  Europ. J. Appl. Math. \textbf{32} (2021), no.~3, 570--597.

\bibitem[Uoh14]{Uohashi:2014aa}
K.~Uohashi, \emph{{Harmonic Maps Relative to $\alpha$-Connections}}, {Geometric
  Theory of Information}, Springer, 2014, pp.~81--96.

\bibitem[vGGOB14]{Gennip:2014aa}
Y.~van Gennip, N.~Guillen, B.~Osting, and A.~L. Bertozzi, \emph{{Mean
  Curvature, Threshold Dynamics, and Phase Field Theory on Finite Graphs}},
  Milan Journal of Mathematics \textbf{82} (2014), no.~1, 3--65.

\bibitem[VSP{\etalchar{+}}17]{vaswani2017attention}
A.~Vaswani, N.~Shazeer, N.~Parmar, J.~Uszkoreit, L.~Jones, A.~N. Gomez, and
  L.~Kaiser, \emph{{Attention Is All You Need}}, NIPS, 2017.

\bibitem[War08]{Wardetzky:2008aa}
M.~Wardetzky, \emph{{Convergence of the Cotangent Formula: An Overview}},
  Discrete Differential Geometry, Oberwolfach Seminars, vol.~38,
  Birkh\"{a}user, 2008, pp.~275--286.

\bibitem[Wei98]{Weickert:1998aa}
J.~Weickert, \emph{{Anisotropic Diffusion in Image Processing}}, B.G. Teubner
  Verlag, 1998.

\bibitem[WY23]{wu2020computing}
T.~Wu and S.-T. Yau, \emph{{Computing Harmonic Maps and Conformal Maps on Point
  Clouds}}, J. Comp. Math. \textbf{41} (2023), 879--908.

\bibitem[ZSPS20]{zeilmannGeometricNumericalIntegration2020}
A.~Zeilmann, F.~Savarino, S.~Petra, and C.~Schn{\"o}rr, \emph{{Geometric
  Numerical Integration of the Assignment Flow}}, Inverse Problems \textbf{36}
  (2020), no.~3, 034004.

\bibitem[ZZPS20a]{Zern:2020ab}
A.~Zern, M.~Zisler, S.~Petra, and C.~Schn\"{o}rr, \emph{{Unsupervised
  Assignment Flow: Label Learning on Feature Manifolds by Spatially Regularized
  Geometric Assignment}}, Journal of Mathematical Imaging and Vision
  \textbf{62} (2020), no.~6--7, 982--1006.

\bibitem[ZZPS20b]{Zisler:2020aa}
M.~Zisler, A.~Zern, S.~Petra, and C.~Schn\"{o}rr, \emph{{Self-Assignment Flows
  for Unsupervised Data Labeling on Graphs}}, SIAM Journal on Imaging Sciences
  \textbf{13} (2020), no.~3, 1113--1156.

\bibitem[ZZS22]{Zern:2020aa}
A.~Zern, A.~Zeilmann, and C.~Schn\"{o}rr, \emph{{Assignment Flows for Data
  Labeling on Graphs: Convergence and Stability}}, Information Geometry
  \textbf{5} (2022), 355--404.

\end{thebibliography}

\clearpage
\end{document}